\documentclass[10pt,aps,prb, preprint,reqno]{amsart}
\usepackage{amsmath}
\usepackage{amssymb}
\usepackage{yfonts}
\usepackage{hyperref}
\usepackage{mathrsfs}
\usepackage {latexsym}
\usepackage{graphics}
\usepackage{txfonts}
\usepackage{enumitem}
\allowdisplaybreaks

\DeclareMathSymbol{:}{\mathord}{operators}{"3A}

\newtheorem{theorem}{Theorem}[section]
\newtheorem{remark}{Remark}[section]
\newtheorem{lemma}[theorem]{Lemma}

\newtheorem{proposition}[theorem]{Proposition}

\newtheorem{define}{Definition}[section]

\allowdisplaybreaks 

\begin{document}
\title[Non-uniqueness of 2-d Navier-Stokes equations]{Non-uniqueness in law for two-dimensional Navier-Stokes equations with diffusion weaker than a full Laplacian}
 
\author{Kazuo Yamazaki}  
\address{Texas Tech University, Department of Mathematics and Statistics, Lubbock, TX, 79409-1042, U.S.A.; Phone: 806-834-6112; Fax: 806-742-1112; E-mail: (kyamazak@ttu.edu)}
\date{}
\maketitle

\begin{abstract}
We study the two-dimensional Navier-Stokes equations forced by random noise with a diffusive term generalized via a fractional Laplacian that has a positive exponent strictly less than one. Because intermittent jets are inherently three-dimensional, we instead adapt the theory of intermittent form of the two-dimensional stationary flows to the stochastic approach presented by Hofmanov$\acute{\mathrm{a}}$, Zhu $\&$ Zhu (2019, arXiv:1912.11841 [math.PR]) and prove its non-uniqueness in law. 
\vspace{5mm}

\end{abstract}
\footnote{2010MSC : 35A02; 35Q30; 35R60}

\section{Introduction}

\subsection{Motivation from physics}\label{Motivation from physics} 

The study of turbulence was pioneered by Novikov \cite{N65} more than half a century ago. Motivations to investigate the two-dimensional (2-d) turbulence include applications in meteorology and atmospheric sciences, and its attraction from researchers that led to remarkable progress can be accredited to many reasons: the 2-d flows are easier to simulate than the counterpart in the three-dimensional (3-d) case; vorticity, in addition to kinetic energy, becomes a bounded quantity allowing more flexibility in directions to explore. Indeed, the 2-d turbulence has been extensively studied theoretically (e.g., \cite{K67}), numerically (e.g., \cite{B69}), as well as experimentally (e.g., \cite{PT98}). 

Various forms of dissipation have been introduced in the physics literature: frictional dissipation in \cite{PBH00}; fractional Laplacian $(-\Delta)^{m}$ as a Fourier operator with its Fourier symbol of $\lvert \xi \rvert^{2m}$ so that $\widehat{(-\Delta)^{m} f} (\xi) = \lvert \xi \rvert^{2m} \hat{f}(\xi)$ in the study of surface quasi-geostrophic equations (e.g., \cite[Equation (1)]{C02}). In fact, the study of the Navier-Stokes (NS) equations with diffusive term in the latter form, to which we shall hereafter refer as the generalized NS (GNS) equations \eqref{GNS equations}, can be traced back as far as \cite[Equation (8.7) on pg. 263]{L59} in 1959 by Lions. The purpose of this manuscript is to prove a certain non-uniqueness for the 2-d GNS equations forced by random noise which we introduce next. 

\subsection{Previous results} 
Throughout this manuscript we define $\mathbb{T}^{n} \triangleq [-\pi, \pi]^{n}$ to be the principal spatial domain for $x = (x^{1}, \hdots, x^{n})$, denote $\partial_{t} \triangleq \frac{\partial}{\partial t}$, $\nabla \triangleq (\partial_{x^{1}}, \hdots, \partial_{x^{n}})$, as well as $u \triangleq (u^{1}, \hdots, u^{n})$, and $\pi$ that map from $\mathbb{R}_{+} \times \mathbb{T}^{n}$ as the $n$-dimensional ($n$-d) velocity vector and pressure scalar fields, respectively. We let $\nu \geq 0$ represent the viscosity coefficient so that the GNS equations read 
\begin{equation}\label{GNS equations} 
\partial_{t} u + \nu (-\Delta)^{m} u + \text{div} (u\otimes u) + \nabla \pi = 0, \hspace{3mm} \nabla\cdot u =0, \hspace{3mm} t > 0, 
\end{equation} 
which recovers the classical NS equations when $m = 1$ and $\nu > 0$, as well as the Euler equations when $ \nu= 0$. We call $u \in C_{t}L_{x}^{2}$ a weak solution to \eqref{GNS equations} over $[0,T]$ if $u(t, \cdot)$ is weakly divergence-free, is mean-zero; i.e., $\int_{\mathbb{T}^{n}} u(t, x) dx= 0$, and satisfies \eqref{GNS equations} distributionally against a smooth and divergence-free function. A Leray-Hopf weak solution, only in case $\nu > 0$, due to \cite{H51, L34} requires an additional regularity of $L_{t}^{2} \dot{H}_{x}^{m}$ and must satisfy an energy inequality 
\begin{equation}\label{energy inequality}
\frac{1}{2} \lVert u(t) \rVert_{L_{x}^{2}}^{2}+  \nu \int_{0}^{t} \lVert u(s) \rVert_{\dot{H}_{x}^{m}}^{2} ds \leq \frac{1}{2} \lVert u(0) \rVert_{L_{x}^{2}}^{2} 
\end{equation}
for all $t \geq 0$ (see \cite[Definitions 3.1, 3.5, and 3.6]{BV19b} for precise statements). Due to the rescaling property of the GNS equations that $(u_{\lambda}, \pi_{\lambda}) (t,x) \triangleq (\lambda^{2m-1} u, \lambda^{4m-2} \pi) (\lambda^{2m} t, \lambda x)$ solves \eqref{GNS equations} if $(u,\pi)(t,x)$ solves it, a standard classification states that \eqref{GNS equations} is sub-critical, critical and super-critical with respect to $L^{2}(\mathbb{T}^{n})$-norm if $m > \frac{1}{2} + \frac{n}{4}, m = \frac{1}{2} + \frac{n}{4}$, and $m < \frac{1}{2} + \frac{n}{4}$, respectively. 

Only a decade after \cite{L59}, Lions (see \cite[Remark 6.11 on pg. 96]{L69}) already claimed the uniqueness of a Leray-Hopf weak solution when $\nu > 0$ and $m \geq  \frac{1}{2} + \frac{n}{4}$. It has been more than 50 years since then, and we still find this threshold to be sharp; specifically, except a logarithmic improvement by Tao \cite{T09} (and also \cite{BMR14} for further logarithmic improvements), it is not known whether  \eqref{GNS equations} with $\nu > 0$ and $m < \frac{1}{2} + \frac{n}{4}$ for $n \geq 3$ admits a unique solution that emanates from a smooth initial data and preserves the initial regularity or not (e.g. see \cite[Theorem 4.1]{W04} for such a result under a smallness constraint on initial data). The case $n = 2$ offers a strikingly different picture when initial data has high regularity. Indeed, Yudovich \cite{Y63} proved that if the vorticity $\nabla \times u$ belongs initially to $L^{1}(\mathbb{R}^{2})\cap L^{\infty}(\mathbb{R}^{2})$, then even the 2-d Euler equations admit a globally unique solution, essentially due to the fact that the nonlinear term vanishes upon an $L^{p}(\mathbb{R}^{2})$-estimate of the vorticity for any $p \in [2, \infty]$ (e.g., \cite[pg. 320]{MB02}). That being said, starting from an arbitrary initial data in $L_{x}^{2}$, the lack of diffusion and therefore a lack of high regularity creates an obstacle in constructing a weak solution via a classical argument relying on Aubin-Lions compactness lemma (e.g. \cite{L69, S86}). 

We now discuss the recent developments on Onsager's conjecture which led to a better understanding of equations of fluid and various new techniques. In 1949 a chemist and a physicist Onsager \cite{O49} conjectured the following dichotomy in any spatial dimension $n \geq 2$: 
\begin{itemize}
\item every weak solution to the $n$-d Euler equations with H$\ddot{\mathrm{o}}$lder regularity in space of exponent $\alpha > \frac{1}{3}$, i.e., $C_{x}^{\alpha}$, conserves kinetic energy $\frac{1}{2} \lVert u(t) \rVert_{L_{x}^{2}}^{2}$; 
\item for any $\alpha \leq \frac{1}{3}$ there exists a weak solution in $C_{x}^{\alpha}$ that dissipates kinetic energy $\frac{1}{2} \lVert u(t) \rVert_{L_{x}^{2}}^{2}$. 
\end{itemize}
The case $\alpha > \frac{1}{3}$ proved to be easier to demonstrate, settled partially by Eyink \cite{E94} and then fully by Constantin, E, and Titi \cite{CET94}.  Towards Onsager's conjecture in case $\alpha \leq \frac{1}{3}$, Scheffer \cite{S93} and subsequently Shnirelman \cite{S97} proved the existence of a weak solution to 2-d Euler equations with compact support in space and time so that kinetic energy is both created and destroyed; however, the solutions in \cite{S93, S97} were only in $L_{T}^{2}L_{x}^{2}$ and thus far from the threshold of $C_{x}^{\alpha}, \alpha \leq \frac{1}{3}$. The remarkable series of breakthroughs which unfolded next were inspired by the work of Nash \cite{N54} who proved the $C^{1}$ isometric embedding by constructing  a sequence of short isometric embeddings, each of which fails to be isometric by a certain error that vanishes in the limit. Gromov considered the work of Nash, as well as that of Kuiper \cite{K55}, as part of $h$-principle (\cite[pg. 3]{G86}) and initiated the theory of convex integration \cite[Part 2.4]{G86}; we refer to \cite{DS12} for further discussions on the $h$-principle. After M$\ddot{\mathrm{u}}$ller and $\check{\mathrm{S}}$ver$\acute{\mathrm{a}}$k \cite{MS03} extended the convex integration to Lipschitz maps, De Lellis and Sz$\acute{\mathrm{e}}$kelyhidi Jr. \cite{DS09} reformulated the Euler equations as a differential inclusion and improved the results of \cite{S93, S97} by constructing a weak solution in $L_{T}^{\infty} L_{x}^{\infty}$ with compact support in space and time in any spatial dimension $n \geq 2$ (see also \cite{DS10}). Subsequently, in \cite{DS13} they proved the existence of weak solutions to 3-d Euler equations in $C([0,T] \times \mathbb{T}^{3})$ which dissipate the kinetic energy through a novel application of Beltrami flows. Together with Buckmaster and Isett in \cite{BDIS15}, they improved the regularity of the solution up to $C_{t,x}^{\alpha}$ for any $\alpha < \frac{1}{5}$, where we write $f \in C_{t,x}^{\alpha}$ if there exists $C \geq 0$ such that 
\begin{equation*}
\lvert f(t+ \Delta t, x + \Delta x)  - f(t,x) \rvert \leq C ( \lvert \Delta t \rvert + \lvert \Delta x \rvert)^{\alpha} \hspace{2mm} \text{ uniformly in } t, x, \Delta t, \text{ and } \Delta x 
\end{equation*}
(see also \eqref{C-t,x}). At last, via a certain gluing approximation technique and Mikado flows, Isett \cite{I18} proved that for any $\alpha < \frac{1}{3}$ there exists a non-zero weak solution to $n$-d Euler equations for $n \geq 3$ in $C_{t,x}^{\alpha}$ that fails to conserve kinetic energy (\cite[Theorem 1]{I18} only states the claim for $n  =3$, but \cite[pg. 877]{I18} claims that it can be extended to any $n \geq 3$). Integrating ideas of intermittency from turbulence to Beltrami flows and constructing intermittent Beltrami waves, Buckmaster and Vicol \cite{BV19a} proved the non-uniqueness of weak solutions to the 3-d NS equations in the class $C_{T}H_{x}^{\beta}$ for some $\beta > 0$, which can be seen to be quite small from its proof. Relying on the intermittent Beltrami waves, Luo and Titi \cite{LT20} extended the result of \cite{BV19a} up to Lions' exponent $m < \frac{5}{4}$ for \eqref{GNS equations} when $n = 3$. Mimicking the benefits of Mikado flows, Buckmaster, Colombo, and Vicol \cite{BCV18} introduced intermittent jets to prove non-uniqueness for a class of weak solutions to 3-d GNS equations with $m < \frac{5}{4}$ which have bounded kinetic energy, integrable vorticity, and are smooth outside a fractal set of singular times with Hausdorff dimension strictly less than one. 

As already mentioned in Subsection \ref{Motivation from physics}, the study of NS equations forced by random noise, to which hereafter we refer as the stochastic NS (SNS) equations, has a long history since \cite{N65} (see also \cite{BT73}). Our focus will be on the following stochastic GNS (SGNS) equations: for $x \in \mathbb{T}^{n}$, 
\begin{equation}\label{[Equation (2), Y20b]}
du + (-\Delta)^{m} u dt + \text{div} (u\otimes u) dt + \nabla \pi dt = F(u) dB, \hspace{3mm} \nabla\cdot u = 0, \hspace{3mm} t > 0, 
\end{equation} 
where $F(u)dB$ represents the random noise, to be specified subsequently. Via a probabilistic Galerkin approximation and variations of Aubin-Lions compactness results aforementioned, Flandoli and Gatarek \cite{FG95} proved the existence of a weak solution to the $n$-d SNS equations for $n \geq 2$ under some assumptions on the noise; their solution has the regularity of $L_{t}^{\infty} L_{x}^{2} \cap L_{t}^{2} \dot{H}_{x}^{1}$ but does not necessarily satisfy the energy inequality (see \cite[Definition 3.1]{FG95} and also \cite[Definition 4.3]{F08}). Via the approach of martingale problem due to Stroock and Varadhan \cite{SV97}, Flandoli and Romito constructed a Leray-Hopf weak solution to the 3-d SNS equations; i.e., the solutions constructed therein have the regularity of $L_{t}^{\infty} L_{x}^{2} \cap L_{t}^{2} \dot{H}_{x}^{1}$ and satisfy a stochastic analogue of the energy inequality (see \cite[MP3 in Definition 3.3]{FR08}). Very recently, Hofmanov$\acute{\mathrm{a}}$, Zhu, and Zhu \cite{HZZ19} adapted the convex integration approach through intermittent jets from \cite[Chapter 7]{BV19b} to the 3-d SNS equations and proved the non-uniqueness in law within a class of weak solutions, which also implies the lack of path-wise uniqueness by Yamada-Watanabe theorem (see also \cite{BFH20, CFF19, HZZ22} for probabilistic convex integration on stochastic Euler equations); we must emphasize that their result does not extend to the Leray-Hopf weak solution from \cite{FR08}.

\begin{remark}
It is worth pointing out that the proof of non-uniqueness in the stochastic case has a layer of complexity that is absent in the deterministic case in the following manner. For example, Buckmaster and Vicol in \cite[Theorem 1.2]{BV19a} specifically proved that there exists $\beta > 0$ such that for any non-negative smooth function $e(t): \hspace{0.5mm} [0,T] \mapsto \mathbb{R}_{+}\cup \{0\}$, there exists $u \in C_{T}H_{x}^{\beta}$ that is a weak solution to the NS equations and satisfies $\lVert u(t) \rVert_{L_{x}^{2}}^{2} = e(t)$ for all $t \in [0,T]$. One may take e.g. $e(t) = e^{t} - 1$ so that $e(0) = 0$. Because $u \equiv 0$ for all $(t,x) \in [0,T]\times \mathbb{T}^{3}$ solves the NS equations and satisfies $\lVert u(0)\rVert_{L_{x}^{2}}^{2} = 0$, this immediately deduces non-uniqueness. This approach clearly fails in the stochastic case because $u \equiv 0$ for all $(t,x) \in [0,T]\times \mathbb{T}^{3}$ does not solve the stochastic NS equations due to the presence of the noise (see \cite[Remark 4.16]{F08} for a similar discussion). More precisely, particularly in the case of an additive noise, one may split \eqref{[Equation (2), Y20b]} to a linear stochastic PDE solved by $z$ and the rest of the terms solved by $v$ as in \eqref{[Equation (24), Y20b]}-\eqref{[Equation (25), Y20b]} in hope to adapt the proof of \cite[Theorem 1.2]{BV19a} to the equation of $v$; unfortunately, $v\equiv 0$ does not solve \eqref{[Equation (25), Y20b]} as aforementioned. Another major difficulty that arises in the stochastic case will be discussed in Remark \ref{Remark 1.2}.
\end{remark}

Similarly to our discussion in Subsection \ref{Motivation from physics}, the 2-d SNS equations have received a considerable amount of attention from researchers who have produced a wealth of results many of which remain open in the 3-d case. Path-wise uniqueness, and consequently uniqueness in law due to Yamada-Watanabe theorem, of the aforementioned weak solution with regularity $L_{t}^{\infty}L_{x}^{2} \cap L_{t}^{2}\dot{H}_{x}^{1}$ that does not necessarily satisfy the energy inequality which was constructed in \cite{FG95} are well-known. In the case of an additive noise, upon considering the difference of two possible solutions, the noise cancels out and thus a deterministic approach immediately implies uniqueness (see  \cite[Exercise 3.1 on p. 72]{F08}); in the case of a multiplicative noise we refer to \cite[Theorem 2.4]{CM10}. Same uniqueness results for the Leray-Hopf weak solutions to the 2-d SNS equations directly follow. We also refer to \cite[Theorem 3.2]{CM10} and \cite{HM06} concerning large deviation principle and ergodicity with hypo-elliptic noise, respectively. The purpose of this manuscript is to prove the non-uniqueness in law, and therefore a lack of path-wise uniqueness, for \eqref{[Equation (2), Y20b]} when $n = 2$ and $m \in (0,1)$, which has been studied by many authors previously (e.g., \cite{CGV14}). 

\begin{remark}\label{Remark 1.2}
As we remarked already, the theory of global well-posedness for \eqref{GNS equations} in the 2-d case is significantly richer than that in the 3-d case. Vice versa, proving non-uniqueness in the 2-d case should present considerable difficulty, not seen in the 3-d case.  A natural approach to prove the non-uniqueness in law for \eqref{[Equation (2), Y20b]} with $n = 2$ and $m \in (0,1)$ will be to try to follow the arguments in \cite{HZZ19} on the 3-d SNS equations. Concerning the fractional Laplacian, we can follow the arguments in \cite{Y20a} in which the analogous result was proven for \eqref{[Equation (2), Y20b]} when $n  =3$ and $m \in (\frac{13}{20}, \frac{5}{4})$. 

First major obstacle arises in the fact that intermittent jets, utilized in \cite{HZZ19, Y20a} following \cite[Chapter 7]{BV19b}, are inherently 3-d in space and thus inapplicable to \eqref{[Equation (2), Y20b]} when $n = 2$; we recall that the lack of a suitable replacement for Mikado flows in the 2-d case is precisely the reason the case $n =2$ was left out in the resolution of Onsager's conjecture by Isett (see \cite[pg. 877]{I18}). Fortunately, a 2-d analogue of the 3-d Beltrami flows from \cite{DS13} was already established by Choffrut, De Lellis, and Sz$\acute{\mathrm{e}}$kelyhidi Jr. \cite{CDS12}, to which we refer as 2-d stationary flows. Moreover, its intermittent form, to which we refer as 2-d intermittent stationary flows, was very recently introduced by Luo and Qu \cite{LQ20}. Thus, a good candidate for strategy now is to somehow adapt the application of 2-d intermittent stationary flows in \cite{LQ20} to the stochastic setting in \cite{HZZ19}. 

Second major obstacle that arises in this endeavor is that the arguments in \cite{LQ20} follow closely those of \cite{BV19a} and not \cite[Chapter 7]{BV19b}, quite naturally because the 2-d intermittent stationary flows is an extension of the intermittent Beltrami waves in \cite{BV19a}, not intermittent jets in \cite[Chapter 7]{BV19b}. It turns out that some of the crucial estimates achieved in \cite{BV19a, LQ20} seem to be difficult in the stochastic setting. E.g., while \cite[Equation (2.4)]{BV19a} and \cite[Equation (2.13)]{LQ20} claim certain bounds on the $C_{t,x}^{1}$-norm of Reynolds stress, our Reynolds stress in \eqref{[Equation (92), Y20b]} includes $R_{\text{com2}}$ defined in \eqref{[Equation (91e), Y20b]} that consists of $z$, and $z$ is known to be only in $C_{t}^{\frac{1}{2} - 2 \delta}$ for $\delta > 0$ from \eqref{[Equation (38), Y20b]}. Therefore, obtaining an analogous estimate to \cite[Equation (2.4)]{BV19a} and \cite[Equation (2.13)]{LQ20} seems to be completely out of reach. Thus, our task is not only to apply the theory of 2-d intermittent stationary flows from \cite{LQ20} but consider an extension of the arguments in \cite{LQ20} to that of \cite[Chapter 7]{BV19b} and then adjust that in the stochastic setting of \cite{HZZ19}, while simultaneously considering the approach of \cite{Y20a} to treat the fractional Laplacian. We will carefully define various parameters, all of which depend on the value of $m$ (e.g., \eqref{[Equation (2.2), LQ20]}-\eqref{[Equation (56), Y20b]}, and \eqref{p ast}). Our proofs are inspired by those of \cite{BV19a, BV19b, HZZ19, LQ20} while on various occasions we need to make crucial modifications (e.g., Remarks \ref{Remark 4.1}-\ref{Remark 4.4}).
\end{remark}  

\section{Statement of main results}\label{Statement of main results}
Only for simplicity of presentations, we assume $\nu = 1$ hereafter. Following \cite{HZZ19} we consider two types of random noises within \eqref{[Equation (2), Y20b]}: additive; linear multiplicative. 

\subsection{The case of an additive noise}
In the case of an additive noise, we consider \eqref{[Equation (2), Y20b]} with $n = 2$, $F \equiv 1$, and $B$ to be a $GG^{\ast}$-Wiener process on a probability space $(\Omega, \mathcal{F}, \textbf{P})$ where $G$ is a certain Hilbert-Schmidt operator to be described in more detail subsequently (see \eqref{[Equations (11) and (12), Y20b]}), and the asterisk denotes the adjoint operator. Finally, $(\mathcal{F}_{t})_{t\geq 0}$ denotes the filtration generated by $B$.  

\begin{theorem}\label{[Theorem 2.1, Y20b]}
Suppose that $n = 2, F \equiv 1, m \in (0,1)$, $B$ is a $GG^{\ast}$-Wiener process, and Tr$((-\Delta)^{2-m + 2 \sigma} G G^{\ast}) < \infty$ for some $\sigma > 0$. Then given $T> 0, K > 1$, and $\kappa \in (0,1)$, there exists $\varepsilon \in (0,1)$ and a $\textbf{P}$-almost surely (a.s.) strictly positive stopping time $\mathfrak{t}$ such that $\textbf{P} (\{ \mathfrak{t} \geq T \}) > \kappa$ and the following is additionally satisfied. There exists an $(\mathcal{F}_{t})_{t \geq 0}$-adapted process $u$ that is a weak solution to \eqref{[Equation (2), Y20b]} starting from a deterministic initial condition $u^{\text{in}}$, satisfies 
\begin{equation}\label{[Equation (4), Y20b]}
\text{esssup}_{\omega \in \Omega} \sup_{s \in [0, \mathfrak{t}]} \lVert u(s, \omega) \rVert_{H_{x}^{\varepsilon}} < \infty, 
\end{equation} 
and on the set $\{\mathfrak{t} \geq T \}$, 
\begin{equation}\label{[Equation (5), Y20b]}
\lVert u(T) \rVert_{L_{x}^{2}} > K \lVert u^{\text{in}} \rVert_{L_{x}^{2}} + K ( T\text{Tr} ( GG^{\ast} ))^{\frac{1}{2}}. 
\end{equation} 
\end{theorem}
\begin{remark}\label{Remark 2.1}
For the 3-d SGNS equations \eqref{[Equation (2), Y20b]} with $m \in (\frac{13}{20}, \frac{5}{4})$, \cite{Y20a} required a hypothesis of $\text{Tr} ((-\Delta)^{\frac{5}{2} -m + 2 \sigma} GG^{\ast}) < \infty$ (see \cite[Remark 2.1]{Y20a}). Here in the 2-d case, we need $\text{Tr} ((-\Delta)^{2- m + 2 \sigma} GG^{\ast} ) < \infty$ for the purpose of Proposition \ref{[Proposition 4.4, Y20b]}. 
\end{remark}
\begin{theorem}\label{[Theorem 2.2, Y20b]}
Suppose that $n = 2, F \equiv 1, m \in (0,1)$, $B$ is a $GG^{\ast}$-Wiener process, and Tr$((-\Delta)^{2- m + 2 \sigma} GG^{\ast} ) < \infty$ for some $\sigma > 0$. Then non-uniqueness in law holds for \eqref{[Equation (2), Y20b]} on $[0,\infty)$. Moreover, for all $T > 0$ fixed, non-uniqueness in law holds for \eqref{[Equation (2), Y20b]} on $[0,T]$. 
\end{theorem}

\subsection{The case of a linear multiplicative noise}
In the case of a linear multiplicative noise, we will consider $F(u) = u$ and $B$ to be an $\mathbb{R}$-valued Wiener process on $(\Omega, \mathcal{F}, \textbf{P})$. 

\begin{theorem}\label{[Theorem 2.3, Y20b]}
Suppose that $n  = 2, F(u) = u$, $m \in (0,1)$, and $B$ is an $\mathbb{R}$-valued Wiener process on $(\Omega, \mathcal{F}, \textbf{P})$. Then given $T > 0, K > 1$, and $\kappa \in (0,1)$, there exists $\varepsilon \in (0,1)$ and a $\textbf{P}$-a.s. strictly positive stopping time $\mathfrak{t}$ such that $\textbf{P}  ( \{ \mathfrak{t} \geq T \}) > \kappa$ and the following is additionally satisfied. There exists an $(\mathcal{F}_{t})_{t \geq 0}$-adapted process $u$ that is a weak solution to \eqref{[Equation (2), Y20b]} starting from a deterministic initial condition $u^{\text{in}}$, satisfies 
\begin{equation}\label{[Equation (7), Y20b]}
\text{esssup}_{\omega \in \Omega} \sup_{s \in [0,\mathfrak{t}]} \lVert u(s, \omega) \rVert_{H_{x}^{\varepsilon}} < \infty, 
\end{equation} 
and on the set $\{\mathfrak{t} \geq T \}$, 
\begin{equation}\label{[Equation (8), Y20b]}
\lVert u(T) \rVert_{L_{x}^{2}} > K e^{\frac{T}{2}} \lVert u^{\text{in}} \rVert_{L_{x}^{2}}. 
\end{equation} 
\end{theorem}

\begin{theorem}\label{[Theorem 2.4, Y20b]}
Suppose that $n  = 2, F(u) = u$, $m \in (0,1)$, and $B$ is an $\mathbb{R}$-valued Wiener process on $(\Omega, \mathcal{F}, \textbf{P})$. Then non-uniqueness in law holds for \eqref{[Equation (2), Y20b]} on $[0,\infty)$. Moreover, for any $T > 0$ fixed, non-uniqueness in law holds for \eqref{[Equation (2), Y20b]} on $[0, T]$. 
\end{theorem}
\begin{remark}
After this work was completed, Cheskidov and Luo \cite{CL21} proved non-uniqueness for the 2-d deterministic Navier-Stokes equations in the class of $C_{t}L_{x}^{p}$ for $p \in [1,2)$. We point out that on one hand, they proved non-uniqueness with a full Laplacian while Theorems \ref{[Theorem 2.1, Y20b]}-\ref{[Theorem 2.4, Y20b]} are concerned with the GNS equations diffused via $(-\Delta)^{m}, m \in (0,1)$. On the other hand, the spatial regularity of the solutions constructed in \cite{CL21} are in $L_{x}^{p}$ for $p \in [1,2)$ while those in Theorems \ref{[Theorem 2.1, Y20b]}-\ref{[Theorem 2.4, Y20b]} are in $H_{x}^{\epsilon}$, although for $\epsilon \in (0,1)$ very small, as can be seen from their proofs. 
\end{remark} 

The rest of this manuscript is organized as follows: Section \ref{Preliminaries} with a minimum amount of notations, assumptions, and past results; Section \ref{Proof in the case of additive noise} with proofs of Theorems \ref{[Theorem 2.1, Y20b]} and \ref{[Theorem 2.2, Y20b]}; Section \ref{Proof in the case of linear multiplicative noise} with proofs of Theorems \ref{[Theorem 2.3, Y20b]} and \ref{[Theorem 2.4, Y20b]}; Appendix with additional past results and details of some proofs. 

\section{Preliminaries}\label{Preliminaries}
We denote $\mathbb{N} \triangleq \{1, 2, \hdots, \}$ and $\mathbb{N}_{0} \triangleq \{0\} \cup \mathbb{N}$. We write $A \lesssim_{a,b}B$ and $A \approx_{a,b} B$ to imply the existence of a constant $C= C(a,b) \geq 0$ such that $A\leq CB$ and $A = CB$, respectively. We write $A \overset{(\cdot)}{\lesssim} B$ to indicate that this inequality is due to an equation $(\cdot)$. For any $\mathbb{R}^{2}$-valued maps $f$ and $g$, we denote a tensor product by $f \otimes g $ while its trace-free part by 
\begin{equation}\label{estimate 1}
f \mathring{\otimes} g \triangleq 
\begin{pmatrix}
f^{1} g^{1} - \frac{1}{2} f\cdot g & f^{1} g^{2} \\
f^{2} g^{1} & f^{2} g^{2} - \frac{1}{2} f \cdot g 
\end{pmatrix}.
\end{equation} 
We write for $p\in[1,\infty]$ 
\begin{equation}\label{C-t,x}
\lVert f \rVert_{L^{p}} \triangleq \lVert f \rVert_{L_{t}^{\infty} L_{x}^{p}}, \hspace{1mm} \lVert f \rVert_{C^{N}} \triangleq \lVert f \rVert_{L_{t}^{\infty} C_{x}^{N}} \triangleq \sum_{0\leq \lvert \alpha \rvert \leq N} \lVert D^{\alpha} f \rVert_{L^{\infty}}, \hspace{1mm}\lVert f \rVert_{C_{t,x}^{N}} \triangleq \sum_{0\leq n + \lvert \alpha \rvert \leq N} \lVert \partial_{t}^{n} D^{\alpha} f \rVert_{L^{\infty}}. 
\end{equation} 
We also define $L_{\sigma}^{2} \triangleq \{f \in L_{x}^{2}: \hspace{0.5mm} \nabla\cdot f = 0 \}$, reserve $\mathbb{P} \triangleq \text{Id} - \nabla \Delta^{-1} \nabla\cdot $ as the Leray projection operator, and $\mathbb{P}_{\leq r}$  to be a Fourier operator with a Fourier symbol of $1_{\lvert \xi \rvert \leq r} (\xi)$. For any Polish space $H$, we write $\mathcal{B}(H)$ to represent the $\sigma$-algebra of Borel sets in $H$. We denote a mathematical expectation with respect to (w.r.t.) any probability measure $P$ by $\mathbb{E}^{P}$. We represent an $L^{2}(\mathbb{T}^{2})$-inner product, a cross variation of $A$ and $B$, and a quadratic variation of $A$ respectively by $\langle A, B \rangle, \langle \langle A, B \rangle \rangle,$ and $\langle \langle A \rangle \rangle \triangleq \langle \langle A, A \rangle \rangle$.  We define $\mathcal{P} (\Omega_{0})$ as the set of all probability measure on $(\Omega_{0}, \mathcal{B})$ where $\Omega_{0} \triangleq C([0,\infty); \hspace{1mm} H^{-3} (\mathbb{T}^{2})) \cap L_{\text{loc}}^{\infty} ([0,\infty); L_{\sigma}^{2})$ and $\mathcal{B}$ is the Borel $\sigma$-field of $\Omega_{0}$ from the topology of locally uniform convergence on $\Omega_{0}$. We define the canonical process $\xi: \hspace{1mm} \Omega_{0} \mapsto H^{-3} (\mathbb{T}^{2})$ by $\xi_{t} (\omega) \triangleq \omega(t).$ Similarly, for $t \geq 0$ we define $\Omega_{t} \triangleq C([t, \infty); H^{-3} (\mathbb{T}^{2})) \cap L_{\text{loc}}^{\infty} ([t, \infty); L_{\sigma}^{2})$ and the following Borel $\sigma$-algebras for $t \geq 0$: $\mathcal{B}^{t} \triangleq \sigma ( \{ \xi(s) : \hspace{0.5mm}s \geq t \})$; $\mathcal{B}_{t}^{0} \triangleq \sigma ( \{ \xi(s): \hspace{0.5mm} s \leq t \})$; $\mathcal{B}_{t} \triangleq \cap_{s > t} \mathcal{B}_{s}^{0}$. For any Hilbert space $U$ we denote by $L_{2}(U, L_{\sigma}^{2})$ the space of all Hilbert-Schmidt operators from $U$ to $L_{\sigma}^{2}$ with the norm $\lVert \cdot \rVert_{L_{2}(U, L_{\sigma}^{2})}$. We require $F: \hspace{0.5mm} L_{\sigma}^{2} \mapsto L_{2} (U, L_{\sigma}^{2})$ to be $\mathcal{B}(L_{\sigma}^{2}) / \mathcal{B} ( L_{2} (U, L_{\sigma}^{2}))$-measurable and that it satisfies for any $\phi \in C^{\infty} (\mathbb{T}^{2}) \cap L_{\sigma}^{2}$ 
\begin{equation}\label{[Equations (11) and (12), Y20b]}
\lVert F(\phi) \rVert_{L_{2} (U, L_{\sigma}^{2})} \leq C(1+ \lVert \phi \rVert_{L_{x}^{2}}) \hspace{2mm} \text{ and } \hspace{2mm}  \lim_{n\to\infty} \lVert F(\theta_{n})^{\ast} \phi - F(\theta)^{\ast} \phi \rVert_{U} = 0 
\end{equation} 
for some constant $C \geq 0$ if $\lim_{n\to\infty} \lVert \theta_{n} - \theta \rVert_{L_{x}^{2}} = 0$. 

The following notations will be useful in the case of a linear multiplicative noise. We assume the existence of another Hilbert space $U_{1}$ such that the embedding $U \hookrightarrow U_{1}$ is Hilbert-Schmidt. We define  $\bar{\Omega} \triangleq C([0,\infty); H^{-3} (\mathbb{T}^{2}) \times U_{1}) \cap L_{\text{loc}}^{\infty} ([0,\infty); L_{\sigma}^{2} \times U_{1})$ and $\mathcal{P} (\bar{\Omega})$ as the set of all probability measures on $(\bar{\Omega}, \bar{\mathcal{B}})$, where $\bar{\mathcal{B}}$ is the Borel $\sigma$-algebra on $\bar{\Omega}$. Analogously we define the canonical process on $\bar{\Omega}$ as $(\xi, \theta): \hspace{0.5mm} \bar{\Omega} \mapsto H^{-3} (\mathbb{T}^{2}) \times U_{1}$ by $(\xi_{t} (\omega), \theta_{t}(\omega)) \triangleq \omega(t)$. We extend the previous definitions of $\mathcal{B}^{t}, \mathcal{B}_{t}^{0}$ and $\mathcal{B}_{t}$ to $\bar{\mathcal{B}}^{t} \triangleq \sigma(\{ (\xi, \theta)(s): \hspace{0.5mm} s \geq t \})$, $\bar{\mathcal{B}}_{t}^{0} \triangleq \sigma ( \{ (\xi, \theta)(s): \hspace{0.5mm} s \leq t \})$, and $\bar{\mathcal{B}}_{t} \triangleq \cap_{s > t} \bar{\mathcal{B}}_{s}^{0}$ for $t \geq 0$, respectively. 

Next, we describe some notations and results concerning the 2-d intermittent stationary flows introduced in \cite{CDS12} (e.g., \cite[Lemma 4]{CDS12}) and extended in \cite{LQ20}. We let 
\begin{equation}\label{[Equation (3.12b), LQ20]}
\Lambda^{+} \triangleq \{ \frac{1}{5} (3e_{1} \pm 4e_{2}), \frac{1}{5} (4e_{1} \pm 3e_{2}) \} \hspace{2mm}  \text{ and } \hspace{2mm}  \Lambda^{-} \triangleq \{\frac{1}{5} (-3e_{1}\mp 4e_{2}), \frac{1}{5} (-4 e_{1} \mp 3e_{2}) \}, 
\end{equation} 
i.e. $\Lambda^{-} = -\Lambda^{+}$, and $\Lambda \triangleq \Lambda^{+} \cup \Lambda^{-}$, where $e_{j}$ for $j \in \{1,2\}$ is a standard basis of $\mathbb{R}^{2}$. It follows immediately that $\Lambda \subset \mathbb{S}^{1} \cap \mathbb{Q}^{2}$, $5 \Lambda \subset \mathbb{Z}^{2}$, and  
\begin{equation}\label{[Equation (4.1b), LQ20]}
\min_{\zeta, \zeta' \in \Lambda: \hspace{0.5mm} \zeta \neq - \zeta'} \lvert \zeta + \zeta' \rvert \geq \frac{\sqrt{2}}{5}
\end{equation} 
(cf. \cite[pg. 110]{BV19a}, \cite[Equation (9)]{LT20}). For all $\zeta \in \Lambda$ and any frequency parameter $\lambda \in 5 \mathbb{N}$, we define $b_{\zeta}$ and its potential $\psi_{\zeta}$ as 
\begin{equation}\label{[Equation (4.2), LQ20]}
b_{\zeta}(x) \triangleq b_{\zeta, \lambda} (x) \triangleq i \zeta^{\bot} e^{i \lambda \zeta \cdot x}, \hspace{3mm} \psi_{\zeta} (x) \triangleq \psi_{\zeta, \lambda} (x) \triangleq \frac{1}{\lambda} e^{i \lambda \zeta \cdot x}
\end{equation} 
(cf. \cite[Equation (14)]{CDS12}). It follows that for all $N \in \mathbb{N}_{0}$, 
\begin{subequations}
\begin{align}
& b_{\zeta}(x)  = \nabla^{\bot} \psi_{\zeta}(x), \hspace{3mm} \nabla\cdot b_{\zeta}(x) = 0, \hspace{3mm} \nabla^{\bot} \cdot b_{\zeta}(x) = \Delta \psi_{\zeta}(x) = -\lambda^{2} \psi_{\zeta}(x), \label{[Equation (4.3), LQ20]}\\
& \overline{b_{\zeta}}(x) = b_{-\zeta}(x), \hspace{2mm} \overline{\psi_{\zeta}}(x) = \psi_{-\zeta}(x), \hspace{2mm}  \lVert b_{\zeta} \rVert_{C_{x}^{N}} \overset{\eqref{C-t,x} }{\leq}(N+1) \lambda^{N}, \hspace{2mm} \lVert \psi_{\zeta} \rVert_{C_{x}^{N}} \overset{\eqref{C-t,x} }{\leq} (N+1) \lambda^{N-1}.   \label{[Equations (4.4) and (4.5), LQ20]} 
\end{align}
\end{subequations} 

\begin{lemma}\label{[Lemma 4.1, LQ20]}
\rm{(Geometric lemma from \cite[Lemma 4.1]{LQ20}; cf. \cite[Lemma 3.2]{DS13}, \cite[Lemma 6]{CDS12})} Denote by $\mathcal{M}$ the linear space of $2\times 2$ symmetric trace-free matrices. Then there exists a set of positive smooth functions $\{\gamma_{\zeta} \in C^{\infty} (\mathcal{M}): \hspace{0.5mm} \zeta \in \Lambda \}$ such that for each $\mathring{R} \in \mathcal{M}$, 
\begin{equation}\label{[Equations (4.6), (4.7) and (4.8), LQ20]}
\gamma_{-\zeta} (\mathring{R}) = \gamma_{\zeta} (\mathring{R}), \hspace{2mm} \mathring{R} = \sum_{\zeta \in \Lambda} (\gamma_{\zeta} (\mathring{R} ))^{2} (\zeta \mathring{\otimes} \zeta), \hspace{2mm} \text{ and } \hspace{2mm}  \gamma_{\zeta} (\mathring{R}) \lesssim (1+ \lvert \mathring{R} \rvert)^{\frac{1}{2}}. 
\end{equation}  
\end{lemma} 
For convenience we set 
\begin{equation}\label{estimate 4} 
C_{\Lambda} \triangleq 2 \sqrt{12} (4 \pi^{2} + 1)^{\frac{1}{2}} \lvert \Lambda \rvert \hspace{1mm}  \text{ and } \hspace{1mm}  M \triangleq C_{\Lambda} \sup_{\zeta \in \Lambda} ( \lVert \gamma_{\zeta} \rVert_{C(B_{\frac{1}{2}} (0))} + \lVert \nabla \gamma_{\zeta} \rVert_{C(B_{\frac{1}{2}} (0))}). 
\end{equation} 
Similarly to \cite[pg. 111]{BV19a} we consider a 2-d Dirichlet kernel for $r \in \mathbb{N}$
\begin{equation}\label{[Equations (4.9) and (4.9a), LQ20]} 
D_{r}(x) \triangleq \frac{1}{2r+1} \sum_{k \in \Omega_{r}} e^{ik\cdot x} \hspace{1mm} \text{ where } \hspace{1mm}
\Omega_{r} \triangleq \{k = 
\begin{pmatrix}
k^{1} & k^{2}
\end{pmatrix}^{T}: \hspace{0.5mm} k^{i} \in \mathbb{Z} \cap [-r, r] \text{ for } i = 1,2 \},
\end{equation}
where $T$ denotes a transpose, that satisfies 
\begin{equation}\label{[Equation (4.10), LQ20]}
\lVert D_{r} \rVert_{L_{x}^{p}} \lesssim r^{1- \frac{2}{p}}, \hspace{2mm} \text{ and } \hspace{2mm}  \lVert D_{r} \rVert_{L_{x}^{2}} = 2 \pi \hspace{3mm} \forall \hspace{1mm} p \in (1, \infty]. 
\end{equation} 
The role of $r$ is to parametrize the number of frequencies along edges of the cube $\Omega_{r}$. We introduce $\sigma$ such that $\lambda \sigma \in 5 \mathbb{N}$ to parametrize the spacing between frequencies, or equivalently such that the resulting rescaled kernel is $(\mathbb{T}/\lambda \sigma)^{2}$-periodic. In particular, this will be needed in application of Lemma \ref{[Lemma 6.2, LQ20]} in \eqref{[Equation (77), Y20b]}. Lastly, $\mu$ measures the amount of temporal oscillation in the building blocks. In sum, the parameters we introduced are required to satisfy 
\begin{equation}\label{[Equation (4.14), LQ20]}
1 \ll r \ll \mu \ll \sigma^{-1} \ll \lambda, \hspace{2mm} r \in \mathbb{N}, \hspace{2mm} \text{ and } \hspace{2mm} \lambda, \lambda \sigma \in 5 \mathbb{N}.
\end{equation}
Now we define the directed-rescaled Dirichlet kernel by 
\begin{equation}\label{[Equation (4.11), LQ20]}
\eta_{\zeta} (t,x) \triangleq \eta_{\zeta, \lambda, \sigma, r, \mu} (t,x) \triangleq 
\begin{cases}
D_{r} (\lambda \sigma (\zeta \cdot x + \mu t), \lambda \sigma \zeta^{\bot} \cdot x) & \text{ if } \zeta \in \Lambda^{+},\\
\eta_{-\zeta, \lambda, \sigma, r, \mu} (t,x) & \text{ if } \zeta \in \Lambda^{-}, 
\end{cases}
\end{equation} 
so that 
\begin{subequations}
\begin{align}
& \frac{1}{\mu} \partial_{t} \eta_{\zeta} (t,x) = \pm (\zeta\cdot\nabla) \eta_{\zeta} (t,x) \hspace{1mm} \forall \hspace{1mm} \zeta \in \Lambda^{\pm}, \label{[Equation (4.12), LQ20]} \\
& \fint_{\mathbb{T}^{2}} \eta_{\zeta}^{2} (t,x) dx = 1, \hspace{2mm} \text{ and } \hspace{2mm}  \lVert \eta_{\zeta} \rVert_{L_{t}^{\infty} L_{x}^{p}} \lesssim r^{1- \frac{2}{p}} \hspace{1mm} \forall \hspace{1mm} p \in (1,\infty]   \label{[Equation (4.13), LQ20]}
\end{align}
\end{subequations} 
(cf. \cite[Equations (3.8)-(3.10)]{BV19a}). Finally, we define the intermittent 2-d stationary flow as 
\begin{equation}\label{[Equation (4.15), LQ20]}
\mathbb{W}_{\zeta} (t,x) \triangleq \mathbb{W}_{\zeta, \lambda, \sigma, r, \mu} (t,x) \triangleq \eta_{\zeta,\lambda,\sigma,r,\mu} (t,x) b_{\zeta,\lambda}(x) 
\end{equation} 
(cf. \cite[Equation (3.11)]{BV19a}). Similarly to the 3-d case in \cite{BV19a} it follows that for all $\zeta, \zeta' \in \Lambda$ (see \cite[Equations (4.16)-(4.19)]{LQ20})
\begin{subequations}
\begin{align}
& \mathbb{P}_{\leq 2 \lambda} \mathbb{P}_{\geq \frac{\lambda}{2}} \mathbb{W}_{\zeta} = \mathbb{W}_{\zeta}, \label{[Equation (4.16), LQ20]}\\
& \mathbb{P}_{\leq 4 \lambda} \mathbb{P}_{\geq \frac{\lambda}{5}} (\mathbb{W}_{\zeta} \mathring{\otimes} \mathbb{W}_{\zeta'}) =\mathbb{W}_{\zeta} \mathring{\otimes} \mathbb{W}_{\zeta'} \hspace{3mm} \text{ if } \zeta + \zeta' \neq 0,\label{[Equation (4.17), LQ20]}\\
& \mathbb{P}_{\geq \frac{\lambda \sigma}{2}} (\mathbb{W}_{\zeta} \mathring{\otimes} \mathbb{W}_{\zeta'}) = \mathbb{P}_{\neq 0} (\mathbb{W}_{\zeta} \mathring{\otimes} \mathbb{W}_{\zeta'}), \label{[Equation (4.18), LQ20]}\\
& \mathbb{P}_{\neq 0} \eta_{\zeta} = \mathbb{P}_{\geq \frac{\lambda \sigma}{2}}\eta_{\zeta}. \label{[Equation (4.19), LQ20]}
\end{align}
\end{subequations} 

\begin{lemma}\label{[Lemma 4.2, LQ20]}
\rm{ (\cite[Lemma 4.2]{LQ20}; cf. \cite[Proposition 3.4]{BV19a})} Define $\mathbb{W}_{\zeta}$ by \eqref{[Equation (4.15), LQ20]}. Then for any $\{a_{\zeta} \}_{\zeta \in \Lambda} \subset \mathbb{C}$ such that $a_{-\zeta} = \bar{a}_{\zeta}$, a function $\sum_{\zeta \in \Lambda} a_{\zeta}$ is $\mathbb{R}$-valued and for all $\mathring{R} \in \mathcal{M}$, 
\begin{equation}\label{[Equation (4.21), LQ20]}
\sum_{\zeta \in \Lambda} (\gamma_{\zeta} (\mathring{R}))^{2} \fint_{\mathbb{T}^{2}} \mathbb{W}_{\zeta} \mathring{\otimes} \mathbb{W}_{-\zeta} dx = - \mathring{R}. 
\end{equation} 
\end{lemma} 
\begin{lemma}\label{[Lemma 4.3, LQ20]}
\rm{ (\cite[Lemma 4.3]{LQ20}; cf. \cite[Proposition 3.5]{BV19a})} Define $\eta_{\zeta}$ and $\mathbb{W}_{\zeta}$ respectively by \eqref{[Equation (4.11), LQ20]} and \eqref{[Equation (4.15), LQ20]}, and assume \eqref{[Equation (4.14), LQ20]}. Then for any $p \in (1,\infty]$, $k$, $N \in \{ 0, 1, 2, 3\}$, 
\begin{subequations}
\begin{align}
& \lVert \nabla^{N} \partial_{t}^{k} \mathbb{W}_{\zeta} \rVert_{L_{t}^{\infty} L_{x}^{p}} \lesssim_{N, k, p} \lambda^{N} (\lambda \sigma r \mu)^{k} r^{1- \frac{2}{p}}, \label{[Equation (4.22), LQ20]}\\
& \lVert \nabla^{N} \partial_{t}^{k} \eta_{\zeta} \rVert_{L_{t}^{\infty} L_{x}^{p}} \lesssim_{N, k, p} (\lambda \sigma r)^{N} (\lambda \sigma r \mu)^{k} r^{1- \frac{2}{p}}. \label{[Equation (4.23), LQ20]}
\end{align}
\end{subequations} 
\end{lemma} 

\section{Proofs of Theorems \ref{[Theorem 2.1, Y20b]}-\ref{[Theorem 2.2, Y20b]}}\label{Proof in the case of additive noise}

\subsection{Proof of Theorem \ref{[Theorem 2.2, Y20b]} assuming Theorem \ref{[Theorem 2.1, Y20b]} }\label{Subsection 4.1}
We first present general results for $F$ defined through \eqref{[Equations (11) and (12), Y20b]}; thereafter, we apply them in case $F \equiv 1$ and $B$ is a $GG^{\ast}$-Wiener process to prove Theorems \ref{[Theorem 2.1, Y20b]}-\ref{[Theorem 2.2, Y20b]}.  We fix $\varepsilon \in (0,1)$ for the following definitions, which are in the spirit of previous works such as \cite{FR08, GRZ09, SV97}. 
\begin{define}\label{[Definition 4.1, Y20b]}
Let $s \geq 0$ and $\xi^{\text{in}} \in L_{\sigma}^{2}$. Then $P \in \mathcal{P} (\Omega_{0})$ is a martingale solution to \eqref{[Equation (2), Y20b]} with initial condition $\xi^{\text{in}}$ at initial time $s$ if 
\begin{enumerate}
\item [] (M1) $P(\{ \xi(t) = \xi^{\text{in}} \hspace{1mm} \forall \hspace{1mm} t \in [0,s]\}) = 1$ and for all $n \in \mathbb{N}$ 
\begin{equation}\label{[Equation (13), Y20b]}
P ( \{ \xi \in \Omega_{0}: \hspace{0.5mm}\int_{0}^{n} \lVert F(\xi(r)) \rVert_{L_{2} (U, L_{\sigma}^{2})}^{2} dr < \infty \}) = 1, 
\end{equation} 
\item [] (M2) for every $\mathfrak{g}_{i} \in C^{\infty} (\mathbb{T}^{2} ) \cap L_{\sigma}^{2}$ and $t \geq s$ 
\begin{equation}\label{[Equation (14), Y20b]}
M_{t,s}^{i} \triangleq \langle \xi(t)- \xi(s), \mathfrak{g}_{i} \rangle + \int_{s}^{t} \langle \text{div} ( \xi(r) \otimes \xi(r)) + (-\Delta)^{m} \xi(r), \mathfrak{g}_{i} \rangle dr 
\end{equation} 
is a continuous, square-integrable $(\mathcal{B}_{t})_{t\geq s}$-martingale under $P$ such that $\langle \langle M_{t,s}^{i} \rangle \rangle = \int_{s}^{t} \lVert F(\xi(r))^{\ast} \mathfrak{g}_{i} \rVert_{U}^{2} dr$, 
\item [] (M3) for any $q \in \mathbb{N}$ there exists a function $t \mapsto C_{t,q} \in \mathbb{R}_{+}$ for all $t\geq s$ such that 
\begin{equation}\label{[Equation (15), Y20b]}
\mathbb{E}^{P} [ \sup_{r \in [0,t]} \lVert \xi(r) \rVert_{L_{x}^{2}}^{2q} + \int_{s}^{t} \lVert \xi(r) \rVert_{H_{x}^{\varepsilon}}^{2} dr] \leq C_{t,q} (1+ \lVert \xi^{\text{in}} \rVert_{L_{x}^{2}}^{2q}). 
\end{equation} 
\end{enumerate}
 The set of all such martingale solutions with the same constants $C_{t,q}$ in \eqref{[Equation (15), Y20b]} for every $q \in \mathbb{N}$ and $t\geq s$ will be denoted by $\mathcal{C}( s, \xi^{\text{in}}, \{C_{t,q} \}_{q \in \mathbb{N}, t \geq s})$. 
\end{define}  
In the case of an additive noise, if $\{\mathfrak{g}_{i} \}_{i=1}^{\infty}$ is a complete orthonormal system consisting of eigenvectors of $GG^{\ast}$, then $M_{t,s}\triangleq \sum_{i=1}^{\infty} M_{t,s}^{i} \mathfrak{g}_{i}$ becomes a $GG^{\ast}$-Wiener process w.r.t. the filtration $(\mathcal{B}_{t})_{t\geq s}$ under $P$. Given any stopping time $\tau: \hspace{0.5mm}\Omega_{0} \mapsto [0,\infty]$ we define the space of trajectories stopped at $\tau$ by 
\begin{equation}\label{[Equation (16), Y20b]}
\Omega_{0,\tau} \triangleq \{\omega( \cdot \wedge \tau(\omega)):\hspace{1mm} \omega \in \Omega_{0} \} 
\end{equation} 
and denote the $\sigma$-field associated to $\tau$ by $\mathcal{B}_{\tau}$. 
\begin{define}\label{[Definition 4.2, Y20b]}
Let $s \geq 0$, $\xi^{\text{in}} \in L_{\sigma}^{2}$, and $\tau \geq s$ be a stopping time of $(\mathcal{B}_{t})_{t \geq s}$. Then $P \in \mathcal{P} (\Omega_{0, \tau})$ is a martingale solution to \eqref{[Equation (2), Y20b]} on $[s, \tau]$ with initial condition $\xi^{\text{in}}$ at initial time $s$ if 
\begin{enumerate}
\item [] (M1) $P(\{ \xi(t) = \xi^{\text{in}} \hspace{1mm} \forall \hspace{1mm} t \in [0, s] \}) = 1$ and for all $n \in \mathbb{N}$
\begin{equation}\label{[Equation (17), 20b]}
P ( \{ \xi \in \Omega_{0}:\hspace{1mm} \int_{0}^{n \wedge \tau} \lVert F(\xi(r)) \rVert_{L_{2}(U, L_{\sigma}^{2})}^{2} dr < \infty \}) = 1, 
\end{equation} 
\item [] (M2) for every $\mathfrak{g}_{i} \in C^{\infty} (\mathbb{T}^{2}) \cap L_{\sigma}^{2}$ and $t \geq s$ 
\begin{equation}\label{[Equation (18), Y20b]} 
M_{t \wedge \tau, s}^{i} \triangleq \langle \xi(t\wedge \tau) - \xi^{\text{in}}, \mathfrak{g}_{i} \rangle + \int_{s}^{t \wedge \tau} \langle \text{div} (\xi(r) \otimes \xi(r)) + (-\Delta)^{m} \xi(r), \mathfrak{g}_{i} \rangle dr 
\end{equation} 
is a continuous, square-integrable $(\mathcal{B}_{t})_{t \geq s}$-martingale under $P$ such that $\langle \langle M_{t\wedge \tau, s}^{i} \rangle \rangle$ $=$ $\int_{s}^{t \wedge \tau} \lVert F(\xi(r))^{\ast} \mathfrak{g}_{i} \rVert_{U}^{2} dr$, 
\item [] (M3) for any $q \in \mathbb{N}$ there exists a function $t \mapsto C_{t,q} \in \mathbb{R}_{+}$ for all $t\geq s$ such that
\begin{equation}\label{[Equation (19), Y20b]}
\mathbb{E}^{P} [ \sup_{r \in [0, t \wedge \tau]} \lVert \xi(r) \rVert_{L_{x}^{2}}^{2q} + \int_{s}^{t \wedge \tau} \lVert \xi(r) \rVert_{H_{x}^{\varepsilon}}^{2} dr] \leq C_{t,q} (1+ \lVert \xi^{\text{in}} \rVert_{L_{x}^{2}}^{2q}). 
\end{equation} 
\end{enumerate} 
\end{define}

\begin{proposition}\label{[Proposition 4.1, Y20b]}
For any $(s, \xi^{\text{in}}) \in [0, \infty) \times L_{\sigma}^{2}$, there exits $P \in \mathcal{P} (\Omega_{0})$ which is a martingale solution to \eqref{[Equation (2), Y20b]} with initial condition $\xi^{\text{in}}$ at initial time $s$ according to Definition \ref{[Definition 4.1, Y20b]}. Additionally, if there exists a family $\{(s_{n}, \xi_{n}) \}_{n \in \mathbb{N}} \subset [0,\infty) \times L_{\sigma}^{2}$ such that $\lim_{n\to\infty} \lVert (s_{n}, \xi_{n}) - (s, \xi^{\text{in}}) \rVert_{\mathbb{R} \times L_{x}^{2}} = 0$ and $P_{n} \in \mathcal{C} ( s_{n}, \xi_{n}, \{C_{t,q} \}_{q \in \mathbb{N}, t \geq s_{n}} )$, then there exists a subsequence $\{P_{n_{k}} \}_{k\in\mathbb{N}}$ that converges weakly to some $P \in \mathcal{C} ( s, \xi^{\text{in}}, \{C_{t,q}\}_{q \in \mathbb{N}, t \geq s } )$. 
\end{proposition} 

\begin{proof}[Proof of Proposition \ref{[Proposition 4.1, Y20b]}]
We omit the proof of the existence of a martingale solution because it has become very standard by now; we refer to \cite[Theorem 4.1]{FR08} for 3-d NS equations, \cite[Theorem 6.2]{GRZ09} for a more general case of spatial dimension, as well as \cite[Theorem 4.2.4]{Z12} for the case of a diffusive term in the form of a fractional Laplacian with an arbitrary small exponent (see also \cite[Theorem 3.1]{Y19}). The stability result can also be proven following the proof of \cite[Theorem 3.1]{HZZ19} (also \cite[Proposition 4.1]{Y20a}); because the estimates can differ slightly due to the arbitrary weak diffusion in the current case, we leave a sketch of proof elaborating on treatments of diffusive terms in the Appendix for completeness. 
\end{proof}
Proposition \ref{[Proposition 4.1, Y20b]} leads to the following results; the proofs of analogous results in \cite{HZZ19} did not depend on spatial dimension or specific form of diffusive terms and thus directly apply to our case. 

\begin{lemma}\label{[Lemma 4.2, Y20b]} 
\rm{(\cite[Proposition 3.2]{HZZ19})} Let $\tau$ be a bounded stopping time of $(\mathcal{B}_{t})_{t\geq 0}$. Then for every $\omega \in \Omega_{0}$ there exists $Q_{\omega} \in \mathcal{P}(\Omega_{0})$ such that 
\begin{subequations}
\begin{align}
&Q_{\omega} (\{ \omega' \in \Omega_{0}: \hspace{0.5mm}  \xi(t, \omega') = \omega(t) \hspace{1mm} \forall \hspace{1mm} t \in [0, \tau(\omega)] \}) = 1, \label{[Equation (20a), Y20b]}\\
&Q_{\omega} (A) = R_{\tau(\omega), \xi(\tau(\omega), \omega)} (A) \hspace{1mm} \forall \hspace{1mm} A \in \mathcal{B}^{\tau(\omega)}, \label{[Equation (20b), Y20b]}
\end{align}
\end{subequations}
where $R_{\tau(\omega), \xi(\tau(\omega), \omega)} \in \mathcal{P}(\Omega_{0})$ is a martingale solution to \eqref{[Equation (2), Y20b]} with initial condition $\xi(\tau(\omega), \omega)$ at initial time $\tau(\omega)$. Furthermore, for every $B \in \mathcal{B}$ the map $\omega \mapsto Q_{\omega}(B)$ is $\mathcal{B}_{\tau}$-measurable.
\end{lemma} 
Let us mention that in the proof of Lemma \ref{[Lemma 4.2, Y20b]}, $Q_{\omega}$ is derived as the unique probability measure 
\begin{equation}\label{[Equation (21), Y20b]}
Q_{\omega} = \delta_{\omega} \otimes_{\tau(\omega)} R_{\tau(\omega), \xi(\tau(\omega), \omega)} \in \mathcal{P} (\Omega_{0}),
\end{equation}
where $\delta_{\omega}$ is the Dirac mass, such that \eqref{[Equation (20a), Y20b]}-\eqref{[Equation (20b), Y20b]} hold. 

\begin{lemma}\label{[Lemma 4.3, Y20b]}
\rm{(\cite[Proposition 3.4]{HZZ19})} Let $\xi^{\text{in}} \in L_{\sigma}^{2}$ and $P$ be a martingale solution to \eqref{[Equation (2), Y20b]} on $[0,\tau]$ with initial condition $\xi^{\text{in}}$ at initial time $0$ according to Definition \ref{[Definition 4.2, Y20b]}. Assume the hypothesis of Lemma \ref{[Lemma 4.2, Y20b]} and additionally that there exists a Borel set $\mathcal{N} \subset \Omega_{0, \tau}$ such that $P(\mathcal{N}) = 0$ and $Q_{\omega}$ from Lemma \ref{[Lemma 4.2, Y20b]} satisfies for every $\omega \in \Omega_{0} \setminus \mathcal{N}$ 
\begin{equation}\label{[Equation (22), Y20b]}
Q_{\omega} (\{\omega' \in \Omega_{0}: \hspace{0.5mm}  \tau(\omega') = \tau(\omega) \}) = 1. 
\end{equation} 
Then a probability measure $P \otimes_{\tau} R \in \mathcal{P} (\Omega_{0})$ defined by 
\begin{equation}\label{[Equation (23), Y20b]}
P \otimes_{\tau} R(\cdot) \triangleq \int_{\Omega_{0}} Q_{\omega} (\cdot) P(d\omega) 
\end{equation} 
satisfies $P \otimes_{\tau} R \rvert_{\Omega_{0,\tau}} = P \rvert_{\Omega_{0, \tau}}$ and it is a martingale solution to \eqref{[Equation (2), Y20b]} on $[0,\infty)$ with initial condition $\xi^{\text{in}}$ at initial time 0. 
\end{lemma} 
Now we split \eqref{[Equation (2), Y20b]} to 
\begin{subequations}
\begin{align}
& dz + (-\Delta)^{m} z dt + \nabla \pi^{1} dt = dB, \hspace{2mm} \nabla\cdot z = 0 \text{ for } t > 0, \hspace{2mm} z(0,x) \equiv 0, \label{[Equation (24), Y20b]}\\
& \partial_{t} v + (-\Delta)^{m} v + \text{div} ((v+z) \otimes (v+z)) + \nabla \pi^{2} = 0, \nabla\cdot v =0 \text{ for } t > 0, v(0,x) = u^{\text{in}}(x) \label{[Equation (25), Y20b]}
\end{align} 
\end{subequations} 
so that $u = v+ z$ solves \eqref{[Equation (2), Y20b]} with $\pi = \pi^{1} + \pi^{2}$ starting from $u^{\text{in}}$ at $t = 0$. We fix a $GG^{\ast}$-Wiener process $B$ on $(\Omega, \mathcal{F}, \textbf{P})$ with $(\mathcal{F}_{t})_{t\geq 0}$ as the canonical filtration of $B$ augmented by all the $\textbf{P}$-negligible sets and apply Definitions \ref{[Definition 4.1, Y20b]}-\ref{[Definition 4.2, Y20b]}, Proposition \ref{[Proposition 4.1, Y20b]}, and Lemmas \ref{[Lemma 4.2, Y20b]}-\ref{[Lemma 4.3, Y20b]} with $F \equiv 1$ and such $B$. 
 
\begin{proposition}\label{[Proposition 4.4, Y20b]}
Suppose that $m \in (0,1)$ and that Tr$((-\Delta)^{2- m + 2 \sigma} GG^{\ast}) < \infty$ for some $\sigma > 0$. Then for all $\delta \in (0,\frac{1}{2})$ and $T> 0$, 
\begin{equation}\label{[Equation (26), Y20b]}
\mathbb{E}^{\textbf{P}} [ \lVert z \rVert_{C_{T} H_{x}^{\frac{4+ \sigma}{2}}} + \lVert z \rVert_{C_{T}^{\frac{1}{2} - \delta} H_{x}^{\frac{2+ \sigma}{2}}}] < \infty. 
\end{equation} 
\end{proposition}
\begin{proof}[Proof of Proposition \ref{[Proposition 4.4, Y20b]}]
Similarly to \cite[Proposition 3.6]{HZZ19} and \cite[Proposition 4.4]{Y20a}, this follows from a straight-forward modification of the proof of \cite[Proposition 34]{D13} and an application of Kolmogorov's test \cite[Theorem 3.3]{DZ14}. Because our diffusion is significantly weaker than the cases in \cite{HZZ19, Y20a}, we require a stronger hypothesis on $G$. In short, one can define 
\begin{equation}
Y(s) \triangleq \frac{\sin(\pi \alpha)}{\pi} \int_{0}^{s} e^{- (-\Delta)^{m} (s-r)} (s-r)^{-\alpha} \mathbb{P} dB(r) \hspace{1mm} \text{ where } \hspace{1mm}  \alpha \in (0, \frac{3\sigma}{4m}), 
\end{equation} 
show that $\mathbb{E}^{\textbf{P}}[ \lVert (-\Delta)^{\frac{4+ \sigma}{4}} Y \rVert_{L_{T}^{2k} L_{x}^{2}}^{2k}] \lesssim_{k} 1$ for all $k \in \mathbb{N}$ using Tr$((-\Delta)^{2- m + 2 \sigma} GG^{\ast}) < \infty$ from hypothesis, use the identities of 
\begin{equation*}
z(t) = \int_{0}^{t} e^{- (t-r) (-\Delta)^{m}} \mathbb{P} dB(r) \hspace{2mm} \text{ and } \hspace{2mm} \int_{r}^{t} (t-s)^{\alpha -1} (s-r)^{-\alpha} ds = \frac{\pi}{\sin(\alpha \pi)} \text{ for any } \alpha \in (0,1) 
\end{equation*} 
respectively from \eqref{[Equation (24), Y20b]} and \cite[pg. 131]{DZ14} to write 
\begin{equation*}
\int_{0}^{t} (t-s)^{\alpha -1} e^{- (-\Delta)^{m} (t-s)} Y(s) ds = z(t), 
\end{equation*} 
and conclude the first bound. This immediately gives for any $\beta < \frac{1}{2}$  
\begin{equation}
\mathbb{E}^{\textbf{P}} [ \sup_{t, t+h \in [0,T]} \lVert (-\Delta)^{\frac{2+ \sigma}{4}} (z(t+h) - z(t)) \rVert_{L_{x}^{2}}^{2k}] \lesssim_{\sigma, m, \beta, k, T} \lvert h \rvert^{2\beta k}
\end{equation} 
(we refer to \cite[Equation (55)]{D13} and \cite[Proposition A.1.1]{DZ96}) so that applying Kolmogorov's test deduces the second bound. We refer to \cite[Proposition 3.4]{D13} for complete details.
\end{proof}
Next, for every $\omega \in \Omega_{0}$ we define 
\begin{subequations}
\begin{align}
& M_{t,0}^{\omega} \triangleq \omega(t) - \omega(0) + \int_{0}^{t} \mathbb{P} \text{div} (\omega(r) \otimes \omega(r)) + (-\Delta)^{m} \omega(r) dr, \label{[Equation (28), Y20b]}\\
& Z^{\omega} (t) \triangleq M_{t,0}^{\omega} - \int_{0}^{t} \mathbb{P} (-\Delta)^{m} e^{- (t-r) (-\Delta)^{m} } M_{r,0}^{\omega} dr. \label{[Equation (29), Y20b]} 
\end{align} 
\end{subequations}
If $P$ is a martingale solution to \eqref{[Equation (2), Y20b]}, then $M$ is a $GG^{\ast}$-Wiener process under $P$ and it follows from \eqref{[Equation (28), Y20b]}-\eqref{[Equation (29), Y20b]} that we can write 
\begin{equation}\label{[Equation (30), Y20b]}
Z(t) = \int_{0}^{t} \mathbb{P} e^{- (t-r) (-\Delta)^{m}} dM_{r,0}. 
\end{equation} 
We can deduce from Proposition \ref{[Proposition 4.4, Y20b]} that $P$-a.s. $Z \in C_{T} H_{x}^{\frac{4+ \sigma}{2}} \cap C_{T}^{\frac{1}{2} - \delta} H_{x}^{\frac{2+ \sigma}{2}}$. For $n \in \mathbb{N}$ and $\delta \in (0, \frac{1}{12})$ we define 
\begin{subequations}
\begin{align}
\tau_{L}^{n} (\omega) \triangleq& \inf\{ t \geq 0: \hspace{0.5mm} C_{S}\lVert Z^{\omega} (t) \rVert_{H_{x}^{\frac{4+ \sigma}{2}}} > (L - \frac{1}{n})^{\frac{1}{4}} \} \nonumber\\
& \hspace{10mm} \wedge \inf \{t \geq 0: \hspace{0.5mm} C_{S}\lVert Z^{\omega} \rVert_{C_{t}^{\frac{1}{2} - 2\delta}H_{x}^{\frac{2+ \sigma}{2}}} >  (L - \frac{1}{n})^{\frac{1}{2}} \} \wedge L, \label{[Equation (31), Y20b]}\\
\tau_{L} \triangleq& \lim_{n\to\infty} \tau_{L}^{n}, \label{[Equation (32), Y20b]}
\end{align} 
\end{subequations}
where $C_{S} > 0$ is the Sobolev constant such that $\lVert f \rVert_{L_{x}^{\infty}} \leq C_{S} \lVert f \rVert_{H_{x}^{\frac{2+ \sigma}{2}}}$ for all $f \in H^{\frac{2+ \sigma}{2}}(\mathbb{T}^{2})$, so that $(\tau_{L}^{n})_{n\in\mathbb{N}}$ is non-decreasing in $n$. By \cite[Lemma 3.5]{HZZ19} it follows that $\tau_{L}^{n}$ is a stopping time of $(\mathcal{B}_{t})_{t\geq 0}$ for all $n \in \mathbb{N}$ and hence so is $\tau_{L}$. 

Next, we shall assume Theorem \ref{[Theorem 2.1, Y20b]} on a probability space $(\Omega, \mathcal{F}, (\mathcal{F}_{t})_{t \geq 0}, \textbf{P})$ and denote by $P$ the law of the solution $u$ constructed from Theorem \ref{[Theorem 2.1, Y20b]}. 

\begin{proposition}\label{[Proposition 4.5, Y20b]}
Let $\tau_{L}$ be defined by \eqref{[Equation (32), Y20b]}. Then $P$, the law of $u$, is a martingale solution of \eqref{[Equation (2), Y20b]} on $[0, \tau_{L}]$ according to Definition \ref{[Definition 4.2, Y20b]}. 
\end{proposition} 

\begin{proof}[Proof of Proposition \ref{[Proposition 4.5, Y20b]}]
For $C_{S} > 0$ from \eqref{[Equation (31), Y20b]}, $L > 1$, and $\delta \in (0, \frac{1}{12})$, we define 
\begin{equation}\label{[Equation (33), Y20b]}
T_{L} \triangleq \inf \{t \geq 0: \hspace{0.5mm} C_{S} \lVert z(t) \rVert_{H_{x}^{\frac{4+ \sigma}{2}}} \geq L^{\frac{1}{4}} \} \wedge \inf\{t \geq 0: \hspace{0.5mm} C_{S} \lVert z \rVert_{C_{t}^{\frac{1}{2} - 2 \delta} H_{x}^{\frac{2+ \sigma}{2}}} \geq L^{\frac{1}{2}} \} \wedge L. 
\end{equation} 
Due to Proposition \ref{[Proposition 4.4, Y20b]} we see that $\textbf{P}$-a.s. $T_{L} > 0$ and $T_{L} \nearrow + \infty$ as $L \nearrow + \infty$. The stopping time $\mathfrak{t}$ in the statement of Theorem \ref{[Theorem 2.1, Y20b]} is actually $T_{L}$ for $L > 0$ sufficiently large. The rest of the proof of Proposition \ref{[Proposition 4.5, Y20b]} follows that of \cite[Proposition 3.7]{HZZ19} (see also \cite[Proposition 4.5]{Y20a}). 
\end{proof}

\begin{proposition}\label{[Proposition 4.6, Y20b]}
Let $\tau_{L}$ be defined by \eqref{[Equation (32), Y20b]} and $P$ denote the law of $u$ constructed from Theorem \ref{[Theorem 2.1, Y20b]}. Then the probability measure $P \otimes_{\tau_{L}} R$ in \eqref{[Equation (23), Y20b]} is a martingale solution to \eqref{[Equation (2), Y20b]} on $[0,\infty)$ according to Definition \ref{[Definition 4.1, Y20b]}. 
\end{proposition}

\begin{proof}[Proof of Proposition \ref{[Proposition 4.6, Y20b]}]
Because $\tau_{L}$ is a stopping time of $(\mathcal{B}_{t})_{t \geq 0}$ that is bounded by $L$ due to \eqref{[Equation (31), Y20b]}, and $P$ is a martingale solution to \eqref{[Equation (2), Y20b]} on $[0, \tau_{L}]$ due to Proposition \ref{[Proposition 4.5, Y20b]}, Lemma \ref{[Lemma 4.3, Y20b]} gives us the desired result once we verify \eqref{[Equation (22), Y20b]}. The rest of the proof follows that of \cite[Proposition 3.8]{HZZ19} (see also \cite[Proposition 4.6]{Y20a}). 
\end{proof}

Taking Theorem \ref{[Theorem 2.1, Y20b]} for granted we are ready to prove Theorem \ref{[Theorem 2.2, Y20b]}. 
\begin{proof}[Proof of Theorem \ref{[Theorem 2.2, Y20b]} assuming Theorem \ref{[Theorem 2.1, Y20b]}]
This follows from the proof of \cite[Theorem 1.2]{HZZ19} (see also the proof of \cite[Theorem 2.2]{Y20a}).  In short, we can fix $T > 0$ arbitrarily, any $\kappa \in (0,1)$ and $K > 1$ such that $\kappa K^{2} \geq 1$, rely on Theorem \ref{[Theorem 2.1, Y20b]} and Proposition \ref{[Proposition 4.6, Y20b]} to deduce the existence of $L > 1$ and a measure $P \otimes_{\tau_{L}} R$ that is a martingale solution to \eqref{[Equation (2), Y20b]} on $[0,\infty)$ and coincides with $P$, the law of the solution constructed in Theorem \ref{[Theorem 2.1, Y20b]}, over a random interval $[0, \tau_{L} ]$. Therefore, $P \otimes_{\tau_{L}} R$ starts  with a deterministic initial condition $\xi^{\text{in}}$ from the proof of Theorem \ref{[Theorem 2.1, Y20b]}. It follows that 
\begin{align}
P \otimes_{\tau_{L}} R( \{ \tau_{L} \geq T \}) \overset{\eqref{[Equation (23), Y20b]}}{=} \int_{\Omega_{0}} Q_{\omega} (\{ \omega' \in \Omega_{0}: \tau_{L} (\omega') \geq T\}) P (d \omega) = \textbf{P} ( \{ T_{L} \geq T \})  
> \kappa    \label{[Equation (34), Y20b]}
\end{align} 
where the last inequality is due to Theorem \ref{[Theorem 2.1, Y20b]}. Consequently, 
\begin{align}
\mathbb{E}^{P \otimes_{\tau_{L}} R} [ \lVert \xi(T) \rVert_{L_{x}^{2}}^{2}] \overset{\eqref{[Equation (5), Y20b]} \eqref{[Equation (34), Y20b]} }{>} \kappa [ K \lVert \xi^{\text{in}} \rVert_{L_{x}^{2}} + K (T \text{Tr} (GG^{\ast} ))^{\frac{1}{2}}]^{2} \geq  \kappa K^{2}(\lVert \xi^{\text{in}} \rVert_{L_{x}^{2}}^{2} + T \text{Tr} (GG^{\ast})). \label{[Equation (35), Y20b]}
\end{align} 
On the other hand, it is well known that a Galerkin approximation can give us another martingale solution $\Theta$ (e.g., \cite{FR08}) which starts from the same initial condition $\xi^{\text{in}}$ and satisfies 
\begin{equation*}
\mathbb{E}^{\Theta} [ \lVert \xi(T)\rVert_{L_{x}^{2}}^{2}] \leq \lVert \xi^{\text{in}} \rVert_{L_{x}^{2}}^{2} + T\text{Tr}(GG^{\ast}).
\end{equation*} 
Because $\kappa K^{2} \geq 1$, this implies $P \otimes_{\tau_{L}} R\neq \Theta$ and hence a lack uniqueness in law for \eqref{[Equation (2), Y20b]}. 
\end{proof}

\subsection{Proof of Theorem \ref{[Theorem 2.1, Y20b]} assuming Proposition \ref{[Proposition 4.8, Y20b]}}\label{Subsection 4.2}
Considering \eqref{[Equation (25), Y20b]}, for $q \in \mathbb{N}_{0}$ we will construct a solution $(v_{q}, \mathring{R}_{q})$ to 
\begin{equation}\label{[Equation (36), Y20b]}
\partial_{t} v_{q} + (-\Delta)^{m} v_{q} + \text{div} ((v_{q} + z) \otimes (v_{q} + z)) + \nabla \pi_{q} = \text{div} \mathring{R}_{q}, \hspace{3mm} \nabla\cdot v_{q} = 0, \hspace{3mm} t > 0, 
\end{equation} 
where $\mathring{R}_{q}$ is assumed to be a trace-free symmetric matrix. For any $a \in 10\mathbb{N}, b \in \mathbb{N}$, $\beta \in (0,1)$, and $L \geq 1$, to be selected more precisely in Sub-Subsection \ref{Sub-subsection 4.3.1}, we define 
\begin{equation}\label{[Equations (37) and (39), Y20b]}
\lambda_{q} \triangleq a^{b^{q}}, \hspace{3mm} \delta_{q} \triangleq \lambda_{q}^{-2\beta}, \hspace{3mm} M_{0}(t) \triangleq L^{4} e^{4Lt},  
\end{equation} 
from which we see that $\lambda_{q+1} \in 10 \mathbb{N} \subset 5 \mathbb{N}$, as required in \eqref{[Equation (4.14), LQ20]}. The reason why we take $a \in 10 \mathbb{N}$ rather than $a \in 5 \mathbb{N}$ will be e.g. explained after \eqref{estimate 19}. Due to Sobolev embedding in $\mathbb{T}^{2}$ we see from \eqref{[Equation (33), Y20b]} that for any $\delta \in (0, \frac{1}{12})$ and $t \in [0, T_{L}]$ 
\begin{equation}\label{[Equation (38), Y20b]}
\lVert z(t) \rVert_{L_{x}^{\infty}} \leq L^{\frac{1}{4}}, \hspace{3mm} \lVert \nabla z(t) \rVert_{L_{x}^{\infty}} \leq L^{\frac{1}{4}}, \hspace{3mm} \lVert z \rVert_{C_{t}^{\frac{1}{2} - 2 \delta} L_{x}^{\infty}} \leq L^{\frac{1}{2}}.
\end{equation} 
Let us observe that if $a^{\beta b} > 3$ and $b \geq 2$, then $\sum_{1\leq \iota \leq q} \delta_{\iota}^{\frac{1}{2}} < \frac{1}{2}$ for any $q \in \mathbb{N}$. We set the convention that $\sum_{1\leq \iota \leq 0} \triangleq 0$, denote by $c_{R} > 0$ a universal small constant to be described subsequently throughout the proof of Proposition \ref{[Proposition 4.8, Y20b]} (e.g., \eqref{[Equation (72), Y20b]}), and assume the following bounds over $t \in [0, T_{L}]$ inductively: 
\begin{subequations}\label{[Equation (40), Y20b]}
\begin{align}
& \lVert v_{q} \rVert_{C_{t}L_{x}^{2}} \leq M_{0}(t)^{\frac{1}{2}} (1 + \sum_{1 \leq \iota \leq q} \delta_{\iota}^{\frac{1}{2}}) \leq 2 M_{0}(t)^{\frac{1}{2}}, \label{[Equation (40a), Y20b]}\\
& \lVert v_{q} \rVert_{C_{t,x}^{1}} \leq M_{0}(t)^{\frac{1}{2}} \lambda_{q}^{4}, \label{[Equation (40b), Y20b]}\\
& \lVert \mathring{R}_{q} \rVert_{C_{t}L_{x}^{1}} \leq M_{0}(t) c_{R} \delta_{q+1}. \label{[Equation (40c), Y20b]}
\end{align} 
\end{subequations} 
We denote an anti-divergence operator by $\mathcal{R}$ in the following proposition (see Lemma \ref{[Definition 9, Lemma 10, CDS12]}). 

\begin{proposition}\label{[Proposition 4.7, Y20b]}
Let 
\begin{equation}\label{[Equation (41), Y20b]}
v_{0} (t,x) \triangleq \frac{L^{2} e^{2Lt}}{2\pi} 
\begin{pmatrix}
\sin(x^{2}) &
0
\end{pmatrix}^{T}. 
\end{equation}
Then together with 
 \begin{equation}\label{[Equation (42), Y20b]}
 \mathring{R}_{0}(t,x) \triangleq \frac{ 2L^{3} e^{2Lt}}{2\pi} 
 \begin{pmatrix}
 0 & - \cos(x^{2}) \\
 - \cos(x^{2}) & 0 
 \end{pmatrix}  
 + (\mathcal{R} (-\Delta)^{m} v_{0} + v_{0} \mathring{\otimes} z + z \mathring{\otimes} v_{0} + z \mathring{\otimes} z)(t,x), 
\end{equation} 
 it satisfies \eqref{[Equation (36), Y20b]} at level $q = 0$. Moreover, \eqref{[Equation (40), Y20b]} are satisfied at level $q = 0$ provided 
 \begin{equation}\label{[Equation (43), Y20b]}
 (50) 9 \pi^{2} < 50 \pi^{2} a^{2 \beta b} \leq c_{R} L \leq c_{R} (a^{4} \pi - 1)
 \end{equation} 
 where the inequality $9 < a^{2\beta b}$ is assumed only for the justification of the second inequality in \eqref{[Equation (40a), Y20b]}. Furthermore, $v_{0}(0,x)$ and $\mathring{R}_{0}(0,x)$ are both deterministic. 
\end{proposition} 

\begin{proof}[Proof of Proposition \ref{[Proposition 4.7, Y20b]}]
Using the facts that the divergence of a matrix $(A^{ij})_{1\leq i,j\leq 2}$ is a 2-d vector, of which $k$-th component is $\sum_{j=1}^{2} \partial_{j} A^{kj}$ and that $\text{div} (v_{0} \otimes v_{0})  = 0$, one can immediately verify that $v_{0}$ and $\mathring{R}_{0}$ from \eqref{[Equation (41), Y20b]}-\eqref{[Equation (42), Y20b]} satisfy \eqref{[Equation (36), Y20b]} at level $q = 0$ if we choose $\pi_{0} = - (v_{0} \cdot z + \frac{1}{2} \lvert z \rvert^{2})$. We also point out that $v_{0}$ is divergence-free while $\mathring{R}_{0}$ is trace-free and symmetric due to Lemma \ref{[Definition 9, Lemma 10, CDS12]}, as required. Next, we can compute  
\begin{equation}\label{[Equations (44a) and (44b), Y20b]} 
\lVert v_{0} (t) \rVert_{L_{x}^{2}} =  \frac{M_{0}(t)^{\frac{1}{2}}}{\sqrt{2}} \leq M_{0}(t)^{\frac{1}{2}}, \hspace{3mm} \lVert v_{0} \rVert_{C_{t,x}^{1}} 
= \frac{L^{2} e^{2L t} (L+1)}{\pi}  \overset{\eqref{[Equation (43), Y20b]}}{\leq} M_{0}(t)^{\frac{1}{2}} \lambda_{0}^{4}, 
\end{equation} 
and thus \eqref{[Equation (40a), Y20b]}-\eqref{[Equation (40b), Y20b]} at level $q = 0$ hold. Next, we can compute 
\begin{equation}\label{estimate 34}
\lVert \mathring{R}_{0} (t) \rVert_{L_{x}^{1}} \overset{\eqref{[Equation (38), Y20b]} \eqref{[Equations (44a) and (44b), Y20b]}}{\leq} 16L^{3} e^{2L t} + 2 \pi \lVert \mathcal{R} (-\Delta)^{m} v_{0} \rVert_{L_{x}^{2}} + 20 \pi M_{0}(t)^{\frac{1}{2}} L^{\frac{1}{4}} + 5 (2\pi)^{2} L^{\frac{1}{2}}. 
\end{equation} 
Using the facts that $v_{0}$ is mean-zero, divergence-free, and satisfies $\Delta v_{0} = - v_{0}$ we can rely on   \eqref{estimate 2} and interpolation to deduce  
\begin{equation}\label{estimate 35}
\lVert \mathcal{R} (-\Delta)^{m} v_{0} \rVert_{L_{x}^{2}} \leq 2( \lVert v_{0} \rVert_{L_{x}^{2}} + \lVert \Delta v_{0} \rVert_{L_{x}^{2}})
= 4 \lVert v_{0} \rVert_{L_{x}^{2}}. 
\end{equation} 
Therefore, due to the second inequality of \eqref{[Equation (43), Y20b]}, continuing from \eqref{estimate 34} we obtain 
\begin{equation}
\lVert \mathring{R}_{0} (t) \rVert_{L_{x}^{1}} \overset{\eqref{[Equations (37) and (39), Y20b]} \eqref{[Equations (44a) and (44b), Y20b]} \eqref{estimate 35}}{\leq} 16L M_{0}(t)^{\frac{1}{2}} + 8\pi M_{0}(t)^{\frac{1}{2}} + 20 \pi M_{0}(t)^{\frac{1}{2}} L^{\frac{1}{4}} + 5 (2\pi)^{2} L^{\frac{1}{2}}  \overset{\eqref{[Equation (43), Y20b]}}{\leq}   M_{0}(t) c_{R} \delta_{1}. 
\end{equation} 
This verifies \eqref{[Equation (40c), Y20b]} at level $q= 0$. Finally, it is clear that $v_{0}(0,x)$ is deterministic, and consequently $\mathring{R}_{0}(0,x)$ is also deterministic because $z(0,x) \equiv 0$ from \eqref{[Equation (24), Y20b]}. 
 \end{proof} 

\begin{proposition}\label{[Proposition 4.8, Y20b]}
Let $L >  (50) 9 \pi^{2} c_{R}^{-1}$ and suppose that $(v_{q}, \mathring{R}_{q})$ is an $(\mathcal{F}_{t})_{t\geq 0}$-adapted process that solves \eqref{[Equation (36), Y20b]} and satisfies \eqref{[Equation (40), Y20b]}. Then there exists a choice of parameters $a, b,$ and $\beta$ such that \eqref{[Equation (43), Y20b]} is fulfilled and an $(\mathcal{F}_{t})_{t\geq 0}$-adapted process $(v_{q+1}, \mathring{R}_{q+1})$ that satisfies \eqref{[Equation (36), Y20b]}, \eqref{[Equation (40), Y20b]} at level $q+1$, and 
\begin{equation}\label{[Equation (45), Y20b]}
\lVert v_{q+1} (t) - v_{q}(t) \rVert_{L_{x}^{2}} \leq M_{0}(t)^{\frac{1}{2}} \delta_{q+1}^{\frac{1}{2}}. 
\end{equation}
Moreover, if $v_{q}(0,x)$ and $\mathring{R}_{q}(0,x)$ are deterministic, then so are $v_{q+1}(0,x)$ and $\mathring{R}_{q+1}(0,x)$. 
\end{proposition}

Taking Proposition \ref{[Proposition 4.8, Y20b]} granted we can now prove Theorem \ref{[Theorem 2.1, Y20b]}. 
\begin{proof}[Proof of Theorem \ref{[Theorem 2.1, Y20b]} assuming Proposition \ref{[Proposition 4.8, Y20b]}] 
The proof is similar to that of \cite[Theorem 1.1]{HZZ19} (see also the proof of \cite[Theorem 2.1]{Y20a}); we sketch it for completeness. Given $T > 0, K > 1$, and $\kappa \in (0,1)$, starting from $(v_{0}, \mathring{R}_{0})$ in Proposition \ref{[Proposition 4.7, Y20b]}, Proposition \ref{[Proposition 4.8, Y20b]} gives us $(v_{q}, \mathring{R}_{q})$ for $q \geq 1$ that satisfies \eqref{[Equation (40), Y20b]} and \eqref{[Equation (45), Y20b]}. Then, for all $\varepsilon \in (0, \frac{\beta}{4+ \beta})$ and $t \in [0, T_{L}]$, by Gagliardo-Nirenberg's inequality, and the fact that $b^{q+1} \geq b(q+1)$ for all $q \geq 0$ and $b\geq 2$, we can deduce 
\begin{equation}\label{[Equation (46), Y20b]}
\sum_{q \geq 0} \lVert v_{q+1}(t) - v_{q}(t) \rVert_{H_{x}^{\varepsilon}} \lesssim \sum_{q\geq 0} M_{0}(t)^{\frac{1-\varepsilon}{2}} \delta_{q+1}^{\frac{1-\varepsilon}{2}} (M_{0}(t)^{\frac{1}{2}} \lambda_{q+1}^{4})^{\varepsilon} \lesssim M_{0}(t)^{\frac{1}{2}}. 
\end{equation} 
Therefore, we can deduce the existence of $\lim_{q\to\infty} v_{q} \triangleq v \in C([0, T_{L}]; H^{\varepsilon} (\mathbb{T}^{2}))$ for which there exists a deterministic constant $C_{L} > 0$ such that 
\begin{equation}\label{[Equation (47), Y20b]}
\sup_{t \in [0, T_{L}]} \lVert v(t) \rVert_{H_{x}^{\varepsilon}} \leq C_{L}. 
\end{equation} 
As each $v_{q}$ is $(\mathcal{F}_{t})_{t\geq 0}$-adapted, it follows that $v$ is also $(\mathcal{F}_{t})_{t\geq 0}$-adapted. Furthermore, for all $t \in [0, T_{L}]$, $\lVert \mathring{R}_{q} \rVert_{C_{t}L_{x}^{1}} \to 0$ as $q\to+\infty$ due to \eqref{[Equation (40c), Y20b]}. Therefore, $v$ is a weak solution to \eqref{[Equation (25), Y20b]} over $[0, T_{L}]$; consequently, we see from \eqref{[Equation (24), Y20b]} that $u = v+ z$ solves \eqref{[Equation (2), Y20b]}. Now for $c_{R} > 0$ to be determined from the proof of Proposition \ref{[Proposition 4.8, Y20b]}, we can choose $L = L(T, K, c_{R}, \text{Tr} (GG^{\ast} )) > (50) 9 \pi^{2} c_{R}^{-1}$ larger if necessary to satisfy 
\begin{equation}\label{[Equation (48), Y20b]}
\frac{3}{2} + \frac{1}{L} < ( \frac{1}{\sqrt{2}} - \frac{1}{2}) e^{LT} \text{ and } L^{\frac{1}{4}} 2 \pi + K (T \text{Tr} (GG^{\ast} ))^{\frac{1}{2}} \leq (e^{LT} - K) \lVert u^{\text{in}} \rVert_{L_{x}^{2}} + L e^{LT}
\end{equation} 
where $u^{\text{in}}(x) = v(0,x)$ as $z(0,x)  \equiv 0$ from \eqref{[Equation (25), Y20b]}. Because $\lim_{L\to\infty} T_{L} = + \infty$ $\textbf{P}$-a.s. due to Proposition \ref{[Proposition 4.4, Y20b]}, for the fixed $T> 0$ and $\kappa > 0$, increasing $L$ larger if necessary allows us to obtain $\textbf{P} ( \{T_{L} \geq T \}) > \kappa$. Now because $z(t)$ from \eqref{[Equation (24), Y20b]} is $(\mathcal{F}_{t})_{t\geq 0}$-adapted, we see that $u$ is $(\mathcal{F}_{t})_{t\geq 0}$-adapted. Moreover, \eqref{[Equation (47), Y20b]} and \eqref{[Equation (38), Y20b]} imply \eqref{[Equation (4), Y20b]}. Next, we compute 
\begin{equation}\label{[Equation (49), Y20b]}
\lVert v(t) - v_{0}(t) \rVert_{L_{x}^{2}}  \overset{\eqref{[Equation (45), Y20b]}}{\leq} M_{0}(t)^{\frac{1}{2}} \sum_{q \geq 0} a^{-b^{q+1} \beta} 
\leq M_{0}(t)^{\frac{1}{2}} \sum_{q\geq 0} a^{-b(q+1) \beta}  \overset{\eqref{[Equation (43), Y20b]}}{<} M_{0}(t)^{\frac{1}{2}} (\frac{1}{2})
\end{equation}
for all $t \in [0, T_{L}]$. We also see by utilizing \eqref{[Equation (48), Y20b]} that 
\begin{equation}\label{[Equation (50), Y20b]}
( \lVert v(0) \rVert_{L_{x}^{2}} + L) e^{LT} \overset{\eqref{[Equations (44a) and (44b), Y20b]} \eqref{[Equation (49), Y20b]}}{\leq} (\frac{3}{2} M_{0} (0)^{\frac{1}{2}} + L) e^{LT} \overset{\eqref{[Equation (48), Y20b]}}{<} (\frac{1}{\sqrt{2}} - \frac{1}{2}) M_{0}(T)^{\frac{1}{2}} \overset{\eqref{[Equations (44a) and (44b), Y20b]} \eqref{[Equation (49), Y20b]}}{<} \lVert v(T) \rVert_{L_{x}^{2}}. 
\end{equation}
Therefore, on $\{T_{L} \geq T \}$ 
\begin{equation}\label{[Equation (51), Y20b]}
\lVert u(T) \rVert_{L_{x}^{2}} \overset{\eqref{[Equation (50), Y20b]}}{>} ( \lVert v(0) \rVert_{L_{x}^{2}} + L)e^{LT} - \lVert z(T) \rVert_{L_{x}^{\infty}} 2\pi 
\overset{\eqref{[Equation (24), Y20b]} \eqref{[Equation (38), Y20b]}\eqref{[Equation (48), Y20b]} }{\geq} K \lVert u^{\text{in}} \rVert_{L_{x}^{2}} + K ( T \text{Tr} ( GG^{\ast} ))^{\frac{1}{2}}, 
\end{equation} 
which implies \eqref{[Equation (5), Y20b]}. At last, because $v_{0} (0,x)$ is deterministic from Proposition \ref{[Proposition 4.7, Y20b]}, Proposition \ref{[Proposition 4.8, Y20b]} implies that $u^{\text{in}} (x) = v(0,x)$ remains deterministic. 
\end{proof}

\subsection{Proof of Proposition \ref{[Proposition 4.8, Y20b]}}

\subsubsection{Choice of parameters}\label{Sub-subsection 4.3.1}
Let us define 
\begin{equation}\label{[Equation (2.2), LQ20]}
m^{\ast} \triangleq 
\begin{cases}
2m -1 & \text{ if } m \in (\frac{1}{2}, 1), \\
0 & \text{ if } m \in (0, \frac{1}{2}]; 
\end{cases} 
\end{equation} 
it follows that $m^{\ast} \in [0, 1)$. Furthermore, we fix 
\begin{equation}\label{[Equation (2.3), LQ20]}
\eta \in \mathbb{Q}_{+} \cap (\frac{1- m^{\ast}}{16}, \frac{1-m^{\ast}}{8}] 
\end{equation} 
from which we see that $\eta \in (0, \frac{1}{8}]$. We also fix $L > (50)9 \pi^{2} c_{R}^{-1}$ and  
\begin{equation}\label{alpha}
\alpha \triangleq \frac{1-m}{400}. 
\end{equation} 
We set 
\begin{equation}\label{estimate 22}
r \triangleq \lambda_{q+1}^{1- 6 \eta}, \hspace{2mm} \mu \triangleq \lambda_{q+1}^{1-4\eta}, \hspace{2mm} \text{ and } \hspace{2mm}  \sigma \triangleq \lambda_{q+1}^{2\eta - 1},
\end{equation} 
from which we immediately observe that $1 \ll r \ll \mu \ll \sigma^{-1} \ll \lambda_{q+1}$ from \eqref{[Equation (4.14), LQ20]} is satisfied.  Moreover, for the $\alpha > 0$ fixed we can choose $b \in \{\iota \in \mathbb{N}: \hspace{0.5mm} \iota > \frac{16}{\alpha}\}$ such that $r \in \mathbb{N}$ and $\lambda_{q+1} \sigma \in 10 \mathbb{N}$ so that the conditions of $r \in \mathbb{N}$ and $\lambda_{q+1} \sigma \in 5 \mathbb{N}$  from \eqref{[Equation (4.14), LQ20]} are satisfied. Indeed, because $\eta \in \mathbb{Q}_{+} \cap (0, \frac{1}{8}]$, we can write $1-6\eta = \frac{n_{1}}{d_{1}}$ and $2\eta = \frac{n_{2}}{d_{2}}$ for some $n_{1}, n_{2}, d_{1}, d_{2} \in \mathbb{N}$, and then take $b \in\mathbb{N}$ to be a multiple of $d_{1}d_{2}$; it follows that $r = \lambda_{q+1}^{1- 6 \eta} = a^{b^{q+1} (1-6\eta)} \in \mathbb{N}$ and $\lambda_{q+1} \sigma = \lambda_{q+1}^{2\eta} = a^{b^{q+1} 2\eta} \in 10 \mathbb{N}$ as $a \in 10 \mathbb{N}$. For the $\alpha$ from \eqref{alpha} and such $b> 0$ fixed, we take $\beta > 0$ sufficiently small so that 
\begin{equation}\label{estimate 21}
\alpha > 16 \beta b. 
\end{equation} 
We also choose 
\begin{equation}\label{[Equation (56), Y20b]}
l \triangleq \lambda_{q+1}^{- \frac{3\alpha}{2}} \lambda_{q}^{-2}. 
\end{equation} 
Together with the condition that $b > \frac{16}{\alpha}$, by taking $a \in 10 \mathbb{N}$ sufficiently large we obtain 
\begin{equation}\label{[Equation (57), Y20b]}
l \lambda_{q}^{4} \leq \lambda_{q+1}^{-\alpha} \hspace{2mm} \text{ and } \hspace{2mm} l^{-1} \leq \lambda_{q+1}^{2\alpha}.
\end{equation} 

\begin{remark}\label{Remark 4.1}
We will have numerous requirements that $\alpha \in (0, C \eta)$ for various constants $C > 0$; e.g., the second inequality of \eqref{[Equation (77), Y20b]} will require that we bound 
\begin{align*}
\lambda_{q+1}^{-\frac{1}{2}} \sigma^{-\frac{1}{2}} l^{-\frac{11}{2}} \overset{\eqref{[Equation (57), Y20b]} \eqref{estimate 22}}{\leq} \lambda_{q+1}^{-\frac{1}{2}} \lambda_{q+1}^{\frac{1}{2} - \eta} \lambda_{q+1}^{11\alpha} 
\end{align*}
by a constant that does not depend on relevant parameters and therefore we need $\alpha \leq \frac{\eta}{11}$. Thus, to be able to fix the value of $\alpha$ explicitly as we did in \eqref{alpha}, we decided to restrict $\eta$ to have the lower bound of $\frac{1-m^{\ast}}{16}$ in \eqref{[Equation (2.3), LQ20]}, differently from \cite[Equation (2.3)]{LQ20}. It follows that $\alpha$ defined in \eqref{alpha} indeed satisfies $\alpha \leq \frac{\eta}{11}$ as 
\begin{align*}
\alpha \overset{\eqref{alpha}}{=} \frac{1-m}{400} \overset{\eqref{[Equation (2.2), LQ20]}}{\leq} (\frac{1-m^{\ast}}{16}) (\frac{1}{11}) \overset{\eqref{[Equation (2.3), LQ20]}}{\leq} \frac{\eta}{11}, 
\end{align*}
and we will see that our choice of $\alpha$ in \eqref{alpha} will satisfy all other instances when it needs to be sufficiently smaller w.r.t. $\eta$.  
\end{remark} 

Concerning \eqref{[Equation (43), Y20b]}, taking $a \in 10 \mathbb{N}$ sufficiently large gives $c_{R} L \leq c_{R}(a^{4} \pi - 1)$  while $\beta>  0$ sufficiently small allows $(50) 9 \pi^{2} < 50 \pi^{2} a^{2\beta b} \leq c_{R} L$. Because we chose $L$ such that $L > (50)9 \pi^{2} c_{R}^{-1}$, this is possible. Thus, we shall hereafter consider such $m^{\ast}, \eta, \alpha, b$, and $l$ fixed, preserving our freedom to take $a \in 10 \mathbb{N}$ larger and $\beta > 0$ smaller as necessary. 

\begin{remark}\label{Remark 4.2}
Let us remark on some differences in our choice of parameters and those of other works. First, the work of \cite{LQ20} did not have a parameter that is equivalent to our $\alpha$ (The ``$\alpha$'' in \cite[Equation (2.3)]{LQ20} is actually our $\eta$ defined in \eqref{[Equation (2.3), LQ20]}). Our $\alpha$ in \eqref{alpha} plays the role of defining $l = \lambda_{q+1}^{- \frac{3\alpha}{2}} \lambda_{q}^{-2}$ in \eqref{[Equation (56), Y20b]}. Instead, the choice of $l = \lambda_{q}^{-20}$ is taken in \cite[Equation (3.1)]{LQ20}, which has appeared in others' previous works (e.g., \cite[Equation (4.16)]{BV19a}). As we described already in Remark \ref{Remark 4.1}, parts of our proof such as \eqref{[Equation (77), Y20b]} required $\alpha$ to be taken small w.r.t. $\eta$ and because $\eta$ in \eqref{[Equation (2.3), LQ20]} depends on $m^{\ast}$ defined in \eqref{[Equation (2.2), LQ20]} which in turn depends on $m$, we chose $l = \lambda_{q+1}^{- \frac{3\alpha}{2}} \lambda_{q}^{-2}$ where $\alpha$ depends on $m$ via \eqref{alpha} following \cite[Equation (4.17)]{HZZ19} and \cite[Equation (69)]{Y20a}.

On the other had, the works of \cite{BV19a, HZZ19} did not have a parameter that is equivalent to our $\eta$ in \eqref{[Equation (2.3), LQ20]} because  \cite{BV19a, HZZ19} were concerned with the Navier-Stokes equations and hence there was no parameter $m$. For further references we note that after this work was completed, a parameter that is analogous to $\eta$ in \eqref{[Equation (2.3), LQ20]} continued to see utility in others' works (e.g., \cite[Equations (2.3)]{CL21a} and \cite[Equation (92)]{Y21}). 
\end{remark} 

\subsubsection{Mollification}
We let $\{\phi_{\epsilon} \}_{\epsilon > 0}$ and $\{\varphi_{\epsilon} \}_{\epsilon > 0}$, specifically $\phi_{\epsilon} (\cdot) \triangleq \frac{1}{\epsilon^{2}} \phi(\frac{\cdot}{\epsilon})$ and $\varphi_{\epsilon}(\cdot) \triangleq \frac{1}{\epsilon} \varphi(\frac{\cdot}{\epsilon})$, respectively be families of standard mollifiers on $\mathbb{R}^{2}$ and $\mathbb{R}$ with mass one where the latter is compactly supported on $\mathbb{R}_{+}$. Then we mollify $v_{q}, \mathring{R}_{q}$, and $z$ to obtain 
\begin{equation}\label{[Equation (58), Y20b]}
v_{l} \triangleq (v_{q} \ast_{x} \phi_{l}) \ast_{t} \varphi_{l}, \hspace{3mm} \mathring{R}_{l} \triangleq (\mathring{R}_{q} \ast_{x} \phi_{l}) \ast_{t} \varphi_{l}, \hspace{3mm} z_{l} \triangleq (z \ast_{x} \phi_{l}) \ast_{t} \varphi_{l}. 
\end{equation} 
It follows from \eqref{[Equation (36), Y20b]} that $v_{l}$ satisfies  
\begin{equation}\label{[Equation (59), Y20b]}
\partial_{t} v_{l} + (-\Delta)^{m} v_{l} + \text{div} ((v_{l} + z_{l}) \otimes (v_{l} + z_{l} )) + \nabla \pi_{l} =  \text{div} (\mathring{R}_{l} + R_{\text{com1}})
\end{equation}
if 
\begin{subequations}\label{estimate 36}
\begin{align}
\pi_{l} \triangleq& (\pi_{q} \ast_{x} \phi_{l}) \ast_{t} \varphi_{l} - \frac{1}{2} (\lvert v_{l} + z_{l} \rvert^{2} - ( \lvert v_{q} + z \rvert^{2} \ast_{x} \phi_{l}) \ast_{t} \varphi_{l}), \label{[Equation (60b), Y20b]} \\
R_{\text{com1}} \triangleq& R_{\text{commutator1}} \triangleq (v_{l} + z_{l}) \mathring{\otimes} (v_{l} + z_{l}) - (((v_{q} + z) \mathring{\otimes} (v_{q} + z) )\ast_{x} \phi_{l} ) \ast_{t} \varphi_{l}. \label{[Equation (60a), Y20b]} 
\end{align}
\end{subequations} 
We can estimate for all $t \in [0, T_{L}]$ and $N \geq 1$, by using the fact that $\beta \ll \alpha$ from \eqref{estimate 21} and taking $a \in 10 \mathbb{N}$ sufficiently large 
\begin{subequations}
\begin{align}
& \lVert v_{q} - v_{l} \rVert_{C_{t}L_{x}^{2}}  \overset{\eqref{[Equation (40b), Y20b]}}{\lesssim} l M_{0}(t)^{\frac{1}{2}} \lambda_{q}^{4} \overset{\eqref{[Equation (57), Y20b]}}{\leq}  \frac{1}{4} M_{0}(t)^{\frac{1}{2}} \delta_{q+1}^{\frac{1}{2}}, \label{[Equation (61a), Y20b]}\\
&\lVert v_{l} \rVert_{C_{t}L_{x}^{2}} \leq \lVert v_{q} \rVert_{C_{t}L_{x}^{2}} \overset{\eqref{[Equation (40a), Y20b]}}{\leq} M_{0}(t)^{\frac{1}{2}} (1+ \sum_{1 \leq \iota \leq q} \delta_{\iota}^{\frac{1}{2}}), \label{[Equation (61b), Y20b]} \\
&\lVert v_{l} \rVert_{C_{t,x}^{N}} \overset{\eqref{[Equation (40b), Y20b]}}{\lesssim} l^{-N+1} M_{0}(t)^{\frac{1}{2}} \lambda_{q}^{4} \overset{\eqref{alpha} \eqref{[Equation (56), Y20b]}}{\leq} l^{-N} M_{0}(t)^{\frac{1}{2}} \lambda_{q+1}^{-\alpha}. \label{[Equation (61c), Y20b]}
\end{align}
\end{subequations} 

\subsubsection{Perturbation}
We let $\chi$ be a smooth function such that 
\begin{equation}\label{[Equation (64), Y20b]}
\chi(z) \triangleq 
\begin{cases}
1 & \text{ if } z \in [0, 1],\\
z & \text{ if } z \in [2, \infty), 
\end{cases}
\end{equation}
and $z \leq 2 \chi(z) \leq 4 z$ for $z \in (1,2)$. We define for $t \in [0, T_{L}]$ and $\omega \in \Omega$ 
\begin{equation}\label{[Equation (65), Y20b]}
\rho(\omega, t, x) \triangleq 4 c_{R} \delta_{q+1} M_{0}(t) \chi((c_{R} \delta_{q+1} M_{0} (t))^{-1} \lvert \mathring{R}_{l} (\omega, t, x) \rvert). 
\end{equation} 
Then it follows that 
\begin{equation}\label{[Equation (66), Y20b]}
\lvert \frac{ \mathring{R}_{l} (\omega, t,x) }{\rho(\omega, t, x) } \rvert = \frac{ \lvert \mathring{R}_{l} (\omega, t, x) \rvert}{ 4c_{R} \delta_{q+1} M_{0}(t) \chi((c_{R} \delta_{q+1} M_{0}(t))^{-1} \lvert \mathring{R}_{l} (\omega, t, x) \rvert)} \leq \frac{1}{2}. 
\end{equation} 
We can estimate for any $p \in [1, \infty]$ and $t \in [0, T_{L}]$ 
\begin{align}
\lVert \rho(\omega) \rVert_{C_{t}L_{x}^{p}} \overset{\eqref{[Equation (64), Y20b]}}{\leq}& \sup_{s \in [0,t]} 4c_{R} \delta_{q+1} M_{0}(s) \lVert 1 + 3 (c_{R} \delta_{q+1} M_{0}(s))^{-1} \lvert \mathring{R}_{l} (\omega, s, x) \rvert \rVert_{L_{x}^{p}} \nonumber\\
& \hspace{25mm} \leq12 ((4\pi^{2})^{\frac{1}{p}} c_{R} \delta_{q+1} M_{0}(t) + \lVert \mathring{R}_{l} (\omega) \rVert_{C_{t}L_{x}^{p}}). \label{[Equation (67), Y20b]} 
\end{align}
Next, for any $N \geq 0$ and $t \in [0, T_{L}]$, due to the embedding of $W^{3,1} (\mathbb{T}^{2}) \hookrightarrow L^{\infty} (\mathbb{T}^{2})$, 
\begin{equation}\label{[Equation (68), Y20b]}
\lVert \mathring{R}_{l} \rVert_{C_{t,x}^{N}} \overset{\eqref{C-t,x}}{\lesssim} \sum_{0\leq n + \lvert \alpha \rvert \leq N} \lVert \partial_{t}^{n} D^{\alpha} (-\Delta)^{\frac{3}{2}} \mathring{R}_{l} \rVert_{L_{t}^{\infty} L_{x}^{1}} \overset{\eqref{[Equation (40c), Y20b]}}{\lesssim} l^{-N - 3} M_{0}(t) c_{R} \delta_{q+1}. 
\end{equation}
For any $N \geq 0$, $k \in \{0,1,2\}$, and $t \in [0, T_{L}]$ we can deduce by taking $a \in 10 \mathbb{N}$ sufficiently large 
\begin{equation}\label{[Equation (69), Y20b]}
\lVert \rho \rVert_{C_{t}C_{x}^{N}} \lesssim c_{R} \delta_{q+1} M_{0}(t) l^{-3-N} \hspace{2mm} \text{ and } \hspace{2mm} \lVert \rho \rVert_{C_{t}^{1}C_{x}^{k}}  \lesssim c_{R} \delta_{q+1} M_{0}(t) l^{-4(k+1)}.
\end{equation} 
Indeed, the first inequality can be computed using \eqref{[Equation (67), Y20b]}-\eqref{[Equation (68), Y20b]} when $N = 0$, while \eqref{[Equation (64), Y20b]}-\eqref{[Equation (65), Y20b]} and \cite[Equation (129)]{BDIS15} in case $N \geq 1$; the second inequality can be computed by directly applying $\partial_{t}$ and $\nabla$ and then relying on \eqref{[Equation (68), Y20b]}. Next, we define the amplitude function by 
\begin{equation}\label{[Equation (70), Y20b]}
a_{\zeta} (\omega, t, x) \triangleq a_{\zeta, q+1} (\omega, t, x) \triangleq \rho(\omega, t, x)^{\frac{1}{2}} \gamma_{\zeta} ( \frac{ \mathring{R}_{l} (\omega, t, x)}{\rho(\omega, t, x)}). 
\end{equation} 
\begin{remark}\label{Remark 4.3}
We note that analogous definitions of $a_{\zeta}$ in previous works had ``$\text{Id}- \cdot$'' in their arguments; e.g., 
\begin{align*}
\text{``}a_{(\xi)} (\omega, t, x) := a_{\xi, q+1} (\omega, t, x) := \rho(\omega, t, x)^{1/2} \gamma_{\xi} \left(\text{Id} - \frac{ \mathring{R}_{l}(\omega, t, x)}{\rho(\omega, t, x)} \right) (2\pi)^{-\frac{3}{4}}\text{''}
\end{align*}
in \cite[Equation (4.26)]{HZZ19} (see also \cite[Equation (4.12)]{BV19a}). The geometric lemma in the 3-d case that was used in \cite{BV19a, HZZ19}, specifically \cite[Proposition 3.2]{BV19a} and \cite[Lemma B.1]{HZZ19}, had a ball around an identity matrix in the space of $3\times 3$ symmetric matrices as the domain of $\gamma_{\zeta}$. On the other hand, the available geometric lemma in the 2-d case, specifically Lemma \ref{[Lemma 4.1, LQ20]} from \cite[Lemma 4.1]{LQ20}, requires that the argument of $\gamma_{\zeta}$ be not only symmetric but also trace-free. Because $\text{Id} - \frac{ \mathring{R}_{l}(\omega, t, x)}{\rho(\omega, t, x)}$ would not be trace-free, we chose $\frac{\mathring{R}_{l} (\omega, t, x)}{\rho(\omega, t ,x)}$ as the argument. 

Furthermore, our choice of the argument of $a_{\zeta}$ also differs from that of \cite[Equation (5.1)]{LQ20} because theirs includes not only $\mathring{R}_{l}$ but also $R_{\text{com1}}$. We chose to refrain from including $R_{\text{com1}}$ within the argument of $\gamma_{\zeta}$ because in contrast to \cite[Equation (3.6)]{LQ20}, our $R_{\text{com1}}$ in \eqref{[Equation (60a), Y20b]} includes $z$ and requires separate delicate treatments (see \eqref{[Equation (105), Y20b]}). 
\end{remark} 
Next, we have the following identity:
\begin{equation}\label{estimate 3}
\sum_{\zeta, \zeta' \in \Lambda} a_{\zeta}(\omega, t, x) a_{\zeta'}(\omega, t, x) \fint_{\mathbb{T}^{2}} \mathbb{W}_{\zeta} \mathring{\otimes} \mathbb{W}_{\zeta'}(t,x) dx = -\mathring{R}_{l} (\omega, t, x).
\end{equation}
Indeed, the fact that $b_{\zeta} (x) \mathring{\otimes} b_{-\zeta}(x) \overset{\eqref{[Equation (4.2), LQ20]}}{=} - \zeta \mathring{\otimes} \zeta$ leads to 
\begin{equation*}
\sum_{\zeta, \zeta' \in \Lambda} \gamma_{\zeta} (\mathring{R}) \gamma_{\zeta'} (\mathring{R}) \fint_{\mathbb{T}^{2}} \mathbb{W}_{\zeta} \mathring{\otimes} \mathbb{W}_{\zeta'} (t,x) dx \overset{\eqref{[Equations (4.6), (4.7) and (4.8), LQ20]} \eqref{[Equation (4.11), LQ20]} \eqref{[Equation (4.13), LQ20]} \eqref{[Equation (4.15), LQ20]}}{=} - \mathring{R}
\end{equation*} 
which in turn gives \eqref{estimate 3} by using \eqref{[Equation (70), Y20b]}. 
\begin{remark}\label{Remark 4.4}
Let us note that this identity \eqref{estimate 3} differs slightly from the analogous ones previous works, e.g., 
\begin{align*}
\text{``}(2\pi)^{\frac{3}{2}} \sum_{\xi \in \Lambda} a_{(\xi)}^{2}\fint_{\mathbb{T}^{3}} W_{(\xi)} \otimes W_{(\xi)} dx = \rho \text{Id} - \mathring{R}_{l}\text{''}
\end{align*} 
in \cite[Equation (4.27)]{HZZ19} (cf. also \cite[Equation (4.14)]{BV19a}, \cite[Equation (7.30)]{BV19b}, \cite[Equation (5.3)]{LQ20}). The identity \eqref{estimate 3} will be necessary in deriving \eqref{[Equation (7.12a), LQ20]} and ultimately \eqref{estimate 27}-\eqref{estimate 28}.
\end{remark}
Concerning $a_{\zeta}$ we can estimate for all $t \in [0, T_{L}]$ with $C_{\Lambda}$ and $M$ from \eqref{estimate 4}
\begin{equation}\label{[Equation (72), Y20b]}
\lVert a_{\zeta} \rVert_{C_{t}L_{x}^{2}} \overset{\eqref{estimate 4} \eqref{[Equation (66), Y20b]} \eqref{[Equation (67), Y20b]}}{\leq} [12 ( 4 \pi^{2} c_{R} \delta_{q+1} M_{0} (t) + \lVert \mathring{R}_{l} (\omega) \rVert_{C_{t} L_{x}^{1}} )]^{\frac{1}{2}} \frac{M}{C_{\Lambda}}  \overset{\eqref{estimate 4} \eqref{[Equation (40c), Y20b]} }{\leq} \frac{c_{R}^{\frac{1}{4}} M_{0}(t)^{\frac{1}{2}} \delta_{q+1}^{\frac{1}{2}}}{2 \lvert \Lambda \rvert} 
\end{equation} 
by requiring $c_{R}^{\frac{1}{4}} \leq \frac{1}{M}$. We also have for all $t \in [0, T_{L}]$, $N \in \mathbb{N}_{0}$, and $k  \in \{ 0, 1, 2\}$, 
\begin{equation}\label{[Equation (73), Y20b]}
\lVert a_{\zeta} \rVert_{C_{t}C_{x}^{N}} \leq c_{R}^{\frac{1}{4}} \delta_{q+1}^{\frac{1}{2}} M_{0}(t)^{\frac{1}{2}} l^{-\frac{3}{2} - 4N} \hspace{2mm} \text{ and } \hspace{2mm} \lVert a_{\zeta} \rVert_{C_{t}^{1}C_{x}^{k}} \leq c_{R}^{\frac{1}{4}} \delta_{q+1}^{\frac{1}{2}} M_{0}(t)^{\frac{1}{2}} l^{- (k+1) 4}.  
\end{equation} 
Indeed, the first inequality in case $N = 0$ follows from \eqref{estimate 4}, \eqref{[Equation (66), Y20b]}, \eqref{[Equation (69), Y20b]}-\eqref{[Equation (70), Y20b]}, while the first inequality in case $N \in \mathbb{N}$ follows from \eqref{[Equation (66), Y20b]}, \eqref{[Equation (69), Y20b]}-\eqref{[Equation (70), Y20b]}, an application of \cite[Equations (129)-(130)]{BDIS15}, and the fact that $\rho(t) \geq 2c_{R} \delta_{q+1} M_{0}(t)$ due to \eqref{[Equation (64), Y20b]}-\eqref{[Equation (65), Y20b]}. Finally, the second inequality can be verified by applying $\partial_{t}$ and $\nabla$, and relying on \eqref{[Equation (66), Y20b]}, \eqref{[Equation (69), Y20b]}-\eqref{[Equation (70), Y20b]}. 

Next, we recall $\psi_{\zeta}, \eta_{\zeta}, \mathbb{W}_{\zeta}$, and $\mu$ respectively from \eqref{[Equation (4.2), LQ20]}, \eqref{[Equation (4.11), LQ20]}, \eqref{[Equation (4.15), LQ20]}, and \eqref{estimate 22}, and  define the perturbation 
\begin{equation}\label{[Equation (76), Y20b]}
w_{q+1} \triangleq w_{q+1}^{(p)} + w_{q+1}^{(c)} + w_{q+1}^{(t)} \text{ and } v_{q+1} \triangleq v_{l}+  w_{q+1} 
\end{equation} 
where 
\begin{equation}\label{[Equation (74), Y20b]}
 w_{q+1}^{(p)} \triangleq \sum_{\zeta \in \Lambda} a_{\zeta} \mathbb{W}_{\zeta}, \hspace{1mm} w_{q+1}^{(c)} \triangleq\sum_{\zeta \in \Lambda} \nabla^{\bot} (a_{\zeta} \eta_{\zeta}) \psi_{\zeta}, \hspace{1mm} w_{q+1}^{(t)} \triangleq \mu^{-1} (\sum_{\zeta \in \Lambda^{+} } - \sum_{\zeta \in \Lambda^{-}}) \mathbb{P} \mathbb{P}_{\neq 0} (a_{\zeta}^{2} \mathbb{P}_{\neq 0} \eta_{\zeta}^{2} \zeta). 
 \end{equation}
 We have the identity of 
\begin{equation}\label{[Equation (5.9), LQ20]}
 (w_{q+1}^{(p)} + w_{q+1}^{(c)}) (t,x) \overset{\eqref{[Equation (4.3), LQ20]} \eqref{[Equation (4.15), LQ20]} }{=}  \nabla^{\bot} (\sum_{\zeta \in \Lambda} a_{\zeta}(t,x) \eta_{\zeta}(t,x) \psi_{\zeta}(x)). 
\end{equation}
It follows that $w_{q+1}$ is divergence-free and mean-zero. Now by \eqref{[Equation (4.2), LQ20]} and \eqref{[Equation (4.11), LQ20]} we see that $\mathbb{W}_{\zeta}$ in \eqref{[Equation (4.15), LQ20]} is $(\mathbb{T}/\lambda_{q+1} \sigma)^{2}$-periodic. Thus, we can apply Lemma \ref{[Lemma 6.2, LQ20]} to deduce 
\begin{equation}\label{[Equation (77), Y20b]}
\lVert w_{q+1}^{(p)} \rVert_{C_{t}L_{x}^{2}} \overset{ \eqref{[Equation (4.22), LQ20]}\eqref{[Equation (72), Y20b]}  }{\lesssim} \sum_{\zeta \in \Lambda} \frac{c_{R}^{\frac{1}{4}} M_{0}(t)^{\frac{1}{2}} \delta_{q+1}^{\frac{1}{2}}}{\lvert \Lambda \rvert} + \lambda_{q+1}^{-\frac{1}{2}} \sigma^{-\frac{1}{2}} c_{R}^{\frac{1}{4}} \delta_{q+1}^{\frac{1}{2}} M_{0}(t)^{\frac{1}{2}} l^{-\frac{11}{2}}  \overset{\eqref{[Equation (57), Y20b]}}{\lesssim} c_{R}^{\frac{1}{4}} \delta_{q+1}^{\frac{1}{2}} M_{0}(t)^{\frac{1}{2}},  
\end{equation} 
where the last inequality used the fact that $11 \alpha - \eta \leq 0$ due to \eqref{[Equation (2.2), LQ20]}-\eqref{alpha}; preserving $c_{R}^{\frac{1}{4}}$ here will be needed in deriving \eqref{[Equation (81), Y20b]}. Next, for all $p \in (1,\infty)$ and $t \in [0, T_{L}]$ we can estimate 
\begin{subequations}\label{[Equation (78), Y20b]}
\begin{align}
& \lVert w_{q+1}^{(p)} \rVert_{C_{t}L_{x}^{p}} 
\overset{\eqref{[Equation (74), Y20b]}}{\leq} \sup_{s \in [0,t]} \sum_{\zeta \in \Lambda} \lVert a_{\zeta} (s) \rVert_{L_{x}^{\infty}} \lVert \mathbb{W}_{\zeta}(s) \rVert_{L_{x}^{p}} \overset{\eqref{[Equation (4.22), LQ20]} \eqref{[Equation (73), Y20b]} }{\lesssim} \delta_{q+1}^{\frac{1}{2}} M_{0}(t)^{\frac{1}{2}} l^{-\frac{3}{2}} r^{1- \frac{2}{p}},    \label{[Equation (78a), Y20b]}\\
& \lVert w_{q+1}^{(c)} \rVert_{C_{t}L_{x}^{p}} \overset{\eqref{[Equation (74), Y20b]}}{\lesssim} \sup_{s \in [0,t]} \sum_{\zeta \in \Lambda} \lVert \nabla^{\bot} (a_{\zeta} \eta_{\zeta})(s) \rVert_{L_{x}^{p}} \lVert \psi_{\zeta} \rVert_{L_{x}^{\infty}} \overset{\eqref{[Equations (4.4) and (4.5), LQ20]} \eqref{[Equation (4.23), LQ20]} \eqref{[Equation (73), Y20b]}}{\lesssim} \delta_{q+1}^{\frac{1}{2}} M_{0}(t)^{\frac{1}{2}} l^{-\frac{11}{2}} \sigma r^{2- \frac{2}{p}},  \label{[Equation (78b), Y20b]}\\
&\lVert w_{q+1}^{(t)} \rVert_{C_{t}L_{x}^{p}} \overset{\eqref{[Equation (74), Y20b]}}{\lesssim} \mu^{-1} \sum_{\zeta \in \Lambda} \lVert a_{\zeta} \rVert_{C_{t}L_{x}^{\infty}}^{2} \lVert \eta_{\zeta} \rVert_{C_{t}L_{x}^{2p}}^{2} 
\overset{\eqref{[Equation (4.23), LQ20]} \eqref{[Equation (73), Y20b]}}{\lesssim} \mu^{-1} \delta_{q+1} M_{0}(t) l^{-3} r^{2- \frac{2}{p}}. \label{[Equation (78c), Y20b]}
\end{align}
\end{subequations} 
The estimates \eqref{[Equation (78b), Y20b]}-\eqref{[Equation (78c), Y20b]} allow us to deduce for all $p \in (1,\infty)$ and $t \in [0, T_{L}]$ 
\begin{align}
&\lVert  w_{q+1}^{(c)} \rVert_{C_{t}L_{x}^{p}} + \lVert w_{q+1}^{(t)} \rVert_{C_{t}L_{x}^{p}} \nonumber\\
\overset{\eqref{[Equation (78b), Y20b]} \eqref{[Equation (78c), Y20b]}\eqref{[Equation (57), Y20b]} }{\lesssim}& \delta_{q+1}^{\frac{1}{2}} M_{0}(t)^{\frac{1}{2}} l^{-3} r^{2- \frac{2}{p}} [ \lambda_{q+1}^{5 \alpha + 2 \eta - 1}  + \lambda_{q+1}^{4 \eta - 1} \delta_{q+1}^{\frac{1}{2}} M_{0}(t)^{\frac{1}{2}}] \lesssim \delta_{q+1} M_{0}(t)   l^{-3} r^{2- \frac{2}{p}} \lambda_{q+1}^{4\eta -1} \label{[Equation (80), Y20b]}
\end{align} 
where the second inequality used that $5 \alpha + 2 \eta < 4 \eta - \beta$ due to \eqref{[Equation (2.3), LQ20]}, \eqref{alpha}, and \eqref{estimate 21}. We deduce from the estimate \eqref{[Equation (80), Y20b]} by taking $a \in 10 \mathbb{N}$ sufficiently large that for all $t \in [0, T_{L}]$ 
\begin{align}\label{[Equation (81), Y20b]}
 \lVert w_{q+1} \rVert_{C_{t}L_{x}^{2}}  \overset{\eqref{[Equation (76), Y20b]} \eqref{[Equation (77), Y20b]} \eqref{[Equation (80), Y20b]}}{\lesssim}& c_{R}^{\frac{1}{4}} \delta_{q+1}^{\frac{1}{2}} M_{0}(t)^{\frac{1}{2}} + \delta_{q+1} M_{0}(t) l^{-3} r \lambda_{q+1}^{4\eta -1} \nonumber\\
\overset{\eqref{estimate 22} \eqref{[Equation (57), Y20b]}}{\leq}& \delta_{q+1}^{\frac{1}{2}} M_{0}(t)^{\frac{1}{2}} [ \frac{3}{8} + C M_{0}(L)^{\frac{1}{2}} \lambda_{q+1}^{6\alpha - 2 \eta}] \leq \frac{3}{4} \delta_{q+1}^{\frac{1}{2}} M_{0}(t)^{\frac{1}{2}}  
\end{align} 
where the second inequality is by taking $c_{R} \ll 1$ and the third inequality used that $6 \alpha - 2 \eta < 0$ due to \eqref{[Equation (2.3), LQ20]}-\eqref{alpha}. We are now ready to verify \eqref{[Equation (40a), Y20b]} at level $q + 1$ and \eqref{[Equation (45), Y20b]} as follows:  
\begin{subequations}\label{estimate 53}
\begin{align}
&\lVert v_{q+1} \rVert_{C_{t}L_{x}^{2}} \overset{\eqref{[Equation (76), Y20b]}}{\leq} \lVert v_{l} \rVert_{C_{t}L_{x}^{2}} + \lVert w_{q+1} \rVert_{C_{t}L_{x}^{2}} \overset{\eqref{[Equation (61b), Y20b]} \eqref{[Equation (81), Y20b]}}{\leq} M_{0}(t)^{\frac{1}{2}} (1+ \sum_{1 \leq \iota \leq q+1} \delta_{\iota}^{\frac{1}{2}}),\\
& \lVert v_{q+1} (t) - v_{q}(t) \rVert_{L_{x}^{2}} \overset{\eqref{[Equation (76), Y20b]}}{\leq} \lVert w_{q+1}(t) \rVert_{L_{x}^{2}} + \lVert v_{l} (t) - v_{q}(t) \rVert_{L_{x}^{2}} \overset{\eqref{[Equation (61a), Y20b]} \eqref{[Equation (81), Y20b]}}{\leq} M_{0}(t)^{\frac{1}{2}} \delta_{q+1}^{\frac{1}{2}}. 
\end{align} 
\end{subequations}  
Next, we estimate norms of higher order. First, for all $t \in [0, T_{L}]$ 
\begin{subequations}\label{[Equation (82), Y20b]}
\begin{align}
 \lVert w_{q+1}^{(p)} \rVert_{C_{t,x}^{1}}  \overset{\eqref{[Equation (74), Y20b]}}{\lesssim}& \sum_{\zeta \in \Lambda} \lVert a_{\zeta} \rVert_{C_{t,x}^{1}} \lVert \mathbb{W}_{\zeta} \rVert_{L_{t}^{\infty} L_{x}^{\infty}} + \lVert a_{\zeta} \rVert_{L_{t}^{\infty} L_{x}^{\infty}} \lVert \mathbb{W}_{\zeta} \rVert_{C_{t,x}^{1}} \nonumber\\
\overset{\eqref{[Equation (4.22), LQ20]} \eqref{[Equation (73), Y20b]} \eqref{[Equation (57), Y20b]}}{\lesssim}&  \delta_{q+1}^{\frac{1}{2}} M_{0}(t)^{\frac{1}{2}} \lambda_{q+1}^{1-6\eta} l^{-\frac{3}{2}} [ \lambda_{q+1}^{8\alpha} + \lambda_{q+1}^{2-8\eta}] \lesssim \delta_{q+1}^{\frac{1}{2}} M_{0}(t)^{\frac{1}{2}} \lambda_{q+1}^{3-14\eta} l^{-\frac{3}{2}}, \label{[Equation (82a), Y20b]} \\
\lVert w_{q+1}^{(c)} \rVert_{C_{t,x}^{1}} \overset{\eqref{[Equation (74), Y20b]}}{\leq}& \sum_{\zeta \in \Lambda} \lVert  (\nabla^{\bot} a_{\zeta} \eta_{\zeta} + a_{\zeta} \nabla^{\bot} \eta_{\zeta}) \psi_{\zeta} \rVert_{C_{t,x}^{1}} \label{[Equation (82b), Y20b]} \\
\overset{\eqref{[Equations (4.4) and (4.5), LQ20]}  \eqref{[Equation (4.23), LQ20]} \eqref{[Equation (73), Y20b]}}{\lesssim}& \delta_{q+1}^{\frac{1}{2}} M_{0}(t)^{\frac{1}{2}} r [ l^{- \frac{19}{2}} \lambda_{q+1}^{-1} + l^{- \frac{11}{2}} \sigma \mu r + l^{-\frac{3}{2}} \lambda_{q+1} \sigma^{2} r^{2} \mu]  \lesssim \delta_{q+1}^{\frac{1}{2}} M_{0}(t)^{\frac{1}{2}} \lambda_{q+1}^{3- 18\eta} l^{- \frac{3}{2}}, \nonumber
\end{align}
\end{subequations}
where the last inequality in \eqref{[Equation (82a), Y20b]} used the fact that $8 \alpha < 2 - 8 \eta$ which can be verified by \eqref{[Equation (2.2), LQ20]}-\eqref{alpha}. Next, due to $\mathbb{P} \mathbb{P}_{\neq 0}$ not being bounded in $C_{t,x}^{1}$, we go down to $L^{p}$ space for $p \in (1, \infty)$ in the expense of $\lambda_{q+1}^{\alpha}$ and estimate for all $t \in [0, T_{L}]$  
\begin{align}
& \lVert w_{q+1}^{(t)} \rVert_{C_{t,x}^{1}} \overset{\eqref{[Equation (74), Y20b]}}{\lesssim} \mu^{-1} \sum_{\zeta \in \Lambda} \lambda_{q+1}^{\alpha} [ \lVert a_{\zeta} \rVert_{C_{t}C_{x}} \lVert a_{\zeta} \rVert_{C_{t}^{1}C_{x}} \lVert \eta_{\zeta} \rVert_{C_{t}C_{x}}^{2} + \lVert a_{\zeta} \rVert_{C_{t}C_{x}}^{2} \lVert \eta_{\zeta} \rVert_{C_{t}C_{x}} \lVert \eta_{\zeta} \rVert_{C_{t}^{1} C_{x}} \label{[Equation (83), Y20b]}\\
& \hspace{12mm} + \lVert a_{\zeta} \rVert_{C_{t}C_{x}} \lVert a_{\zeta} \rVert_{C_{t}C_{x}^{1}} \lVert \eta_{\zeta} \rVert_{C_{t}C_{x}}^{2} + \lVert a_{\zeta} \rVert_{C_{t}C_{x}}^{2} \lVert \eta_{\zeta} \rVert_{C_{t}C_{x}} \lVert \eta_{\zeta} \rVert_{C_{t}C_{x}^{1}}] \nonumber\\
&\hspace{8mm} \overset{\eqref{[Equation (4.23), LQ20]} \eqref{[Equation (57), Y20b]} \eqref{[Equation (73), Y20b]}}{\lesssim} \lambda_{q+1}^{4 \eta - 1} \lambda_{q+1}^{\alpha} \delta_{q+1} M_{0}(t) l^{-3} (\lambda_{q+1}^{1- 6 \eta})^{2} [\lambda_{q+1}^{8\alpha} + \lambda_{q+1}^{2- 8 \eta}]  \lesssim \lambda_{q+1}^{3 - 16 \eta + \alpha} \delta_{q+1} M_{0}(t) l^{-3} \nonumber 
\end{align}
where the last inequality used the fact that $8 \alpha < 2 - 8 \eta$ due to \eqref{[Equation (2.2), LQ20]}-\eqref{alpha}. 
Therefore, by taking $a \in 10 \mathbb{N}$ sufficiently large we conclude that \eqref{[Equation (40b), Y20b]} at level $q + 1$ holds as follows:
\begin{equation}\label{estimate 5}
\lVert v_{q+1} \rVert_{C_{t,x}^{1}} \overset{\eqref{[Equation (61c), Y20b]} \eqref{[Equation (82), Y20b]} \eqref{[Equation (83), Y20b]}}{\leq} M_{0} (t)^{\frac{1}{2}} [ l^{-1} \lambda_{q+1}^{-\alpha} + C \lambda_{q+1}^{3- 14 \eta} l^{- \frac{3}{2}} + C \lambda_{q+1}^{3- 16 \eta+ \alpha} M_{0}(t)^{\frac{1}{2}} l^{-3} ] 
\leq M_{0}(t)^{\frac{1}{2}} \lambda_{q+1}^{4} 
\end{equation} 
where the second inequality is due to 
\begin{subequations}\label{estimate 38}
\begin{align}
& l^{-1} \lambda_{q+1}^{-\alpha} \overset{\eqref{[Equation (57), Y20b]}}{\leq} \lambda_{q+1}^{\alpha} \overset{\eqref{alpha}}{\leq} \frac{1}{4} \lambda_{q+1}^{4}, \label{estimate 23}\\
& C\lambda_{q+1}^{3- 14\eta} l^{- \frac{3}{2}} \overset{\eqref{[Equation (57), Y20b]}}{\leq} C \lambda_{q+1}^{3 - 14 \eta} \lambda_{q+1}^{3\alpha} \overset{\eqref{[Equation (2.2), LQ20]} \eqref{[Equation (2.3), LQ20]} \eqref{alpha}}{\leq} \frac{1}{4} \lambda_{q+1}^{4}, \label{estimate 24}\\
& C \lambda_{q+1}^{3- 16 \eta + \alpha} M_{0}(t)^{\frac{1}{2}} l^{-3} \overset{\eqref{[Equation (57), Y20b]}}{\leq} C \lambda_{q+1}^{3- 16 \eta + 7\alpha} M_{0}(L)^{\frac{1}{2}} \overset{\eqref{[Equation (2.2), LQ20]} \eqref{[Equation (2.3), LQ20]} \eqref{alpha}}{\leq} \frac{1}{4} \lambda_{q+1}^{4}.\label{estimate 25}
\end{align}
\end{subequations} 
Finally, we estimate for all $p \in (1, \infty)$ and $t \in [0, T_{L}]$ 
\begin{subequations}\label{[Equation (88), Y20b]}
\begin{align}
& \lVert w_{q+1}^{(p)} + w_{q+1}^{(c)} \rVert_{C_{t}W_{x}^{1,p}} \overset{\eqref{[Equation (5.9), LQ20]}}{=} \lVert \nabla^{\bot} (\sum_{\zeta \in \Lambda} a_{\zeta} \eta_{\zeta} \psi_{\zeta}) \rVert_{C_{t}W_{x}^{1,p}}  \label{[Equation (88a), Y20b]}\\
& \overset{\eqref{[Equations (4.4) and (4.5), LQ20]} \eqref{[Equation (4.23), LQ20]} \eqref{[Equation (73), Y20b]} }{\lesssim} \delta_{q+1}^{\frac{1}{2}} M_{0}(t)^{\frac{1}{2}} r^{1- \frac{2}{p}} [ l^{- \frac{19}{2}} \lambda_{q+1}^{-1} + l^{- \frac{3}{2}} \sigma^{2} r^{2} \lambda_{q+1} + l^{- \frac{3}{2}} \lambda_{q+1}] \lesssim \delta_{q+1}^{\frac{1}{2}} M_{0}(t)^{\frac{1}{2}} r^{1- \frac{2}{p}} l^{- \frac{3}{2}} \lambda_{q+1},  \nonumber\\
 &\lVert w_{q+1}^{(t)} \rVert_{C_{t}W_{x}^{1,p}} \overset{ \eqref{[Equation (74), Y20b]}}{\lesssim} \mu^{-1} \sum_{\zeta \in \Lambda} \lVert a_{\zeta} \rVert_{C_{t}C_{x}} \lVert a_{\zeta} \rVert_{C_{t}C_{x}^{1}} \lVert \eta_{\zeta} \rVert_{C_{t}L_{x}^{2p}}^{2} + \lVert a_{\zeta} \rVert_{C_{t}C_{x}}^{2} \lVert \eta_{\zeta} \rVert_{C_{t}L_{x}^{2p}} \lVert \eta_{\zeta} \rVert_{C_{t}W_{x}^{1,2p}} \label{[Equation (88b), Y20b]} \\
 & \hspace{20mm} \overset{\eqref{[Equation (4.23), LQ20]}\eqref{[Equation (73), Y20b]} }{\lesssim} \mu^{-1} \delta_{q+1} M_{0}(t) l^{-3} r^{2- \frac{2}{p}} [ l^{-4} + \lambda_{q+1} \sigma r]  \overset{\eqref{[Equation (57), Y20b]}}{\lesssim} \mu^{-1} \delta_{q+1} M_{0}(t) l^{-3} r^{3- \frac{2}{p}} \lambda_{q+1} \sigma. \nonumber
\end{align}
\end{subequations} 

\subsubsection{Reynolds stress}
We can compute from \eqref{[Equation (36), Y20b]}, \eqref{[Equation (59), Y20b]}, and \eqref{[Equation (76), Y20b]} that 
\begin{align}
& \text{div} \mathring{R}_{q+1} - \nabla \pi_{q+1} \label{estimate 49}\\
=& \underbrace{(-\Delta)^{m} w_{q+1} + \partial_{t} (w_{q+1}^{(p)} + w_{q+1}^{(c)}) + \text{div} ((v_{l} + z_{l}) \otimes w_{q+1} + w_{q+1} \otimes (v_{l} + z_{l}))}_{\text{div} (R_{\text{lin}} ) + \nabla \pi_{\text{lin}}} \nonumber\\
&+ \underbrace{\text{div} (( w_{q+1}^{(c)} + w_{q+1}^{(t)}) \otimes w_{q+1} + w_{q+1}^{(p)} \otimes (w_{q+1}^{(c)} + w_{q+1}^{(t)}))}_{\text{div}(R_{\text{cor}}) + \nabla \pi_{\text{cor}}} + \underbrace{\text{div}(w_{q+1}^{(p)} \otimes w_{q+1}^{(p)} + \mathring{R}_{l}) + \partial_{t} w_{q+1}^{(t)}}_{\text{div}(R_{\text{osc}}) + \nabla \pi_{\text{osc}}} \nonumber \\
&+ \underbrace{\text{div} ( v_{q+1} \otimes z - v_{q+1} \otimes z_{l} + z \otimes v_{q+1} - z_{l} \otimes v_{q+1} + z \otimes z - z_{l} \otimes z_{l})}_{\text{div} (R_{\text{com2}} ) + \nabla \pi_{\text{com2}}} + \text{div} R_{\text{com1}} - \nabla \pi_{l} \nonumber 
\end{align} 
within which we specify 
\begin{subequations}\label{[Equation (91), Y20b]}
\begin{align}
R_{\text{lin}} \triangleq& R_{\text{linear}} \nonumber\\
\triangleq& \mathcal{R} ( -\Delta)^{m} w_{q+1}+ \mathcal{R} \partial_{t} (w_{q+1}^{(p)} + w_{q+1}^{(c)}) + (v_{l} + z_{l}) \mathring{\otimes} w_{q+1} + w_{q+1} \mathring{\otimes} (v_{l} + z_{l}), \label{[Equation (91a), Y20b]}\\
\pi_{\text{lin}} \triangleq& \pi_{\text{linear}} \triangleq (v_{l} + z_{l}) \cdot w_{q+1}, \label{[Equation (91b), Y20b]}\\
 R_{\text{cor}} \triangleq& R_{\text{corrector}} \triangleq (w_{q+1}^{(c)} + w_{q+1}^{(t)} ) \mathring{\otimes} w_{q+1} + w_{q+1}^{(p)} \mathring{\otimes} (w_{q+1}^{(c)} + w_{q+1}^{(t)}), \label{[Equation (91c), Y20b]}\\
 \pi_{\text{cor}} \triangleq & \pi_{\text{corrector}} \triangleq \frac{1}{2} [ (w_{q+1}^{(c)} + w_{q+1}^{(t)}) \cdot w_{q+1} + w_{q+1}^{(p)} \cdot (w_{q+1}^{(c)} + w_{q+1}^{(t)}) ], \label{[Equation (91d), Y20b]}\\
R_{\text{com2}} \triangleq&  R_{\text{commutator2}} \triangleq v_{q+1} \mathring{\otimes} (z - z_{l}) + (z - z_{l}) \mathring{\otimes} v_{q+1} + (z- z_{l}) \mathring{\otimes} z + z_{l} \mathring{\otimes} (z - z_{l}), \label{[Equation (91e), Y20b]} \\
\pi_{\text{com2}} \triangleq&  \pi_{\text{commutator2}} \triangleq v_{q+1} \cdot (z- z_{l}) + \frac{1}{2} \lvert z \rvert^{2} - \frac{1}{2} \lvert z_{l} \rvert^{2}.  \label{[Equation (91f), Y20b]}
\end{align}
\end{subequations}  
Concerning $R_{\text{osc}}$ that is arguably the most technical, first we can write 
\begin{equation}\label{estimate 6}
\text{div} ( w_{q+1}^{(p)} \otimes w_{q+1}^{(p)}) = \text{div} (w_{q+1}^{(p)} \mathring{\otimes} w_{q+1}^{(p)}) + \nabla \frac{1}{2} \lvert w_{q+1}^{(p)} \rvert^{2}, 
\end{equation} 
while 
\begin{equation}\label{[Equation (7.12a), LQ20]}
w_{q+1}^{(p)} \mathring{\otimes} w_{q+1}^{(p)} + \mathring{R}_{l} \overset{\eqref{estimate 3}\eqref{[Equation (74), Y20b]} }{=} \sum_{\zeta, \zeta' \in \Lambda} a_{\zeta} a_{\zeta'} \mathbb{P}_{\neq 0} (\mathbb{W}_{\zeta} \mathring{\otimes} \mathbb{W}_{\zeta'}) 
= \sum_{\zeta, \zeta' \in \Lambda} a_{\zeta} a_{\zeta'} \mathbb{P}_{\geq \frac{ \lambda_{q+1} \sigma}{2}} (\mathbb{W}_{\zeta} \mathring{\otimes} \mathbb{W}_{\zeta'}) 
\end{equation} 
because the minimal separation between active frequencies of $\mathbb{W}_{\zeta} \otimes \mathbb{W}_{\zeta'}$ and the zero frequency is given by $\lambda_{q+1}\sigma$ for $\zeta' = - \zeta$ and by $\frac{\lambda_{q+1}}{5} \geq \lambda_{q+1} \sigma$ for $\zeta' \neq - \zeta$ due to \eqref{[Equation (4.17), LQ20]}-\eqref{[Equation (4.18), LQ20]}  (cf. \cite[Equation (5.12)]{BV19a}). This leads to 
\begin{align}
& \text{div} (w_{q+1}^{(p)} \mathring{\otimes} w_{q+1}^{(p)} + \mathring{R}_{l}) \overset{\eqref{[Equation (7.12a), LQ20]}}{=} \mathbb{P}_{\neq 0} (\sum_{\zeta, \zeta' \in \Lambda} \nabla (a_{\zeta} a_{\zeta'}) \cdot \mathbb{P}_{\geq \frac{\lambda_{q+1} \sigma}{2}} (\mathbb{W}_{\zeta} \mathring{\otimes} \mathbb{W}_{\zeta'})  \label{[Equation (3.12b), LQ20]} \\
& \hspace{35mm} + a_{\zeta} a_{\zeta'} \nabla\cdot \mathbb{P}_{\geq \frac{\lambda_{q+1}\sigma}{2}} (\mathbb{W}_{\zeta} \mathring{\otimes} \mathbb{W}_{\zeta'} ))= \frac{1}{2} \sum_{\zeta, \zeta' \in \Lambda} \mathcal{E}_{\zeta, \zeta', 1} + \frac{1}{2} \sum_{\zeta, \zeta' \in \Lambda} \mathcal{E}_{\zeta, \zeta', 2}, \nonumber 
\end{align}
where 
\begin{subequations}
\begin{align}
& \mathcal{E}_{\zeta, \zeta', 1} \triangleq \mathbb{P}_{\neq 0} (\nabla (a_{\zeta} a_{\zeta'}) \cdot \mathbb{P}_{\geq \frac{\lambda_{q+1} \sigma}{2}} (\mathbb{W}_{\zeta} \mathring{\otimes} \mathbb{W}_{\zeta'} + \mathbb{W}_{\zeta'} \mathring{\otimes} \mathbb{W}_{\zeta})), \label{estimate 7}\\
& \mathcal{E}_{\zeta, \zeta', 2} \triangleq \mathbb{P}_{\neq 0} ( a_{\zeta} a_{\zeta'} \nabla\cdot (\mathbb{W}_{\zeta} \mathring{\otimes} \mathbb{W}_{\zeta'} + \mathbb{W}_{\zeta'} \mathring{\otimes} \mathbb{W}_{\zeta})), \label{estimate 8}
\end{align}
\end{subequations} 
in which we used symmetry, and also dropped the unnecessary frequency projection $\mathbb{P}_{\geq \frac{\lambda_{q+1} \sigma}{2}}$ in $\mathcal{E}_{\zeta, \zeta', 2}$. Now for any $\zeta, \zeta' \in \Lambda \subset \mathbb{S}^{1}$, we can compute
\begin{align}
&(\zeta^{\bot} \otimes \zeta'^{\bot} + \zeta'^{\bot} \otimes \zeta^{\bot}) (\zeta + \zeta') = 
\begin{pmatrix}
\zeta^{1}\zeta^{2}\zeta'^{2} + \zeta^{2}\zeta'^{1}\zeta'^{2}-(\zeta^{2})^{2} \zeta'^{1} - \zeta^{1} (\zeta'^{2})^{2} \\
- (\zeta^{1})^{2} \zeta'^{2} - \zeta^{2} (\zeta'^{1})^{2} + \zeta^{1} \zeta^{2} \zeta'^{1} + \zeta^{1} \zeta'^{1} \zeta'^{2} 
\end{pmatrix} \label{[Equation (7.13a), LQ20]}\\
& \hspace{10mm} = \begin{pmatrix}
\zeta^{1} [ \zeta^{2}\zeta'^{2} +(\zeta'^{1})^{2} - 1] + \zeta'^{1}[\zeta^{2}\zeta'^{2} + (\zeta^{1})^{2} - 1] \\
\zeta^{2} [(\zeta'^{2})^{2} + \zeta^{1}\zeta'^{1} - 1 ] + \zeta'^{2}[(\zeta^{2})^{2} + \zeta^{1}\zeta'^{1} - 1] 
\end{pmatrix} =  (\zeta^{\bot} \cdot \zeta'^{\bot} - 1)\text{Id} (\zeta + \zeta'). \nonumber 
\end{align}  
It follows that 
\begin{align}
\nabla\cdot (b_{\zeta} \mathring{\otimes} b_{\zeta'} + b_{\zeta'} \mathring{\otimes} b_{\zeta})(x) =& \nabla \cdot (b_{\zeta} \otimes b_{\zeta'} + b_{\zeta'} \otimes b_{\zeta} - b_{\zeta} \cdot b_{\zeta'} \text{Id} )(x) \nonumber\\
\overset{\eqref{[Equation (4.2), LQ20]} \eqref{[Equation (7.13a), LQ20]}}{=}& i\lambda_{q+1}e^{i\lambda_{q+1} (\zeta + \zeta') \cdot x} (\zeta + \zeta')  \overset{\eqref{[Equation (4.2), LQ20]}}{=} \nabla (\lambda_{q+1}^{2} \psi_{\zeta} \psi_{\zeta'})(x).  \label{[Equation (7.13b), LQ20]}
\end{align} 
Consequently,
\begin{equation}\label{[Equation (7.13c), LQ20]}
 \nabla\cdot (\mathbb{W}_{\zeta} \mathring{\otimes} \mathbb{W}_{\zeta'} + \mathbb{W}_{\zeta'} \mathring{\otimes} \mathbb{W}_{\zeta}) \overset{\eqref{[Equation (7.13b), LQ20]}}{=} ( b_{\zeta} \mathring{\otimes} b_{\zeta'} + b_{\zeta'} \mathring{\otimes} b_{\zeta}) \cdot \nabla (\eta_{\zeta} \eta_{\zeta'}) + (\eta_{\zeta} \eta_{\zeta'}) \nabla (\lambda_{q+1}^{2} \psi_{\zeta} \psi_{\zeta'}).
\end{equation}
After splitting $\frac{1}{2} \sum_{\zeta, \zeta' \in \Lambda} \mathcal{E}_{\zeta, \zeta', 2} = \frac{1}{2}( \sum_{\zeta, \zeta' \in \Lambda: \hspace{0.5mm}  \zeta + \zeta' \neq 0} + \sum_{\zeta, \zeta' \in \Lambda: \hspace{0.5mm} \zeta + \zeta' = 0}) \mathcal{E}_{\zeta, \zeta', 2}$, this allows us to write  
\begin{align}
& \frac{1}{2} \sum_{\zeta, \zeta' \in \Lambda: \hspace{0.5mm} \zeta + \zeta' \neq 0} \mathcal{E}_{\zeta, \zeta', 2} \label{[Equation (7.13d), LQ20]}\\
\overset{ \eqref{[Equation (4.17), LQ20]}\eqref{estimate 8} }{=}&  \frac{1}{2} \sum_{\zeta, \zeta' \in \Lambda: \hspace{0.5mm} \zeta + \zeta' \neq 0}  \mathbb{P}_{\neq 0} ( a_{\zeta} a_{\zeta'} \nabla\cdot \mathbb{P}_{\geq \frac{\lambda_{q+1}}{10}}(\mathbb{W}_{\zeta} \mathring{\otimes} \mathbb{W}_{\zeta'} + \mathbb{W}_{\zeta'} \mathring{\otimes} \mathbb{W}_{\zeta}))
\overset{\eqref{[Equation (7.13c), LQ20]}}{=}  \frac{1}{2} \sum_{\zeta, \zeta' \in \Lambda} \sum_{k=1}^{4} \mathcal{E}_{\zeta, \zeta', 2, k} \nonumber 
\end{align} 
where
\begin{subequations}\label{estimate 43}
\begin{align}
& \mathcal{E}_{\zeta, \zeta', 2, 1} \triangleq \mathbb{P}_{\neq 0} (a_{\zeta} a_{\zeta'} \mathbb{P}_{\geq \frac{\lambda_{q+1}}{10}} [(b_{\zeta} \mathring{\otimes} b_{\zeta'} + b_{\zeta'} \mathring{\otimes} b_{\zeta}) \cdot \nabla (\eta_{\zeta} \eta_{\zeta'})]1_{\zeta + \zeta' \neq 0}, \label{estimate 39} \\
& \mathcal{E}_{\zeta, \zeta', 2, 2} \triangleq  \nabla \mathbb{P}_{\neq 0} (a_{\zeta} a_{\zeta'} \mathbb{P}_{\geq \frac{\lambda_{q+1}}{10}} (\eta_{\zeta} \eta_{\zeta'} \lambda_{q+1}^{2} \psi_{\zeta} \psi_{\zeta'}))1_{\zeta + \zeta' \neq 0},  \label{estimate 40}\\
& \mathcal{E}_{\zeta, \zeta', 2, 3} \triangleq - \mathbb{P}_{\neq 0} (\nabla (a_{\zeta} a_{\zeta'} ) \mathbb{P}_{\geq \frac{\lambda_{q+1}}{10}} (\eta_{\zeta} \eta_{\zeta'} \lambda_{q+1}^{2} \psi_{\zeta} \psi_{\zeta'}) )1_{\zeta + \zeta' \neq 0}, \label{estimate 41} \\
& \mathcal{E}_{\zeta, \zeta', 2, 4} \triangleq  - \mathbb{P}_{\neq 0} (a_{\zeta} a_{\zeta'} \mathbb{P}_{\geq \frac{\lambda_{q+1}}{10}} (\nabla (\eta_{\zeta} \eta_{\zeta'}) \lambda_{q+1}^{2} \psi_{\zeta} \psi_{\zeta'}))1_{\zeta + \zeta' \neq 0}\label{estimate 42}
\end{align}
\end{subequations} 
(cf. \cite[pg. 131]{BV19a}). On the other hand, in case $\zeta + \zeta' = 0$ we have  $\nabla (\lambda_{q+1}^{2}\psi_{\zeta} \psi_{-\zeta}) \overset{\eqref{[Equation (4.2), LQ20]}}{=} 0$, while we can multiply \eqref{[Equation (4.12), LQ20]} by $2 \eta_{\zeta}$ to deduce $\mu^{-1} \partial_{t} \lvert \eta_{\zeta} \rvert^{2}= \pm (\zeta\cdot\nabla) \lvert \eta_{\zeta} \rvert^{2}$ for all $\zeta \in \Lambda^{\pm}$. Hence,  
\begin{align}
&\nabla\cdot ( \mathbb{W}_{\zeta} \mathring{\otimes} \mathbb{W}_{-\zeta} + \mathbb{W}_{-\zeta} \mathring{\otimes} \mathbb{W}_{\zeta}) \overset{\eqref{[Equation (4.15), LQ20]}\eqref{[Equation (7.13b), LQ20]}}{=}  [b_{\zeta} \mathring{\otimes} b_{-\zeta} + b_{-\zeta} \mathring{\otimes} b_{\zeta}] \nabla (\eta_{\zeta} \eta_{-\zeta})   \nonumber \\
\overset{\eqref{[Equation (4.2), LQ20]}}{=}& 2 \zeta^{\bot} \mathring{\otimes} \zeta^{\bot}\nabla \eta_{\zeta}^{2}    =  [ \text{Id} - 2 \zeta \otimes \zeta]\nabla \eta_{\zeta}^{2}  = \nabla \eta_{\zeta}^{2} - 2 (\zeta\cdot\nabla) \eta_{\zeta}^{2} \zeta = \nabla \eta_{\zeta}^{2} \mp 2 \mu^{-1} (\partial_{t} \eta_{\zeta}^{2}) \zeta. \label{[Equation (7.14b), LQ20]}
\end{align} 
This allows us to write 
\begin{align}
 \frac{1}{2} \sum_{\zeta, \zeta' \in \Lambda: \hspace{0.5mm} \zeta + \zeta' = 0} \mathcal{E}_{\zeta, \zeta', 2}  \overset{\eqref{[Equations (4.6), (4.7) and (4.8), LQ20]}\eqref{estimate 8} }{=}& \frac{1}{2} \sum_{\zeta \in \Lambda} \mathbb{P}_{\neq 0} (a_{\zeta}^{2} \nabla\cdot ( \mathbb{W}_{\zeta} \mathring{\otimes} \mathbb{W}_{-\zeta} + \mathbb{W}_{-\zeta} \mathring{\otimes} \mathbb{W}_{\zeta} )) \label{[Equation (7.14c), LQ20]}\\
 \overset{\eqref{[Equation (7.14b), LQ20]}}{=}& \frac{1}{2} \sum_{\zeta \in \Lambda} \nabla (a_{\zeta}^{2} \mathbb{P}_{\geq \frac{\lambda_{q+1} \sigma}{2}} \eta_{\zeta}^{2}) - \mathbb{P}_{\neq 0} (\nabla a_{\zeta}^{2} \mathbb{P}_{\geq \frac{\lambda_{q+1} \sigma}{2}} \eta_{\zeta}^{2}) \nonumber \\
& - \mu^{-1} (\sum_{\zeta \in \Lambda^{+}} - \sum_{\zeta \in \Lambda^{-}}) \partial_{t} \mathbb{P}_{\neq 0} (a_{\zeta}^{2} \mathbb{P}_{\neq 0} (\eta_{\zeta}^{2} \zeta) ) - \mathbb{P}_{\neq 0} (\partial_{t} a_{\zeta}^{2} \mathbb{P}_{\geq \frac{\lambda_{q+1} \sigma}{2}} (\eta_{\zeta}^{2} \zeta))  \nonumber 
\end{align}  
where we also used that $\eta_{\zeta}$ is $(\mathbb{T}/\lambda_{q+1} \sigma)^{2}$-periodic and hence $\mathbb{P}_{\geq \frac{\lambda_{q+1} \sigma}{2}} \eta_{\zeta}^{2} = \mathbb{P}_{\neq 0} \eta_{\zeta}^{2}$. At last, we obtain by using the definition of $\mathbb{P} = \text{Id} - \nabla \Delta^{-1} \nabla\cdot$
\begin{align}
& \frac{1}{2} \sum_{\zeta, \zeta' \in \Lambda: \hspace{0.5mm} \zeta + \zeta' = 0} \mathcal{E}_{\zeta, \zeta', 2} + \partial_{t} w_{q+1}^{(t)} \nonumber\\
\overset{\eqref{[Equation (74), Y20b]} \eqref{[Equation (7.14c), LQ20]}}{=}& \frac{1}{2} \sum_{\zeta \in \Lambda} \nabla (a_{\zeta}^{2} \mathbb{P}_{\geq \frac{\lambda_{q+1} \sigma}{2}} \eta_{\zeta}^{2}) - \mathbb{P}_{\neq 0} (\nabla a_{\zeta}^{2} \mathbb{P}_{\geq \frac{\lambda_{q+1} \sigma}{2}} \eta_{\zeta}^{2}) \nonumber\\
& - \mu^{-1} ( \sum_{\zeta \in \Lambda^{+}} - \sum_{\zeta \in \Lambda^{-}}) \partial_{t} \mathbb{P}_{\neq 0} (a_{\zeta}^{2} \mathbb{P}_{\neq 0} (\eta_{\zeta}^{2} \zeta)) - \mathbb{P}_{\neq 0} (\partial_{t} a_{\zeta}^{2} \mathbb{P}_{\geq \frac{\lambda_{q+1} \sigma}{2}} (\eta_{\zeta}^{2} \zeta)) \nonumber \\
&+ \mu^{-1} (\sum_{\zeta \in \Lambda^{+}} - \sum_{\zeta \in \Lambda^{-}}) (\text{Id} - \nabla \Delta^{-1} \nabla\cdot) \partial_{t} \mathbb{P}_{\neq 0} (a_{\zeta}^{2} \mathbb{P}_{\neq 0} \eta_{\zeta}^{2} \zeta) = \sum_{k=1}^{4}A_{k}\label{[Equation (7.14e), LQ20]}
\end{align}
where 
\begin{subequations}\label{estimate 48}
\begin{align}
& A_{1} \triangleq \frac{1}{2} \sum_{\zeta \in \Lambda} \nabla (a_{\zeta}^{2} \mathbb{P}_{\geq \frac{\lambda_{q+1}\sigma}{2}} \eta_{\zeta}^{2}), \label{estimate 44} \\
& A_{2} \triangleq - \frac{1}{2} \sum_{\zeta \in \Lambda} \mathbb{P}_{\neq 0} (\nabla a_{\zeta}^{2} \mathbb{P}_{\geq \frac{\lambda_{q+1} \sigma}{2}} \eta_{\zeta}^{2}),  \label{estimate 45}\\
& A_{3} \triangleq \mu^{-1} (\sum_{\zeta \in \Lambda^{+}} - \sum_{\zeta \in \Lambda^{-}}) \mathbb{P}_{\neq 0} (\partial_{t} a_{\zeta}^{2} \mathbb{P}_{\geq \frac{\lambda_{q+1} \sigma}{2}} (\eta_{\zeta}^{2} \zeta)),\label{estimate 46} \\
& A_{4} \triangleq - \nabla \Delta^{-1} \nabla\cdot \mu^{-1} (\sum_{\zeta \in \Lambda^{+}} - \sum_{\zeta \in \Lambda^{-}}) \mathbb{P}_{\neq 0} \partial_{t} (a_{\zeta}^{2} \mathbb{P}_{\neq 0} \eta_{\zeta}^{2} \zeta). \label{estimate 47}
\end{align}
\end{subequations}  
Therefore,
\begin{align}
& \text{div} (w_{q+1}^{(p)} \otimes w_{q+1}^{(p)} + \mathring{R}_{l}) + \partial_{t}w_{q+1}^{(t)} \nonumber \\
\overset{\eqref{estimate 6} \eqref{[Equation (3.12b), LQ20]}}{=}& \frac{1}{2} \sum_{\zeta, \zeta' \in \Lambda} \mathcal{E}_{\zeta, \zeta', 1} + \frac{1}{2} \sum_{\zeta, \zeta' \in \Lambda} \mathcal{E}_{\zeta, \zeta', 2} + \partial_{t} w_{q+1}^{(t)} + \nabla \frac{1}{2} \lvert w_{q+1}^{(p)}\rvert^{2} \nonumber \\
\overset{\eqref{[Equation (7.13d), LQ20]}\eqref{[Equation (7.14e), LQ20]} }{=}& \frac{1}{2} \sum_{\zeta, \zeta' \in \Lambda} \mathcal{E}_{\zeta, \zeta', 1} + \frac{1}{2} \sum_{\zeta, \zeta' \in \Lambda} \sum_{k= 1, 3, 4} \mathcal{E}_{\zeta, \zeta', 2, k} + A_{2} + A_{3} \nonumber \\
&+ \nabla [ \frac{1}{2} \lvert w_{q+1}^{(p)} \rvert^{2} + \frac{1}{2} \sum_{\zeta, \zeta' \in \Lambda} \mathbb{P}_{\neq 0} (a_{\zeta} a_{\zeta'} \mathbb{P}_{\geq \frac{\lambda_{q+1}}{10}} (\eta_{\zeta} \eta_{\zeta'} \lambda_{q+1}^{2} \psi_{\zeta} \psi_{\zeta'})) \nonumber \\
&  + \frac{1}{2} \sum_{\zeta \in \Lambda} a_{\zeta}^{2} \mathbb{P}_{\geq \frac{\lambda_{q+1} \sigma}{2}} \eta_{\zeta}^{2} - \Delta^{-1} \nabla\cdot \mu^{-1} (\sum_{\zeta \in \Lambda^{+}} - \sum_{\zeta \in \Lambda^{-}}) \mathbb{P}_{\neq 0} \partial_{t} (a_{\zeta}^{2} \mathbb{P}_{\neq 0} \eta_{\zeta}^{2} \zeta )],   \label{estimate 26}
\end{align} 
which finally leads us to define 
\begin{subequations}
\begin{align}
R_{\text{osc}} \triangleq& R_{\text{oscillation}} \triangleq \mathcal{R} (\frac{1}{2} \sum_{\zeta, \zeta' \in \Lambda} \mathcal{E}_{\zeta, \zeta', 1} + \frac{1}{2} \sum_{\zeta, \zeta' \in \Lambda} \sum_{k= 1, 3, 4} \mathcal{E}_{\zeta, \zeta', 2, k} + A_{2} + A_{3} ), \label{estimate 27}\\
\pi_{\text{osc}} \triangleq& \pi_{\text{oscillation}} \triangleq \frac{1}{2} \lvert w_{q+1}^{(p)} \rvert^{2} + \frac{1}{2} \sum_{\zeta, \zeta' \in \Lambda} \mathbb{P}_{\neq 0} (a_{\zeta} a_{\zeta'} \mathbb{P}_{\geq \frac{\lambda_{q+1} \sigma}{2}} (\eta_{\zeta} \eta_{\zeta'} \lambda_{q+1}^{2} \psi_{\zeta} \psi_{\zeta'}))1_{\zeta + \zeta' \neq 0}   \label{estimate 28}\\
& \hspace{12mm} + \frac{1}{2} \sum_{\zeta \in \Lambda} a_{\zeta}^{2} \mathbb{P}_{\geq \frac{\lambda_{q+1} \sigma}{2}} \eta_{\zeta}^{2} - \Delta^{-1} \nabla\cdot \mu^{-1} (\sum_{\zeta \in \Lambda^{+}} - \sum_{\zeta \in \Lambda^{-}}) \mathbb{P}_{\neq 0} \partial_{t} (a_{\zeta}^{2} \mathbb{P}_{\neq 0} \eta_{\zeta}^{2} \zeta ). \nonumber 
\end{align} 
\end{subequations} 
Considering \eqref{estimate 49} we define  
\begin{equation}\label{[Equation (92), Y20b]}
\pi_{q+1} \triangleq \pi_{l} - \pi_{\text{lin}} - \pi_{\text{cor}} - \pi_{\text{osc}} - \pi_{\text{com2}} \text{ and } \mathring{R}_{q+1} \triangleq R_{\text{lin}} + R_{\text{cor}} + R_{\text{osc}} + R_{\text{com2}} + R_{\text{com1}}. 
\end{equation} 

Now we choose 
\begin{equation}\label{p ast}
p^{\ast} \triangleq \frac{16 (1- 6\eta)}{300 \alpha + 16 (1-7\eta)}, 
\end{equation} 
which can be readily verified to be an element in $(1,2)$ using \eqref{[Equation (2.2), LQ20]}-\eqref{alpha}. For $R_{\text{lin}}$ we first estimate by Gagliardo-Nirenberg's inequality for all $t \in [0, T_{L}]$ 
\begin{align}
 \lVert \mathcal{R} (-\Delta)^{m} w_{q+1} \rVert_{C_{t}L_{x}^{p^{\ast}}} \lesssim& \lVert w_{q+1} \rVert_{C_{t}L_{x}^{p^{\ast}}}^{1-m^{\ast}} (\lVert \nabla (w_{q+1}^{(p)} + w_{q+1}^{(c)}) \rVert_{C_{t}L_{x}^{p^{\ast}}} + \lVert \nabla w_{q+1}^{(t)} \rVert_{C_{t}L_{x}^{p^{\ast}}})^{m^{\ast}} \nonumber\\
&   \overset{\eqref{[Equation (76), Y20b]} \eqref{[Equation (78a), Y20b]} \eqref{[Equation (80), Y20b]}\eqref{[Equation (88), Y20b]} }{\lesssim} \delta_{q+1}^{\frac{1}{2}} M_{0}(t)^{\frac{1}{2}} r^{1- \frac{2}{p^{\ast}}} (l^{-\frac{3}{2}} + \delta_{q+1}^{\frac{1}{2}} M_{0}(t)^{\frac{1}{2}} l^{-3} r \lambda_{q+1}^{4 \eta -1} )^{1- m^{\ast}} \nonumber\\
& \hspace{22mm} \times (l^{-\frac{3}{2}} \lambda_{q+1} + \mu^{-1} \delta_{q+1}^{\frac{1}{2}} M_{0}(t)^{\frac{1}{2}} l^{-3} r^{2} \lambda_{q+1} \sigma)^{m^{\ast}}. \label{estimate 9}
\end{align} 
Second, for all $t \in [0, T_{L}]$  
\begin{align}
 \lVert \mathcal{R} \partial_{t} (w_{q+1}^{(p)} + w_{q+1}^{(c)}) & \rVert_{C_{t}L_{x}^{p^{\ast}}} 
\overset{\eqref{[Equation (5.9), LQ20]}}{\lesssim}  \sum_{\zeta \in \Lambda} \lVert \partial_{t}(a_{\zeta} \eta_{\zeta} ) \psi_{\zeta} \rVert_{C_{t} L_{x}^{p^{\ast}}} \nonumber \\
\overset{\eqref{[Equations (4.4) and (4.5), LQ20]} \eqref{[Equation (4.23), LQ20]} \eqref{[Equation (73), Y20b]}}{\lesssim}&  \delta_{q+1}^{\frac{1}{2}} M_{0}(t)^{\frac{1}{2}} r^{1- \frac{2}{p^{\ast}}} [ l^{-4} \lambda_{q+1}^{-1} + l^{- \frac{3}{2}} \sigma \mu r] \lesssim \delta_{q+1}^{\frac{1}{2}} M_{0}(t)^{\frac{1}{2}} r^{1- \frac{2}{p^{\ast}}} l^{- \frac{3}{2}} \lambda_{q+1}^{1- 8 \eta}.  \label{estimate 10}
\end{align} 
Finally, we can estimate for all $t \in [0, T_{L}]$ 
\begin{align}
& \lVert ( v_{l} + z_{l}) \mathring{\otimes} w_{q+1} + w_{q+1} \mathring{\otimes} (v_{l} + z_{l}) \rVert_{C_{t}L_{x}^{p^{\ast}}}  \lesssim (\lVert v_{q} \rVert_{C_{t}C_{x}} + \lVert z  \rVert_{C_{t}C_{x}}) \lVert w_{q+1} \rVert_{C_{t}L_{x}^{p^{\ast}}} \nonumber\\
& \hspace{15mm} \overset{\eqref{[Equation (38), Y20b]} \eqref{[Equation (40b), Y20b]} \eqref{[Equation (78a), Y20b]} \eqref{[Equation (80), Y20b]}}{\lesssim}  M_{0}(t)^{\frac{1}{2}} \lambda_{q}^{4} [ \delta_{q+1}^{\frac{1}{2}} M_{0}(t)^{\frac{1}{2}} l^{-\frac{3}{2}} r^{1- \frac{2}{p^{\ast}}} + \delta_{q+1} M_{0}(t) l^{-3} r^{2- \frac{2}{p^{\ast}}} \lambda_{q+1}^{4 \eta - 1}].  \label{estimate 11}
\end{align} 
Due to \eqref{estimate 9}-\eqref{estimate 11} we obtain for all $t \in [0, T_{L}]$  
\begin{align}
\lVert R_{\text{lin}} \rVert_{C_{t}L_{x}^{p^{\ast}}} \nonumber
\overset{\eqref{[Equation (91a), Y20b]}}{\leq}& \lVert \mathcal{R} (-\Delta)^{m} w_{q+1} \rVert_{C_{t}L_{x}^{p^{\ast}}} + \lVert \mathcal{R} \partial_{t} (w_{q+1}^{(p)} + w_{q+1}^{(c)}) \rVert_{C_{t}L_{x}^{p^{\ast}}} \nonumber\\
& + \lVert (v_{l} + z_{l}) \mathring{\otimes} w_{q+1} + w_{q+1} \mathring{\otimes} (v_{l} + z_{l}) \rVert_{C_{t}L_{x}^{p^{\ast}}}  \nonumber\\
\overset{\eqref{estimate 9}\eqref{estimate 10}\eqref{estimate 11}}{\lesssim}& \delta_{q+1}^{\frac{1}{2}} M_{0}(t)^{\frac{1}{2}} r^{1- \frac{2}{p^{\ast}}}( l^{-\frac{3}{2}} + \delta_{q+1}^{\frac{1}{2}} M_{0}(t)^{\frac{1}{2}}l^{-3} r \lambda_{q+1}^{4 \eta -1} )^{1 - m^{\ast}}\nonumber  \\
& \hspace{20mm}  \times (l^{- \frac{3}{2}} \lambda_{q+1} + \mu^{-1} \delta_{q+1}^{\frac{1}{2}} M_{0}(t)^{\frac{1}{2}} l^{-3} r^{2} \lambda_{q+1} \sigma )^{m^{\ast}}  \nonumber  \\
&+ \delta_{q+1}^{\frac{1}{2}} M_{0}(t)^{\frac{1}{2}} r^{1- \frac{2}{p^{\ast}}} l^{- \frac{3}{2}} \lambda_{q+1}^{1- 8 \eta}  \nonumber\\
&+ M_{0}(t)^{\frac{1}{2}} \lambda_{q}^{4} [ \delta_{q+1}^{\frac{1}{2}} M_{0}(t)^{\frac{1}{2}} l^{-\frac{3}{2}} r^{1- \frac{2}{p^{\ast}}} + \delta_{q+1} M_{0}(t) l^{-3} r^{2- \frac{2}{p^{\ast}}} \lambda_{q+1}^{4 \eta - 1}]  \nonumber  \\
\lesssim& M_{0}(t)^{\frac{1}{2}} r^{1 - \frac{2}{p^{\ast}}} l^{-\frac{3}{2}} \lambda_{q+1}^{m^{\ast}} + M_{0}(t)^{\frac{1}{2}}r^{1- \frac{2}{p^{\ast}}} l^{-\frac{3}{2}} \lambda_{q+1}^{1- 8 \eta}  + M_{0}(t) r^{1- \frac{2}{p^{\ast}}} l^{-\frac{3}{2}} \lambda_{q}^{4}. \label{estimate 12}
\end{align}  
Now within the right hand side of \eqref{estimate 12}, first we can estimate using $2 \beta b < \frac{\alpha}{8}$ from \eqref{estimate 21} and taking $a \in 10 \mathbb{N}$ sufficiently large 
\begin{align}
& M_{0}(t)^{\frac{1}{2}} r^{1- \frac{2}{p^{\ast}}} l^{- \frac{3}{2}} \lambda_{q+1}^{m^{\ast}} \nonumber\\
=&
\begin{cases}
M_{0}(t)^{\frac{1}{2}} r^{1- \frac{2}{p^{\ast}}} l^{-\frac{3}{2}} \overset{\eqref{[Equation (57), Y20b]}}{\lesssim} M_{0}(t) \delta_{q+2}   \lambda_{q+2}^{2\beta} \lambda_{q+1}^{(1- 6 \eta)(1- \frac{2}{p^{\ast}})} \lambda_{q+1}^{3\alpha} & \text{ if } m \in (0, \frac{1}{2}), \\
M_{0}(t)^{\frac{1}{2}} r^{1- \frac{2}{p^{\ast}}} l^{- \frac{3}{2}} \lambda_{q+1}^{2m-1} \overset{\eqref{[Equation (57), Y20b]}}{\lesssim} M_{0}(t) \delta_{q+2}  \lambda_{q+2}^{2 \beta}  \lambda_{q+1}^{(1-6\eta)(1 - \frac{2}{p^{\ast}})} \lambda_{q+1}^{3\alpha} \lambda_{q+1}^{2m-1}    &   \text{ if } m \in [\frac{1}{2}, 1),  
\end{cases} \nonumber\\
\overset{\eqref{p ast}}{\lesssim}& M_{0}(t) \delta_{q+2} \lambda_{q+1}^{- \frac{275 \alpha}{8}} \ll (2\pi)^{-2 ( \frac{p^{\ast} - 1}{p^{\ast}})} \frac{M_{0}(t) c_{R} \delta_{q+2}}{15}.\label{estimate 13}
\end{align} 
Second within \eqref{estimate 12} we estimate using $2 \beta b < \frac{\alpha}{8}$ from \eqref{estimate 21} and taking $a \in 10 \mathbb{N}$ sufficiently large 
\begin{align}
M_{0}(t)^{\frac{1}{2}} r^{1- \frac{2}{p^{\ast}}} l^{-\frac{3}{2}} \lambda_{q+1}^{1- 8 \eta} \overset{\eqref{[Equation (57), Y20b]}}{\lesssim}& M_{0}(t) \delta_{q+2} \lambda_{q+1}^{\frac{\alpha}{8}} (\lambda_{q+1}^{1- 6 \eta})^{1- \frac{2}{p^{\ast}}} \lambda_{q+1}^{3\alpha} \lambda_{q+1}^{1- 8 \eta}  \nonumber \\
\overset{\eqref{p ast}}{\approx}& M_{0}(t) \delta_{q+2} \lambda_{q+1}^{- \frac{275 \alpha}{8}}   \ll (2\pi)^{-2 ( \frac{p^{\ast} - 1}{p^{\ast}})} \frac{M_{0}(t) c_{R} \delta_{q+2}}{15}.  \label{estimate 14}
\end{align}
Third within \eqref{estimate 12} we estimate also using $2 \beta b < \frac{\alpha}{8}$ from \eqref{estimate 21} and taking $a \in 10 \mathbb{N}$ sufficiently large 
\begin{align}
M_{0}(t) r^{1- \frac{2}{p^{\ast}}} l^{-\frac{3}{2}} \lambda_{q}^{4} \overset{\eqref{[Equation (57), Y20b]}}{\lesssim}&  M_{0}(t) \delta_{q+2} \lambda_{q+1}^{\frac{\alpha}{8} + 4\alpha} (\lambda_{q+1}^{1- 6 \eta})^{1- \frac{2}{p^{\ast}}}  \nonumber\\
\overset{\eqref{p ast}}{\lesssim}& M_{0}(t) \delta_{q+2} \lambda_{q+1}^{- \frac{267 \alpha}{8}}  \ll (2\pi)^{-2 ( \frac{p^{\ast} - 1}{p^{\ast}})} \frac{M_{0}(t) c_{R} \delta_{q+2}}{15}.  \label{estimate 15}
\end{align}
By applying \eqref{estimate 13}-\eqref{estimate 15} to \eqref{estimate 12}, we obtain 
\begin{equation}\label{[Equation (98), Y20b]}
\lVert R_{\text{lin}} \rVert_{C_{t}L_{x}^{p^{\ast}}} \leq (2\pi)^{-2 (\frac{p^{\ast} - 1}{p^{\ast}})} \frac{M_{0}(t) c_{R} \delta_{q+2}}{5}. 
\end{equation} 

Next, for all $t \in [0, T_{L}]$ we estimate by H$\ddot{\mathrm{o}}$lder's inequality, utilizing $2 \beta b < \frac{\alpha}{8}$ due to \eqref{estimate 21}, and taking $a \in 10 \mathbb{N}$ sufficiently large, 
\begin{align} 
 \lVert R_{\text{cor}} \rVert_{C_{t}L_{x}^{p^{\ast}}} 
\overset{\eqref{[Equation (76), Y20b]}\eqref{[Equation (91c), Y20b]}  }{\lesssim}& ( \lVert w_{q+1}^{(c)} \rVert_{C_{t}L_{x}^{2p^{\ast}}} + \lVert w_{q+1}^{(t)} \rVert_{C_{t}L_{x}^{2p^{\ast}}}) ( \lVert w_{q+1}^{(c)} \rVert_{C_{t}L_{x}^{2p^{\ast}}} + \lVert w_{q+1}^{(t)} \rVert_{C_{t}L_{x}^{2p^{\ast}}} + \lVert w_{q+1}^{(p)} \rVert_{C_{t}L_{x}^{2p^{\ast}}}) \nonumber\\
\overset{\eqref{[Equation (57), Y20b]} \eqref{[Equation (78), Y20b]} \eqref{[Equation (80), Y20b]} }{\lesssim}& [ M_{0}(t)^{\frac{1}{2}} r^{2- \frac{1}{p^{\ast}}} l^{-3} ( \lambda_{q+1}^{5\alpha} \lambda_{q+1}^{2 \eta -1} + M_{0}(t)^{\frac{1}{2}} \lambda_{q+1}^{4 \eta -1} )]\nonumber\\
& \times [ M_{0}(t)^{\frac{1}{2}} r^{1- \frac{1}{p^{\ast}}} l^{- \frac{3}{2}} (\lambda_{q+1}^{- 2 \eta} \lambda_{q+1}^{3\alpha} M_{0}(t)^{\frac{1}{2}} + 1 )] \nonumber\\
\overset{\eqref{p ast}}{\lesssim}& \delta_{q+2} M_{0}(t) \lambda_{q+1}^{-\frac{227 \alpha}{8}} M_{0}(t)^{\frac{1}{2}} \leq (2\pi)^{-2 (\frac{p^{\ast} -1}{p^{\ast}})} \frac{M_{0} (t) c_{R}\delta_{q+2}  }{5}. \label{[Equation (100), Y20b]}
\end{align} 

Next, we estimate $R_{\text{oscillation}}$ from \eqref{estimate 27}. First, we rely on Lemma \ref{[Lemma 7.4, LQ20]}, use that $2 \beta b < \frac{\alpha}{8}$ due to \eqref{estimate 21}, and take $a \in 10 \mathbb{N}$ sufficiently large to deduce for all $t \in [0, T_{L}]$  
\begin{align} 
& \lVert \mathcal{R} ( \frac{1}{2} \sum_{\zeta, \zeta' \in \Lambda} \mathcal{E}_{\zeta, \zeta', 1} ) \rVert_{C_{t}L_{x}^{p^{\ast}}} 
\overset{\eqref{estimate 7}}{\lesssim} ( \frac{\lambda_{q+1} \sigma}{2})^{-1} \sum_{\zeta,\zeta' \in \Lambda} \lVert \nabla (a_{\zeta} a_{\zeta'} ) \rVert_{C_{t}C_{x}^{2}} \lVert \mathbb{W}_{\zeta} \mathring{\otimes} \mathbb{W}_{\zeta'} + \mathbb{W}_{\zeta'} \mathring{\otimes} \mathbb{W}_{\zeta} \rVert_{C_{t}L_{x}^{p^{\ast}}} \nonumber\\
&  \overset{\eqref{[Equation (4.22), LQ20]} \eqref{[Equation (73), Y20b]} }{\lesssim} \lambda_{q+1}^{-2\eta} \delta_{q+1} M_{0}(t) l^{-15} r^{2- \frac{2}{p^{\ast}}}  
\overset{\eqref{[Equation (57), Y20b]} \eqref{p ast}}{\lesssim} \delta_{q+2} M_{0}(t) \lambda_{q+1}^{- \frac{59\alpha}{8}}   \leq (2\pi)^{-2 (\frac{p^{\ast} - 1}{p^{\ast}})} \frac{c_{R} \delta_{q+2} M_{0}(t)}{25}.  \label{estimate 16}
\end{align} 
Here the hypothesis of Lemma \ref{[Lemma 7.4, LQ20]} requires that $\frac{\lambda_{q+1} \sigma}{2} \in \mathbb{N}$ which is satisfied because $\lambda_{q+1} \sigma \in 10 \mathbb{N}$ by our choice; we also clearly see that $\lambda_{q+1} \sigma \in 5 \mathbb{N}$ would not have been sufficient for this purpose. Similarly to \eqref{estimate 16}, relying on Lemma \ref{[Lemma 7.4, LQ20]} we can estimate for all $t \in [0, T_{L}]$  
\begin{align}
&\lVert \mathcal{R} (\frac{1}{2} \sum_{\zeta, \zeta' \in \Lambda} \mathcal{E}_{\zeta, \zeta', 2, 3}) \rVert_{C_{t}L_{x}^{p^{\ast}}} \overset{\eqref{estimate 41}}{\lesssim} \sum_{\zeta, \zeta' \in \Lambda} (\frac{\lambda_{q+1}}{10})^{-1} \lVert \nabla (a_{\zeta} a_{\zeta'} ) \rVert_{C_{t}C_{x}^{2}} \lVert \eta_{\zeta} \eta_{\zeta'} \lambda_{q+1}^{2} \psi_{\zeta} \psi_{\zeta'} \rVert_{C_{t}L_{x}^{p^{\ast}}}  \label{estimate 19} \\ 
& \hspace{6mm} \overset{\eqref{[Equations (4.4) and (4.5), LQ20]} \eqref{[Equation (4.23), LQ20]} \eqref{[Equation (73), Y20b]}}{\lesssim}  \lambda_{q+1}^{-1} \delta_{q+1} M_{0}(t) l^{-15} r^{2- \frac{2}{p^{\ast}}} \lesssim \delta_{q+2} M_{0}(t) \lambda_{q+1}^{- \frac{59\alpha}{8} - 1 + 2 \eta }   \leq (2\pi)^{-2 (\frac{p^{\ast} - 1}{p^{\ast}})} \frac{c_{R} \delta_{q+2} M_{0}(t)}{25}.  \nonumber 
\end{align}
Here the hypothesis of Lemma \ref{[Lemma 7.4, LQ20]} requires $\frac{\lambda_{q+1}}{10} \in \mathbb{N}$ and thus $\lambda_{q+1} \in 10\mathbb{N}$ instead of $\lambda_{q+1} \in 5 \mathbb{N}$ was needed. Next, for all $t \in [0, T_{L}]$ we estimate also relying on Lemma \ref{[Lemma 7.4, LQ20]}, using that $2 \beta b < \frac{\alpha}{8}$ due to \eqref{estimate 21}, and taking $a \in 10 \mathbb{N}$ sufficiently large, 
\begin{align}
& \lVert \mathcal{R}(\frac{1}{2} \sum_{\zeta, \zeta' \in \Lambda} \mathcal{E}_{\zeta, \zeta', 2, 1}) \rVert_{C_{t}L_{x}^{p^{\ast}}} 
\overset{\eqref{estimate 39}}{\lesssim} \sum_{\zeta, \zeta' \in \Lambda} (\frac{\lambda_{q+1}}{10})^{-1} \lVert a_{\zeta} a_{\zeta'} \rVert_{C_{t}C_{x}^{2}} \lVert (b_{\zeta} \mathring{\otimes} b_{\zeta'} + b_{\zeta'} \mathring{\otimes} b_{\zeta}) \cdot \nabla (\eta_{\zeta} \eta_{\zeta'}) \rVert_{C_{t}L_{x}^{p^{\ast}}} \nonumber\\
&\overset{\eqref{[Equations (4.4) and (4.5), LQ20]} \eqref{[Equation (4.23), LQ20]} \eqref{[Equation (73), Y20b]}}{\lesssim} M_{0}(t) l^{-11} \lambda_{q+1}^{-4\eta} r^{2- \frac{2}{p^{\ast}}} \overset{\eqref{[Equation (57), Y20b]}}{\lesssim} \delta_{q+2} M_{0}(t) \lambda_{q+1}^{-\frac{123 \alpha}{8} - 2 \eta} \leq (2\pi)^{-2 (\frac{p^{\ast} - 1}{p^{\ast}})} \frac{c_{R} \delta_{q+2} M_{0}(t)}{25}.\label{estimate 17} 
\end{align} 
Next, relying also on Lemma \ref{[Lemma 7.4, LQ20]} we can estimate for all $t \in [0, T_{L}]$ similarly to \eqref{estimate 17}
\begin{align}
 \lVert \mathcal{R} (\frac{1}{2} \sum_{\zeta, \zeta' \in \Lambda}  \mathcal{E}_{\zeta, \zeta', 2, 4}) \rVert_{C_{t}L_{x}^{p^{\ast}}} 
\overset{\eqref{estimate 42}}{\lesssim}& \sum_{\zeta, \zeta' \in \Lambda} (\frac{\lambda_{q+1}}{10})^{-1} \lVert a_{\zeta} a_{\zeta'} \rVert_{C_{t}C_{x}^{2}} \lVert \nabla (\eta_{\zeta} \eta_{\zeta'}) \lambda_{q+1}^{2} \psi_{\zeta} \psi_{\zeta'} \rVert_{C_{t}L_{x}^{p^{\ast}}} \label{estimate 18}\\
& \overset{\eqref{[Equations (4.4) and (4.5), LQ20]} \eqref{[Equation (4.23), LQ20]} \eqref{[Equation (73), Y20b]}}{\lesssim} M_{0}(t) l^{-11} \lambda_{q+1}^{-4 \eta} r^{2- \frac{2}{p^{\ast}}} \overset{\eqref{[Equation (57), Y20b]}}{\leq} (2\pi)^{-2 (\frac{p^{\ast} - 1}{p^{\ast}})} \frac{c_{R} \delta_{q+2} M_{0}(t)}{25}. \nonumber 
\end{align}  
Next, we estimate for all $t \in [0, T_{L}]$ by applying Lemma \ref{[Lemma 7.4, LQ20]}, using that $2 \beta b < \frac{\alpha}{8}$ due to \eqref{estimate 21}, and taking $a \in 10 \mathbb{N}$ sufficiently large 
\begin{align}
& \lVert \mathcal{R} (A_{2} +A_{3}) \rVert_{C_{t}L_{x}^{p^{\ast}}} \overset{\eqref{estimate 48}}{\lesssim}  (\frac{\lambda_{q+1} \sigma}{2})^{-1}  \sum_{\zeta \in \Lambda} \lVert \nabla a_{\zeta}^{2} \rVert_{C_{t}C_{x}^{2}} \lVert \eta_{\zeta}^{2} \rVert_{C_{t}L_{x}^{p^{\ast}}} + \mu^{-1}  \lVert \partial_{t} a_{\zeta}^{2} \rVert_{C_{t}C_{x}^{2}} \lVert \eta_{\zeta}^{2} \rVert_{C_{t}L_{x}^{p^{\ast}}} \label{estimate 20}\\
&\overset{\eqref{[Equation (73), Y20b]} }{\lesssim} \lambda_{q+1}^{-2\eta} [M_{0}(t) l^{-15} + \lambda_{q+1}^{4\eta -1} M_{0}(t) l^{-\frac{27}{2}}] r^{2- \frac{2}{p^{\ast}}} \overset{\eqref{[Equation (57), Y20b]}}{\lesssim} M_{0}(t) \delta_{q+2} \lambda_{q+1}^{- \frac{59\alpha}{8}}  \leq (2\pi)^{-2 (\frac{p^{\ast} -1}{p^{\ast}})} \frac{c_{R} \delta_{q+2} M_{0}(t)}{25}. \nonumber  
\end{align} 
Therefore, we conclude from \eqref{estimate 16}-\eqref{estimate 20} applied to \eqref{estimate 27} that 
\begin{equation}\label{estimate 29}
\lVert R_{\text{osc}} \rVert_{C_{t}L_{x}^{p^{\ast}}} \leq (2\pi)^{-2 (\frac{p^{\ast} - 1}{p^{\ast}})} \frac{c_{R} \delta_{q+2} M_{0}(t)}{5}. 
\end{equation} 

Next, for all $t \in [0, T_{L}]$ we estimate using that $\delta \in (0, \frac{1}{12})$, $2\beta b < \frac{\alpha}{8}$ from \eqref{estimate 21}, $\alpha b > 16$ due to our choice of $b$, and taking  $a \in 10 \mathbb{N}$ sufficiently large 
\begin{align}
& \lVert R_{\text{com1}} \rVert_{C_{t}L_{x}^{1}} \overset{\eqref{[Equation (60a), Y20b]}}{\lesssim} l (\lVert v_{q} \rVert_{C_{t,x}^{1}} + \lVert z \rVert_{C_{t}C_{x}^{1}}) (\lVert v_{q} \rVert_{C_{t}L_{x}^{2}} + \lVert z \rVert_{C_{t,x}}) + l^{\frac{1}{2} - 2 \delta} \lVert z \rVert_{C_{t}^{\frac{1}{2} - 2 \delta} C_{x}} (\lVert v_{q} \rVert_{C_{t}L_{x}^{2}} + \lVert z \rVert_{C_{t,x}}) \nonumber \\
& \hspace{2mm} \overset{\eqref{[Equation (38), Y20b]}}{\lesssim} l^{\frac{1}{2} - 2 \delta} M_{0}(t) \lambda_{q}^{4}  \overset{\eqref{[Equation (56), Y20b]}}{ \lesssim} \delta_{q+2} M_{0}(t) a^{b^{q} [ - \frac{\alpha b}{2} + \frac{10}{3} + \frac{\alpha b}{8}]} \lesssim \delta_{q+2} M_{0}(t) a^{b^{q} [ - \frac{8}{3}]} \leq \frac{M_{0}(t) c_{R} \delta_{q+2}}{5}.  \label{[Equation (105), Y20b]} 
\end{align}

Lastly, for all $t \in [0, T_{L}]$ we can estimate by using that $l^{\frac{1}{2} - 2\delta} \lambda_{q}^{4} \ll \frac{c_{R} \delta_{q+2}}{5}$ in \eqref{[Equation (105), Y20b]}, \eqref{[Equation (40a), Y20b]} at level $q+1$ that we already verified,  and taking $a \in 10 \mathbb{N}$ sufficiently large 
\begin{align}
\lVert R_{\text{com2}} \rVert_{C_{t}L_{x}^{1}} \overset{\eqref{[Equation (91e), Y20b]}}{\lesssim}& \sup_{s\in [0,t]} [ \lVert v_{q+1}(s) \rVert_{L_{x}^{1}} + \lVert z(s) \rVert_{L_{x}^{1}}] l^{\frac{1}{2} - 2\delta} (\lVert z \rVert_{C_{t}^{\frac{1}{2} - 2 \delta} L_{x}^{\infty}} + \lVert z \rVert_{C_{t}C_{x}^{\frac{1}{2} - 2 \delta}}) \nonumber\\
& \hspace{20mm} \overset{\eqref{[Equation (38), Y20b]}}{\lesssim} M_{0}(t) l^{\frac{1}{2} - 2 \delta} \leq \frac{M_{0}(t) c_{R} \delta_{q+2}}{5}. \label{[Equation (106), Y20b]}
\end{align}  
Therefore, we can now conclude from \eqref{[Equation (98), Y20b]}, \eqref{[Equation (100), Y20b]}, \eqref{estimate 29}-\eqref{[Equation (106), Y20b]} that
\begin{align}
& \lVert \mathring{R}_{q+1} \rVert_{C_{t}L_{x}^{1}}  \overset{\eqref{[Equation (92), Y20b]}}{\leq} (2\pi)^{2(\frac{p^{\ast} - 1}{p^{\ast}})}[ \lVert R_{\text{lin}} \rVert_{C_{t}L_{x}^{p^{\ast}}} + \lVert R_{\text{cor}} \rVert_{C_{t}L_{x}^{p^{\ast}}} + \lVert R_{\text{osc}} \rVert_{C_{t}L_{x}^{p^{\ast}}}] \nonumber\\
& \hspace{65mm} + \frac{2M_{0}(t) c_{R}\delta_{q+2}}{5} \leq M_{0}(t) c_{R} \delta_{q+2} \label{[Equation (107), Y20b]}
\end{align} 
due to H$\ddot{\mathrm{o}}$lder's inequality. This verifies \eqref{[Equation (40c), Y20b]} at level $q+ 1$. 

At last, similarly to the argument in \cite{HZZ19} we can conclude by commenting on how $(v_{q+1}, \mathring{R}_{q+1})$ is $(\mathcal{F}_{t})_{t\geq 0}$-adapted and that $(v_{q+1}, \mathring{R}_{q+1})(0,x)$ are both deterministic if $(v_{q}, \mathring{R}_{q})(0,x)$ are deterministic. First, we recall that $z$ in \eqref{[Equation (24), Y20b]} is $(\mathcal{F}_{t})_{t\geq 0}$-adapted. Due to the compact support of $\varphi_{l}$ in $\mathbb{R}_{+}$, it follows that $z_{l}$ from \eqref{[Equation (58), Y20b]} is $(\mathcal{F}_{t})_{t\geq 0}$-adapted. Similarly, because $(v_{q}, \mathring{R}_{q})$ are both $(\mathcal{F}_{t})_{t\geq 0}$-adapted by hypothesis, it follows that $(v_{l}, \mathring{R}_{l})$ from \eqref{[Equation (58), Y20b]} are both $(\mathcal{F}_{t})_{t\geq 0}$-adapted. Because $M_{0}(t)$ from \eqref{[Equations (37) and (39), Y20b]} is deterministic, it follows that $\rho$ from \eqref{[Equation (65), Y20b]} is also $(\mathcal{F}_{t})_{t\geq 0}$-adapted. Due to $\rho$ and $\mathring{R}_{l}$ being $(\mathcal{F}_{t})_{t\geq 0}$-adapted, $a_{\zeta}$ from \eqref{[Equation (70), Y20b]} is also $(\mathcal{F}_{t})_{t\geq 0}$-adapted. Because $\mathbb{W}_{\zeta}, \eta_{\zeta}$, and $\psi_{\zeta}$ respectively from \eqref{[Equation (4.15), LQ20]}, \eqref{[Equation (4.11), LQ20]}, and \eqref{[Equation (4.2), LQ20]} are all deterministic, it follows that all of $w_{q+1}^{(p)}, w_{q+1}^{(c)}$, and $w_{q+1}^{(t)}$ from \eqref{[Equation (74), Y20b]} are $(\mathcal{F}_{t})_{t\geq 0}$-adapted. Consequently, $w_{q+1}$ from \eqref{[Equation (76), Y20b]} is $(\mathcal{F}_{t})_{t\geq 0}$-adapted, which in turn implies that $v_{q+1}$ from \eqref{[Equation (76), Y20b]} is $(\mathcal{F}_{t})_{t\geq 0}$-adapted. Moreover, it is also clear from the compact support of $\varphi_{l}$ in $\mathbb{R}_{+}$ that if $v_{q}(0,x)$ and $\mathring{R}_{q}(0,x)$ are deterministic, then so are $v_{l} (0,x), \mathring{R}_{l}(0,x)$, and $\partial_{t}\mathring{R}_{l}(0,x)$. Because $z(0,x) \equiv 0$ by \eqref{[Equation (24), Y20b]}, $R_{\text{com1}}(0,x)$ from \eqref{[Equation (60a), Y20b]} is also deterministic. Because $M_{0}(t)$ is deterministic, we see that $\rho(0,x)$ and $\partial_{t} \rho(0,x)$ from \eqref{[Equation (65), Y20b]} are also deterministic; this implies that $a_{\zeta}(0,x)$ and $\partial_{t}a_{\zeta}(0,x)$ from \eqref{[Equation (70), Y20b]} are also deterministic. As $\mathbb{W}_{\zeta}, \eta_{\zeta}$, and $\psi_{\zeta}$ respectively from \eqref{[Equation (4.15), LQ20]}, \eqref{[Equation (4.11), LQ20]}, and \eqref{[Equation (4.2), LQ20]} are all deterministic, we see that all of $w_{q+1}^{(p)}(0,x)$, $\partial_{t}w_{q+1}^{(p)}(0,x)$, $w_{q+1}^{(c)}(0,x)$, $\partial_{t}w_{q+1}^{(c)}(0,x)$, and $w_{q+1}^{(t)}(0,x)$ from \eqref{[Equation (74), Y20b]} are deterministic and consequently $w_{q+1}(0,x)$ from \eqref{[Equation (76), Y20b]} is deterministic. Because $v_{l}(0,x)$ is deterministic, it follows that $v_{q+1}(0,x)$ from \eqref{[Equation (76), Y20b]} is deterministic. Moreover, we see that all of $R_{\text{lin}}(0,x)$, $R_{\text{cor}}(0,x)$, and $R_{\text{com2}}(0,x)$ from \eqref{[Equation (91), Y20b]} are deterministic. Finally, $\sum_{\zeta, \zeta' \in \Lambda} \mathcal{E}_{\zeta, \zeta', 1} \rvert_{t=0}$, $\sum_{\zeta, \zeta' \in \Lambda} \sum_{k=1,3,4} \mathcal{E}_{\zeta, \zeta', 2, k} \rvert_{t=0}$, and $A_{2} + A_{3} \rvert_{t=0}$ respectively from \eqref{estimate 7}, \eqref{estimate 43}, and \eqref{estimate 48} are all deterministic and hence $R_{\text{osc}}(0,x)$ from \eqref{estimate 27} is deterministic, and consequently, so is $\mathring{R}_{q+1}(0,x)$ from \eqref{[Equation (92), Y20b]}. 

\section{Proofs of Theorems \ref{[Theorem 2.3, Y20b]}-\ref{[Theorem 2.4, Y20b]}}\label{Proof in  the case of linear multiplicative noise} 
\subsection{Proof of Theorem \ref{[Theorem 2.2, Y20b]} assuming Theorem \ref{[Theorem 2.1, Y20b]} }
Let us recall the definitions of $U_{1}, \bar{\Omega}$, and $\bar{\mathcal{B}}_{t}$ from Section \ref{Preliminaries}. We first present general results for $F$ defined through \eqref{[Equations (11) and (12), Y20b]} and $\theta$; thereafter, we apply them in case $F(u) = u$ and $B$ is an $\mathbb{R}$-valued Wiener process to prove Theorems \ref{[Theorem 2.3, Y20b]}-\ref{[Theorem 2.4, Y20b]}. We fix any $\varepsilon \in (0,1)$ for the purpose of the following definitions. 
\begin{define}\label{[Definition 5.1, Y20b]}
Let $s \geq 0$, $\xi^{\text{in}} \in L_{\sigma}^{2}$, and $\theta^{\text{in}} \in U_{1}$. A probability measure $P \in \mathcal{P} (\bar{\Omega})$ is a probabilistically weak solution to \eqref{[Equation (2), Y20b]} with initial condition $(\xi^{\text{in}}, \theta^{\text{in}})$ at initial time $s$ if  
\begin{itemize}
\item [] (M1) $P(\{ \xi(t) = \xi^{\text{in}}, \theta(t) = \theta^{\text{in}} \hspace{1mm} \forall \hspace{1mm} t \in [0,s]\}) = 1$ and for all $n \in \mathbb{N}$ 
\begin{equation}\label{[Equation (108), Y20b]}
P ( \{ (\xi, \theta) \in \bar{\Omega}: \hspace{0.5mm} \int_{0}^{n} \lVert F(\xi(r)) \rVert_{L_{2} (U, L_{\sigma}^{2})}^{2} dr < \infty \} ) = 1, 
\end{equation} 
\item [] (M2) under $P$, $\theta$ is a cylindrical $(\bar{\mathcal{B}}_{t})_{t\geq s}$-Wiener process on $U$ starting from initial condition $\theta^{\text{in}}$ at initial time $s$ and for every $\mathfrak{g}_{i} \in C^{\infty} (\mathbb{T}^{2}) \cap L_{\sigma}^{2}$ and $t\geq s$, 
\begin{equation}\label{[Equation (109), Y20b]}
\langle \xi(t) - \xi(s), \mathfrak{g}_{i} \rangle + \int_{s}^{t} \langle \text{div} (\xi(r) \otimes \xi(r)) + (-\Delta)^{m} \xi(r), \mathfrak{g}_{i} \rangle dr = \int_{s}^{t} \langle \mathfrak{g}_{i}, F(\xi(r)) d\theta(r) \rangle, 
\end{equation} 
\item [] (M3) for any $q \in \mathbb{N}$ there exists a function $t \mapsto C_{t,q} \in \mathbb{R}_{+}$ for all $t\geq s$ such that
\begin{equation}\label{[Equation (110), Y20b]}
\mathbb{E}^{P} [ \sup_{r \in [0,t]} \lVert \xi(r) \rVert_{L_{x}^{2}}^{2q} + \int_{s}^{t} \lVert \xi(r) \rVert_{H_{x}^{\varepsilon}}^{2} dr] \leq C_{t,q} (1+ \lVert \xi^{\text{in}} \rVert_{L_{x}^{2}}^{2q}). 
\end{equation}
\end{itemize}
The set of all such probabilistically weak solutions with the same constant $C_{t,q}$ in \eqref{[Equation (110), Y20b]} for every $q \in \mathbb{N}$ and $t \geq s$ is denoted by $\mathcal{W}(s, \xi^{\text{in}}, \theta^{\text{in}}, \{C_{t,q}\}_{q\in\mathbb{N}, t \geq s})$. 
\end{define} 
For any stopping time $\tau$ we set 
\begin{equation}\label{[Equation (111), Y20b]}
\bar{\Omega}_{\tau} \triangleq \{ \omega (\cdot \wedge \tau (\omega)) : \hspace{0.5mm} \omega \in \bar{\Omega} \} 
\end{equation} 
and denote the $\sigma$-field associated to $\tau$ by $\bar{\mathcal{B}}_{\tau}$. 
\begin{define}\label{[Definition 5.2, Y20b]}
Let $s \geq 0$, $\xi^{\text{in}} \in L_{\sigma}^{2}$, and $\theta^{\text{in}} \in U_{1}$. Let $\tau \geq s$ be a stopping time of $(\bar{\mathcal{B}}_{t})_{t \geq s}$. A probability measure $P \in \mathcal{P} (\bar{\Omega}_{\tau})$ is a probabilistically weak solution to \eqref{[Equation (2), Y20b]} on $[s, \tau]$ with initial condition $(\xi^{\text{in}}, \theta^{\text{in}})$ at initial time $s$ if 
\begin{itemize}
\item [] (M1) $P(\{ \xi(t) = \xi^{\text{in}}, \theta(t) = \theta^{\text{in}} \hspace{1mm} \forall \hspace{1mm} t \in [0,s] \}) = 1$ and for all $n \in \mathbb{N}$
\begin{equation}\label{[Equation (112), Y20b]}
P ( \{ ( \xi, \theta) \in \bar{\Omega}: \hspace{0.5mm} \int_{0}^{n \wedge \tau} \lVert F(\xi(r)) \rVert_{L_{2}(U, L_{\sigma}^{2})}^{2} dr < \infty \}) = 1, 
\end{equation}
\item  [] (M2) under $P$, $\langle \theta(\cdot \wedge \tau), l_{i} \rangle_{U}$, where $\{l_{i} \}_{i\in \mathbb{N}}$ is an orthonormal basis of $U$,  is a continuous, square-integrable $(\bar{\mathcal{B}}_{t})_{t \geq s}$-martingale with initial condition $\langle \theta^{\text{in}}, l_{i} \rangle$ at initial time $s$ with its quadratic variation process given by $(t \wedge \tau - s) \lVert l_{i} \rVert_{U}^{2}$ and for every $\mathfrak{g}_{i} \in C^{\infty} (\mathbb{T}^{2} ) \cap L_{\sigma}^{2}$ and $t \geq s$ 
\begin{equation}\label{[Equation (113), Y20b]}
\langle \xi(t\wedge \tau) - \xi(s),\mathfrak{g}_{i} \rangle + \int_{s}^{t \wedge \tau} \langle \text{div} (\xi(r) \otimes \xi(r)) + (-\Delta)^{m} \xi(r), \mathfrak{g}_{i} \rangle dr = \int_{s}^{t \wedge \tau} \langle \mathfrak{g}_{i}, F(\xi(r)) d\theta(r) \rangle, 
\end{equation}
\item [] (M3) for any $q \in \mathbb{N}$ there exists a function $t\mapsto C_{t,q} \in \mathbb{R}_{+}$ for all $t\geq s$ such that
\begin{equation}\label{[Equation (114), Y20b]}
\mathbb{E}^{P} [ \sup_{r \in [0, t \wedge \tau]} \lVert \xi(r) \rVert_{L_{x}^{2}}^{2q} + \int_{s}^{t \wedge \tau} \lVert \xi(r) \rVert_{H_{x}^{\varepsilon}}^{2} dr] \leq C_{t,q} (1+ \lVert \xi^{\text{in}} \rVert_{L_{x}^{2}}^{2q}). 
\end{equation} 
\end{itemize} 
\end{define} 
The joint uniqueness in law for \eqref{[Equation (2), Y20b]} is equivalent to the uniqueness of probabilistically weak solution in Definition \ref{[Definition 5.1, Y20b]}, which holds if probabilistically weak solutions starting from the same initial distributions are unique. 
\begin{proposition}\label{[Proposition 5.1, Y20b]}
For every $(s, \xi^{\text{in}}, \theta^{\text{in}}) \in [0,\infty) \times L_{\sigma}^{2} \times U_{1}$, there exists a probabilistically weak solution $P \in \mathcal{P} (\bar{\Omega})$ to \eqref{[Equation (2), Y20b]} with  initial condition $(\xi^{\text{in}}, \theta^{\text{in}})$ at initial time $s$ according to Definition \ref{[Definition 5.1, Y20b]}. Moreover, if there exists a family $(s_{n}, \xi_{n}, \theta_{n}) \subset [0,\infty) \times L_{\sigma}^{2} \times U_{1}$ such that $\lim_{n\to\infty} \lVert (s_{n}, \xi_{n}, \theta_{n}) - (s, \xi^{\text{in}}, \theta^{\text{in}}) \rVert_{\mathbb{R} \times L_{x}^{2} \times U_{1}} = 0$ and $P_{n} \in \mathcal{W} (s_{n}, \xi_{n}, \theta_{n}, \{C_{t,q}\}_{q\in\mathbb{N}, t \geq s_{n}})$, then there exists a subsequence $\{P_{n_{k}} \}_{k\in\mathbb{N}}$ that converges weakly to some $P \in \mathcal{W}(s,\xi^{\text{in}}, \theta^{\text{in}}$, $\{C_{t,q}\}_{q\in\mathbb{N}, t \geq s})$. 
\end{proposition} 
\begin{proof}[Proof of Proposition \ref{[Proposition 5.1, Y20b]}]
The existence of the probabilistically weak solution according to Definition \ref{[Definition 5.1, Y20b]} follows from Proposition \ref{[Proposition 4.1, Y20b]} and an application of martingale representation theorem (e.g., \cite[Theorem 8.2]{DZ14}) while the proof of stability result can follow that of \cite[Theorem 5.1]{HZZ19} with appropriate modifications concerning the differences in spatial dimension and fractional Laplacian, similarly to the proof of Proposition \ref{[Proposition 4.1, Y20b]} (see also \cite[Proposition 5.1]{Y20a}). 
\end{proof}
Next, we have the following results as a consequence of Proposition \ref{[Proposition 5.1, Y20b]}; the proofs of analogous results from \cite{HZZ19} did not rely on the specific form of the diffusive term or the spatial dimension and thus apply to our case. 
\begin{lemma}\label{[Lemma 5.2, Y20b]}
\rm{(\cite[Proposition 5.2]{HZZ19})} Let $\tau$ be a bounded stopping time of $(\bar{\mathcal{B}}_{t})_{t \geq 0}$. Then for every $\omega \in \bar{\Omega}$ there exists $Q_{\omega} \in \mathcal{P}(\bar{\Omega})$ such that 
\begin{subequations}
\begin{align}
& Q_{\omega} ( \{ \omega' \in \bar{\Omega}: \hspace{0.5mm} \hspace{1mm}  ( \xi, \theta) (t, \omega') = (\xi, \theta) (t,\omega) \hspace{1mm} \forall \hspace{1mm} t \in [0, \tau(\omega)] \}) = 1, \\
& Q_{\omega} (A) = R_{\tau(\omega), \xi(\tau(\omega), \omega), \theta(\tau(\omega), \omega)} (A) \hspace{1mm} \forall \hspace{1mm} A \in \bar{\mathcal{B}}^{\tau(\omega)}, 
\end{align} 
\end{subequations}
where $R_{\tau(\omega), \xi(\tau(\omega), \omega), \theta(\tau(\omega), \omega)} \in \mathcal{P} (\bar{\Omega})$ is a probabilistically weak solution to \eqref{[Equation (2), Y20b]} with initial condition $(\xi(\tau(\omega), \omega), \theta(\tau(\omega), \omega))$ at initial time $\tau(\omega)$. Moreover, for every $A \in \bar{\mathcal{B}}$ the map $\omega \mapsto Q_{\omega}(A)$ is $\bar{\mathcal{B}}_{\tau}$-measurable. 
\end{lemma}

\begin{lemma}\label{[Lemma 5.3, Y20b]}
\rm{(\cite[Proposition 5.3]{HZZ19})} Let $\xi^{\text{in}} \in L_{\sigma}^{2}$ and $P$ be a probabilistically weak solution to \eqref{[Equation (2), Y20b]} on $[0,\tau]$ with initial condition $(\xi^{\text{in}}, 0)$ at initial time $0$ according to Definition \ref{[Definition 5.2, Y20b]}. In addition to the hypothesis of Lemma \ref{[Lemma 5.2, Y20b]}, suppose that there exists a Borel set $\mathcal{N} \subset \bar{\Omega}_{\tau}$ such that $P(\mathcal{N}) = 0$ and $Q_{\omega}$ from Lemma \ref{[Lemma 5.2, Y20b]} satisfies for every $\omega \in \bar{\Omega}_{\tau} \setminus \mathcal{N}$ 
\begin{equation}\label{[Equation (116), Y20b]}
Q_{\omega} (\{ \omega' \in \bar{\Omega}: \hspace{0.5mm}  \tau(\omega') = \tau(\omega) \}) = 1. 
\end{equation} 
Then the probability measure $P\otimes_{\tau}R \in \mathcal{P} (\bar{\Omega})$ defined by 
\begin{equation}\label{[Equation (117), Y20b]}
P \otimes_{\tau} R (\cdot) \triangleq \int_{\bar{\Omega}} Q_{\omega} (\cdot) P(d \omega)
\end{equation} 
satisfies $P\otimes_{\tau}R \rvert_{\bar{\Omega}_{\tau}} = P \rvert_{\bar{\Omega}_{\tau}}$ and it is a probabilistically weak solution to \eqref{[Equation (2), Y20b]} on $[0,\infty)$ with initial condition $(\xi^{\text{in}}, 0)$ at initial time $0$. 
\end{lemma}

Now we fix an $\mathbb{R}$-valued Wiener process $B$ on $(\Omega, \mathcal{F}, \mathbb{P})$ and apply Definitions \ref{[Definition 5.1, Y20b]}-\ref{[Definition 5.2, Y20b]}, Proposition \ref{[Proposition 5.1, Y20b]}, and Lemmas \ref{[Lemma 5.2, Y20b]}-\ref{[Lemma 5.3, Y20b]} with $F(u) = u$ and such $B$. For $n \in \mathbb{N}, L > 1$, and $\delta \in (0, \frac{1}{12})$ we define similarly to \eqref{[Equation (31), Y20b]}-\eqref{[Equation (32), Y20b]} 
\begin{subequations}\label{estimate 51}
\begin{align}
&\tau_{L}^{n} (\omega) \triangleq \inf \{t \geq 0: \hspace{0.5mm}  \lvert \theta(t,\omega) \rvert > (L - \frac{1}{n})^{\frac{1}{4}} \} \wedge \inf\{t > 0: \hspace{0.5mm}  \lVert \theta(\omega) \rVert_{C_{t}^{\frac{1}{2} - 2\delta}} > (L - \frac{1}{n})^{\frac{1}{2}} \} \wedge L, \label{[Equation (118a), Y20b]}\\
&\tau_{L} \triangleq \lim_{n\to\infty} \tau_{L}^{n}.  \label{[Equation (118b), Y20b]}
\end{align}
\end{subequations} 
It follows from \cite[Lemma 3.5]{HZZ19} that $\tau_{L}^{n}$ is a stopping time of $(\bar{\mathcal{B}}_{t})_{t\geq 0}$ and thus so is $\tau_{L}$. For the fixed $(\Omega, \mathcal{F}, \textbf{P})$ we assume Theorem \ref{[Theorem 2.3, Y20b]} and denote by $u$ the solution constructed from Theorem \ref{[Theorem 2.3, Y20b]} on $[0, \mathfrak{t}]$ where $\mathfrak{t} = T_{L}$ for $L$ sufficiently large and 
\begin{equation}\label{[Equation (119), Y20b]}
T_{L} \triangleq \inf\{t > 0: \hspace{0.5mm}  \lvert B(t) \rvert \geq L^{\frac{1}{4}} \} \wedge \inf\{t > 0: \hspace{0.5mm} \lVert B \rVert_{C_{t}^{\frac{1}{2} - 2\delta}} \geq L^{\frac{1}{2}} \} \wedge L \text{ with } \delta \in (0, \frac{1}{12}). 
\end{equation} 
We observe that $T_{L} \nearrow + \infty$ $\textbf{P}$-a.s. as $L \nearrow + \infty$. Let us also denote the law of $(u,B)$ by $P$. 
\begin{proposition}\label{[Proposition 5.4, Y20b]}
Let $\tau_{L}$ be defined by \eqref{[Equation (118b), Y20b]}. Then $P$, the law of $(u,B)$, is a probabilistically weak solution to \eqref{[Equation (2), Y20b]} on $[0, \tau_{L}]$ according to Definition \ref{[Definition 5.2, Y20b]}. 
\end{proposition}
\begin{proof}[Proof of Proposition \ref{[Proposition 5.4, Y20b]}]
The proof is similar to that of Proposition \ref{[Proposition 4.5, Y20b]} making use of the fact that 
\begin{equation}\label{[Equation (120), Y20b]}
\theta(t, (u, B)) = B(t) \hspace{1mm} \forall \hspace{1mm} t \in [0, T_{L}] \hspace{1mm} \textbf{P}\text{-almost surely}
\end{equation}
(see also the proofs of \cite[Propositions 3.7 and 5.4]{HZZ19} and \cite[Proposition 4.5]{Y20a}). 
\end{proof}
Next, we extend $P$ on $[0,\tau_{L}]$ to $[0,\infty)$. 
\begin{proposition}\label{[Proposition 5.5, Y20b]}
Let $\tau_{L}$ be defined by \eqref{[Equation (118b), Y20b]} and $P$ denote the law of $(u,B)$ constructed from Theorem \ref{[Theorem 2.3, Y20b]}. Then the probability measure $P \otimes_{\tau_{L}} R$ in \eqref{[Equation (117), Y20b]} is a probabilistically weak solution to \eqref{[Equation (2), Y20b]} on $[0,\infty)$ according to Definition \ref{[Definition 5.1, Y20b]}.
\end{proposition}
\begin{proof}[Proof of Proposition \ref{[Proposition 5.5, Y20b]}]
Because $\tau_{L}$ is a stopping time of $(\bar{\mathcal{B}}_{t})_{t\geq 0}$ that is bounded by $L$ due to \eqref{[Equation (118a), Y20b]}, the hypothesis of Lemma \ref{[Lemma 5.2, Y20b]} is verified. By Proposition \ref{[Proposition 5.4, Y20b]}, $P$ is a probabilistically weak solution to \eqref{[Equation (2), Y20b]} on $[0,\tau_{L}]$.  Therefore, Lemma \ref{[Lemma 5.3, Y20b]} gives us the desired result once we verify the existence of a Borel set $\mathcal{N} \subset \bar{\Omega}_{\tau}$ such that $P(\mathcal{N}) = 0$ and \eqref{[Equation (116), Y20b]} holds for every $\omega \in  \bar{\Omega}_{\tau} \setminus \mathcal{N}$, and that can be achieved similarly to the proof of Proposition \ref{[Proposition 4.6, Y20b]} (see also the proofs of \cite[Propositions 3.8 and 5.5]{HZZ19} and \cite[Proposition 4.6]{Y20a}).
\end{proof}

Taking Theorem \ref{[Theorem 2.3, Y20b]} for granted, we are now able to prove Theorem \ref{[Theorem 2.4, Y20b]}. 
\begin{proof}[Proof of Theorem \ref{[Theorem 2.4, Y20b]} assuming Theorem \ref{[Theorem 2.3, Y20b]} ]
The proof is similar to that of Theorem \ref{[Theorem 2.2, Y20b]} assuming Theorem \ref{[Theorem 2.1, Y20b]} in Subsection \ref{Subsection 4.1}; we sketch it for completeness. We fix $T> 0$ arbitrarily, any $\kappa \in (0,1)$, and $K > 1$ such that $\kappa K^{2} \geq 1$. The probability measure $P \otimes_{\tau_{L}} R$ from Proposition \ref{[Proposition 5.5, Y20b]} satisfies 
\begin{equation*}
P \otimes_{\tau_{L}} R ( \{ \tau_{L} \geq T \}) 
\overset{\eqref{[Equation (117), Y20b]}}{=} \textbf{P} ( \{ \tau_{L} (u, B)) \geq T \}) \overset{\eqref{estimate 51} \eqref{[Equation (119), Y20b]} \eqref{[Equation (120), Y20b]}}{=} \textbf{P} ( \{T_{L} \geq T \}) > \kappa, 
\end{equation*} 
where the last inequality is due to Theorem \ref{[Theorem 2.3, Y20b]}. This leads us to $\mathbb{E}^{P \otimes_{\tau_{L}} R} [ \lVert \xi(T) \rVert_{L_{x}^{2}}^{2}]$ $>$ $\kappa K^{2} e^{T} \lVert \xi^{\text{in}} \rVert_{L_{x}^{2}}^{2}$, where $\xi^{\text{in}}$ is the deterministic initial condition constructed through Theorem \ref{[Theorem 2.3, Y20b]}. On the other hand, via a classical Galerkin approximation scheme (e.g., \cite{FR08}) one can readily construct a probabilistically weak solution $\Theta$ to \eqref{[Equation (2), Y20b]} starting also from $\xi^{\text{in}}$ such that $\mathbb{E}^{\Theta} [ \lVert \xi(T) \rVert_{L_{x}^{2}}^{2}] \leq e^{T} \lVert \xi^{\text{in}} \rVert_{L_{x}^{2}}^{2}$. Because $\kappa K^{2} \geq 1$, this implies the lack of uniqueness of probabilistically weak solution to \eqref{[Equation (2), Y20b]} and equivalently the lack of joint uniqueness in law for \eqref{[Equation (2), Y20b]}, and consequently the non-uniqueness in law for \eqref{[Equation (2), Y20b]} by \cite[Theorem C.1]{HZZ19}, which is an infinite-dimensional version of \cite[Theorem 3.1]{C03} due to Cherny. 
\end{proof}

\subsection{Proof of Theorem \ref{[Theorem 2.3, Y20b]} assuming Proposition \ref{[Proposition 5.7, Y20b]}}

We define $\Upsilon(t)\triangleq e^{B(t)}$ and $v  \triangleq\Upsilon^{-1} u$ for $t \geq 0$. It follows from Ito's product formula (e.g., \cite[Theorem 4.4.13]{A09}) on \eqref{[Equation (2), Y20b]} that 
\begin{equation}\label{[Equation (121), Y20b]}
\partial_{t} v + \frac{1}{2} v + (-\Delta)^{m} v + \Upsilon \text{div} (v\otimes v) + \Upsilon^{-1} \nabla \pi = 0, \hspace{3mm} \nabla\cdot v =0, \hspace{3mm} t> 0.
\end{equation} 
Considering \eqref{[Equation (121), Y20b]}, for every $q \in \mathbb{N}_{0}$ we will construct $(v_{q}, \mathring{R}_{q})$ that solves 
\begin{equation}\label{[Equation (122), Y20b]}
\partial_{t} v_{q} + \frac{1}{2} v_{q} + (-\Delta)^{m} v_{q} + \Upsilon \text{div} (v_{q}\otimes v_{q}) + \nabla p_{q} = \text{div} \mathring{R}_{q}, \hspace{3mm} \nabla\cdot v_{q} =0, \hspace{3mm} t>0, 
\end{equation} 
when $\mathring{R}_{q}$ is assumed to be a trace-free symmetric matrix. Similarly to \eqref{[Equations (37) and (39), Y20b]} in the additive case, we continue to define $\lambda_{q} \triangleq a^{b^{q}}, \delta_{q} \triangleq \lambda_{q}^{-2\beta}$ for $a \in 10 \mathbb{N}$, $b \in \mathbb{N}$, and $\beta \in (0,1)$ so that the requirement of $\lambda_{q+1} \in 5 \mathbb{N}$ of \eqref{[Equation (4.14), LQ20]} is satisfied, while differently from \eqref{[Equations (37) and (39), Y20b]} we define 
\begin{equation}\label{[Equation (125), Y20b]}
M_{0}(t) \triangleq e^{4Lt + 2L} \hspace{2mm} \text{ and } \hspace{2mm} m_{L} \triangleq \sqrt{3} L^{\frac{1}{4}} e^{\frac{1}{2} L^{\frac{1}{4}}}.
\end{equation}  
Due to \eqref{[Equation (119), Y20b]} we obtain for all $L > 1, \delta \in (0, \frac{1}{12})$, and $t \in [0, T_{L}]$ 
\begin{equation}\label{[Equation (123), Y20b]}
\lvert B(t) \rvert \leq L^{\frac{1}{4}} \text{ and } \lVert B \rVert_{C_{t}^{\frac{1}{2} - 2 \delta}} \leq L^{\frac{1}{2}}
\end{equation}
which immediately implies 
\begin{equation}\label{[Equation (124), Y20b]}
\lVert \Upsilon \rVert_{C_{t}^{\frac{1}{2} - 2\delta}} + \lvert \Upsilon(t) \rvert + \lvert \Upsilon^{-1}(t) \rvert \leq e^{L^{\frac{1}{4}}} L^{\frac{1}{2}} + 2e^{L^{\frac{1}{4}}} \leq m_{L}^{2}.
\end{equation} 
For induction we assume that $(v_{q}, \mathring{R}_{q})$ satisfy the following bounds on $[0, T_{L}]$: 
\begin{subequations}\label{[Equation (126), Y20b]}
\begin{align}
& \lVert v_{q} \rVert_{C_{t}L_{x}^{2}} \leq m_{L} M_{0}(t)^{\frac{1}{2}} (1+ \sum_{1 \leq   \iota \leq q} \delta_{\iota}^{\frac{1}{2}}) \leq 2m_{L} M_{0}(t)^{\frac{1}{2}}, \label{[Equation (126a), Y20b]}\\
& \lVert v_{q} \rVert_{C_{t,x}^{1}} \leq m_{L} M_{0}(t)^{\frac{1}{2}} \lambda_{q}^{4}, \label{[Equation (126b), Y20b]}\\
& \lVert \mathring{R}_{q} \rVert_{C_{t}L_{x}^{1}} \leq M_{0}(t) c_{R} \delta_{q+1}, \label{[Equation (126c), Y20b]}
\end{align}
\end{subequations}  
where $c_{R} > 0$ is again a universal constant to be determined subsequently and we assumed again $a^{\beta b} > 3$, as formally stated in \eqref{[Equation (129), Y20b]}, in order to deduce $\sum_{1\leq \iota} \delta_{\iota}^{\frac{1}{2}} < \frac{1}{2}$. 

\begin{proposition}\label{[Proposition 5.6, Y20b]}
Let $L > 1$ and define 
\begin{equation}\label{[Equation (127), Y20b]}
v_{0}(t,x) \triangleq \frac{m_{L} e^{2L t + L}}{2\pi} \begin{pmatrix}
\sin(x^{2}) & 0
\end{pmatrix}^{T}.
\end{equation} 
Then together with 
\begin{equation}\label{[Equation (128), Y20b]}
\mathring{R}_{0}(t,x) \triangleq \frac{ m_{L} (2L + \frac{1}{2}) e^{2L t + L}}{2\pi} 
\begin{pmatrix}
0 & - \cos(x^{2}) \\
-\cos(x^{2}) & 0 
\end{pmatrix} 
+ \mathcal{R} (-\Delta)^{m} v_{0}(t,x), 
\end{equation} 
it satisfies \eqref{[Equation (122), Y20b]} at level $q = 0$. Moreover, \eqref{[Equation (126), Y20b]} is satisfied at level $q = 0$ provided 
\begin{equation}\label{[Equation (129), Y20b]}
72\sqrt{3} < 8 \sqrt{3} a^{2 \beta b} \leq \frac{c_{R} e^{L - \frac{1}{2} L^{\frac{1}{4}}}}{L^{\frac{1}{4}} (2L + \frac{1}{2}+ \pi)}, \hspace{3mm} L \leq a^{4} \pi - 1, 
\end{equation} 
where the inequality $9 < a^{2 \beta b}$ is assumed for the sake of second inequality in \eqref{[Equation (126a), Y20b]}. Furthermore, $v_{0}(0,x)$ and $\mathring{R}_{0}(0,x)$ are both deterministic. 
\end{proposition} 

\begin{proof}[Proof of Proposition \ref{[Proposition 5.6, Y20b]}]
The proof is similar to that of Proposition \ref{[Proposition 4.7, Y20b]}. Let us observe that $v_{0}$ is divergence-free, while $\mathring{R}_{0}$ is trace-free and symmetric. It may be immediately verified that $(v_{0}, \mathring{R}_{0})$ solves \eqref{[Equation (122), Y20b]} with $p_{0} \equiv 0$ by using the fact that $(v_{0} \cdot \nabla) v_{0} = 0$ and Lemma \ref{[Definition 9, Lemma 10, CDS12]}. Next, for all $t \in [0, T_{L}]$ we can compute similarly to \eqref{[Equations (44a) and (44b), Y20b]}
\begin{equation}\label{[Equation (130), Y20b]}
\lVert v_{0}(t) \rVert_{L_{x}^{2}} =  \frac{m_{L} M_{0}(t)^{\frac{1}{2}}}{\sqrt{2}} \leq m_{L} M_{0}(t)^{\frac{1}{2}}, \hspace{1mm}  \lVert v_{0} \rVert_{C_{t,x}^{1}} = \frac{m_{L} (1+L) M_{0}(t)^{\frac{1}{2}}}{\pi} \overset{\eqref{[Equation (129), Y20b]}}{\leq} m_{L} M_{0}(t)^{\frac{1}{2}} \lambda_{0}^{4}.
\end{equation}
Finally, using $ \lVert \mathcal{R} ( -\Delta)^{m} v_{0} \rVert_{L_{x}^{2}}  \leq 4 \lVert v_{0} \rVert_{L_{x}^{2}}$ due to $\Delta v_{0} = -v_{0}$ and \eqref{estimate 35} we can compute 
\begin{align}
\lVert \mathring{R}_{0} (t) \rVert_{L_{x}^{1}} \leq m_{L} (2L + \frac{1}{2}) M_{0}(t)^{\frac{1}{2}} 8 + (2\pi) 4 \lVert v_{0}(t) \rVert_{L_{x}^{2}} \overset{\eqref{[Equation (129), Y20b]} \eqref{[Equation (130), Y20b]} }{\leq} M_{0}(t) c_{R} \delta_{1}. 
\end{align}  
\end{proof}

We point out that 
\begin{equation}\label{[Equation (132), Y20b]}
72 \sqrt{3} < \frac{c_{R} e^{L - \frac{1}{2} L^{\frac{1}{4}}}}{L^{\frac{1}{4}} ( 2L + \frac{1}{2} + \pi)}
\end{equation} 
is not sufficient but necessary to satisfy \eqref{[Equation (129), Y20b]}.
\begin{proposition}\label{[Proposition 5.7, Y20b]}
Let $L > 1$ satisfy \eqref{[Equation (132), Y20b]} and suppose that $(v_{q}, \mathring{R}_{q})$ is an $(\mathcal{F}_{t})_{t\geq 0}$-adapted solution to \eqref{[Equation (122), Y20b]} that satisfies \eqref{[Equation (126), Y20b]}. Then there exists a choice of parameters $a, b$, and $\beta$ such that \eqref{[Equation (129), Y20b]} is fulfilled and an $(\mathcal{F}_{t})_{t\geq 0}$-adapted process $(v_{q+1}, \mathring{R}_{q+1})$ that satisfies \eqref{[Equation (122), Y20b]}, \eqref{[Equation (126), Y20b]} at level $q + 1$, and 
\begin{equation}\label{[Equation (133), Y20b]}
\lVert v_{q+1}(t) - v_{q}(t) \rVert_{L_{x}^{2}} \leq m_{L} M_{0}(t)^{\frac{1}{2}} \delta_{q+1}^{\frac{1}{2}} \hspace{3mm} \forall \hspace{1mm} t \in [0,T_{L}].
\end{equation} 
Furthermore, if $v_{q}(0, x)$ and $\mathring{R}_{q}(0,x)$ are deterministic, then so are $v_{q+1}(0,x)$ and $\mathring{R}_{q+1}(0,x)$. 
\end{proposition} 

Taking Proposition \ref{[Proposition 5.7, Y20b]} for granted, we can now prove Theorem \ref{[Theorem 2.3, Y20b]}. 
\begin{proof}[Proof of Theorem \ref{[Theorem 2.3, Y20b]} assuming Proposition \ref{[Proposition 5.7, Y20b]}]
This proof is similar to the proof of Theorem \ref{[Theorem 2.1, Y20b]} assuming Proposition \ref{[Proposition 4.8, Y20b]} in Subsection \ref{Subsection 4.2}; we sketch it in the Appendix for completeness. 
\end{proof}

\subsection{Proof of Proposition \ref{[Proposition 5.7, Y20b]}}

\subsubsection{Choice of parameters}

We fix $L$ sufficiently large so that it satisfies \eqref{[Equation (132), Y20b]}. We take the same choices of $m^{\ast}, \eta, \alpha, r, \mu$, and $\sigma$ in \eqref{[Equation (2.2), LQ20]} - \eqref{estimate 22}, and $b \in\{\iota \in \mathbb{N}: \iota > \frac{16}{\alpha} \}$ such that $r \in \mathbb{N}$ and $\lambda_{q+1} \sigma \in 10 \mathbb{N}$ so that both requirements of $r \in \mathbb{N}$ and $\lambda_{q+1} \sigma \in 5 \mathbb{N}$ from \eqref{[Equation (4.14), LQ20]} are satisfied. Then we define $\beta > 0$ sufficiently small to satisfy \eqref{estimate 21} and $l$ by \eqref{[Equation (56), Y20b]} so that \eqref{[Equation (57), Y20b]} remains valid. We take $a \in 10 \mathbb{N}$ larger if necessary so that $a^{26} \geq \sqrt{3} L^{\frac{1}{4}} e^{\frac{1}{2} L^{\frac{1}{4}}}$; because $\alpha b > 16$ and $c_{R} \ll 1$ we see that this implies 
\begin{equation}\label{special} 
m_{L} \overset{\eqref{[Equation (125), Y20b]}}{\leq} a^{\frac{3 \alpha b}{2} + 2} \overset{\eqref{[Equation (56), Y20b]}}{\leq} l^{-1} \text{ and } m_{L} \overset{\eqref{[Equation (125), Y20b]} \eqref{[Equation (132), Y20b]}}{\leq} c_{R} e^{L} \leq M_{0}(t)^{\frac{1}{2}}. 
\end{equation} 
Lastly, taking $a \in 10 \mathbb{N}$ even larger can guarantee $L \leq a^{4} \pi -1$ in \eqref{[Equation (129), Y20b]}  while taking $\beta > 0$ even smaller if necessary allows the other inequalities in \eqref{[Equation (129), Y20b]} to be satisfied, namely 
\begin{equation*}
72 \sqrt{3} < 8 \sqrt{3} a^{2\beta b} \leq \frac{c_{R} e^{L - \frac{1}{2} L^{\frac{1}{4}}}}{L^{\frac{1}{4}} (2L + \frac{1}{2} + \pi)}. 
\end{equation*}
Thus, hereafter we consider such $m^{\ast}, \eta, \alpha, b,$ and $l$ fixed, preserving our freedom to take $a \in 10 \mathbb{N}$ larger and $\beta > 0$ smaller as needed. 

\subsubsection{Mollification} 
We mollify $v_{q}, \mathring{R}_{q}$, and $\Upsilon(t) = e^{B(t)}$ by $\phi_{l}$ and $\varphi_{l}$ again so that 
\begin{equation}\label{[Equation (137), Y20b]}
v_{l} \triangleq (v_{q} \ast_{x} \phi_{l}) \ast_{t} \varphi_{l}, \hspace{2mm} \mathring{R}_{l} \triangleq (\mathring{R}_{q} \ast_{x} \phi_{l}) \ast_{t} \varphi_{l}, \hspace{2mm} \text{ and } \hspace{2mm} \Upsilon_{l} \triangleq \Upsilon \ast_{t} \varphi_{l}. 
\end{equation} 
By \eqref{[Equation (122), Y20b]} we see that $v_{l}, \mathring{R}_{l}$, and $\Upsilon_{l}$ satisfy 
\begin{equation}\label{[Equation (138), Y20b]}
\partial_{t} v_{l} + \frac{1}{2} v_{l} + (-\Delta)^{m} v_{l} + \Upsilon_{l} \text{div} (v_{l} \otimes v_{l}) + \nabla p_{l} = \text{div} (\mathring{R}_{l} + R_{\text{com1}})
\end{equation} 
where 
\begin{subequations}
\begin{align}
p_{l} \triangleq& (p_{q} \ast_{x} \phi_{l}) \ast_{t} \varphi_{l} - \frac{1}{2} (\Upsilon_{l} \lvert v_{l} \rvert^{2} - (( \Upsilon \lvert v_{q} \rvert^{2} ) \ast_{x} \phi_{l} ) \ast_{t} \varphi_{l} ), \label{estimate 55}\\
R_{\text{com1}} \triangleq& R_{\text{commutator1}} \triangleq - ((\Upsilon (v_{q} \mathring{\otimes} v_{q} )) \ast_{x} \phi_{l} )\ast_{t} \varphi_{l} + \Upsilon_{l} (v_{l} \mathring{\otimes} v_{l}). \label{estimate 56}
\end{align}
\end{subequations}
Next, making use of the fact that $\alpha b > 16$ and taking $a \in 10 \mathbb{N}$ sufficiently large we obtain for all $t \in [0, T_{L}]$ and $N \geq 1$ 
\begin{subequations}
\begin{align}
&\lVert v_{q} - v_{l} \rVert_{C_{t}L_{x}^{2}}  \overset{\eqref{[Equation (57), Y20b]} \eqref{[Equation (126b), Y20b]}}{\lesssim} m_{L} M_{0}(t)^{\frac{1}{2}} \lambda_{q+1}^{-\alpha}  \leq \frac{m_{L}}{4} M_{0}(t)^{\frac{1}{2}} \delta_{q+1}^{\frac{1}{2}}, \label{[Equation (140a), Y20b]}\\
& \lVert v_{l} \rVert_{C_{t}L_{x}^{2}} \overset{\eqref{[Equation (126a), Y20b]}}{\leq} m_{L} M_{0}(t)^{\frac{1}{2}} (1+ \sum_{1 \leq \iota \leq q} \delta_{\iota}^{\frac{1}{2}}) \overset{\eqref{[Equation (129), Y20b]}}{\leq} 2m_{L} M_{0}(t)^{\frac{1}{2}}, \label{[Equation (140b), Y20b]}\\
& \lVert v_{l} \rVert_{C_{t,x}^{N}} \overset{\eqref{[Equation (126b), Y20b]}}{\lesssim} l^{-N + 1} m_{L} M_{0}(t)^{\frac{1}{2}} \lambda_{q}^{4} \overset{\eqref{[Equation (56), Y20b]}}{\leq} l^{-N} m_{L} M_{0}(t)^{\frac{1}{2}} \lambda_{q+1}^{-\alpha}. \label{[Equation (140c), Y20b]} 
\end{align}
\end{subequations}

\subsubsection{Perturbation}
We proceed with the same definition of $\chi$ in \eqref{[Equation (64), Y20b]} and $\rho$ in \eqref{[Equation (65), Y20b]} identically except that $M_{0}(t)$ is now defined by \eqref{[Equation (125), Y20b]} instead of \eqref{[Equations (37) and (39), Y20b]}. Although our definition of $\mathring{R}_{0}$ in \eqref{[Equation (128), Y20b]} differs from that of \eqref{[Equation (42), Y20b]}, the estimates of \eqref{[Equation (66), Y20b]} and \eqref{[Equation (67), Y20b]} remain valid as their proofs depend only on the definitions of $\rho$ and $\chi$, not $M_{0}(t)$ or $\mathring{R}_{l}$. We define a modified amplitude function to be 
\begin{equation}\label{[Equation (141), Y20b]}
\bar{a}_{\zeta} (\omega, t, x) \triangleq \bar{a}_{\zeta, q+1}(\omega, t, x) \triangleq \Upsilon_{l}^{-\frac{1}{2}} a_{\zeta} (\omega, t, x), 
\end{equation} 
where $a_{\zeta} (\omega, t,x)$ is identical to that defined in \eqref{[Equation (70), Y20b]}. For convenience let us observe a simple estimate of 
\begin{equation}\label{simple}
\lVert \Upsilon_{l}^{-\frac{1}{2}} \rVert_{C_{t}} \overset{\eqref{[Equation (125), Y20b]} \eqref{[Equation (123), Y20b]}}{\leq} m_{L}. 
\end{equation}
Using this estimate, for all $t \in [0, T_{L}]$ by taking $c_{R} \ll M^{-4}$ we can obtain 
\begin{equation}
\lVert \bar{a}_{\zeta} \rVert_{C_{t}L_{x}^{2}} \overset{\eqref{[Equation (66), Y20b]} \eqref{[Equation (67), Y20b]} }{\leq} m_{L} \sqrt{12} [ 4 \pi^{2} c_{R} \delta_{q+1} M_{0}(t) + \lVert \mathring{R}_{l} (\omega) \rVert_{C_{t}L_{x}^{1}}]^{\frac{1}{2}} (\frac{M}{C_{\Lambda}}) \overset{\eqref{estimate 4} }{\leq} \frac{c_{R}^{\frac{1}{4}} m_{L} M_{0}(t)^{\frac{1}{2}} \delta_{q+1}^{\frac{1}{2}}}{2 \lvert \Lambda \rvert}. 
\end{equation} 
Because \eqref{[Equation (40c), Y20b]} and \eqref{[Equation (126c), Y20b]} are identical except the definitions of $M_{0}(t)$, tracing the proof of \eqref{[Equation (68), Y20b]} we see that we still have \eqref{[Equation (68), Y20b]} which leads us to \eqref{[Equation (69), Y20b]} as well as \eqref{[Equation (73), Y20b]}. For all $t \in [0, T_{L}]$, $N \geq 0$ and $k \in\{  0, 1, 2\}$, along with \eqref{simple} this allows us to deduce the estimates of  
\begin{equation}\label{[Equation (143), Y20b]}
\lVert \bar{a}_{\zeta} \rVert_{C_{t}C_{x}^{N}} \overset{\eqref{[Equation (73), Y20b]}}{\leq} m_{L} c_{R}^{\frac{1}{4}} \delta_{q+1}^{\frac{1}{2}} M_{0}(t)^{\frac{1}{2}} l^{- \frac{3}{2} - 4N}, \hspace{2mm} \lVert \bar{a}_{\zeta} \rVert_{C_{t}^{1}C_{x}^{k}} \overset{\eqref{[Equation (73), Y20b]} \eqref{[Equation (124), Y20b]} \eqref{special}}{\leq} m_{L} c_{R}^{\frac{1}{8}} \delta_{q+1}^{\frac{1}{2}} M_{0}(t)^{\frac{1}{2}} l^{- \frac{13}{2} - 4 k}, 
\end{equation}
where we took $c_{R} \ll 1$ to eliminate implicit constants in the second inequality. 

Now we define $w_{q+1}^{(p)}$ and $w_{q+1}^{(c)}$ as in \eqref{[Equation (74), Y20b]} with $a_{\zeta}$ replaced by $\bar{a}_{\zeta}$ from \eqref{[Equation (141), Y20b]} and $M_{0}(t)$ from \eqref{[Equation (125), Y20b]} within the definition of $\rho(\omega, t, x)$, and finally $w_{q+1}^{(t)}$ identically as in \eqref{[Equation (74), Y20b]} with $a_{\zeta}$ from \eqref{[Equation (70), Y20b]}, only with $M_{0}(t)$ from \eqref{[Equation (125), Y20b]}. Then we define the perturbation identically as in \eqref{[Equation (76), Y20b]}: 
\begin{equation}\label{[Equation (145), Y20b]}
w_{q+1} \triangleq w_{q+1}^{(p)} + w_{q+1}^{(c)} + w_{q+1}^{(t)} \text{ and } v_{q+1} \triangleq v_{l}+  w_{q+1}. 
\end{equation} 
We see that as a consequence of \eqref{[Equation (5.9), LQ20]} 
\begin{equation}\label{estimate 52}
 (w_{q+1}^{(p)} + w_{q+1}^{(c)}) (t,x) \overset{\eqref{[Equation (5.9), LQ20]} \eqref{[Equation (141), Y20b]}}{=}  \Upsilon_{l}^{-\frac{1}{2}}(t) \nabla^{\bot}  (\sum_{\zeta \in \Lambda} a_{\zeta}(t,x) \eta_{\zeta}(t,x) \psi_{\zeta}(x)).
\end{equation} 
Consequently, we see that $w_{q+1}$ is both divergence-free and mean-zero. Next, the following estimates for all $t \in [0, T_{L}]$ and $p \in (1,\infty)$ are essentially immediate consequences of \eqref{[Equation (77), Y20b]}, \eqref{[Equation (78a), Y20b]}, \eqref{[Equation (78b), Y20b]}, and \eqref{simple}: 
\begin{subequations}
\begin{align}
& \lVert w_{q+1}^{(p)} \rVert_{C_{t}L_{x}^{2}} \overset{\eqref{[Equation (141), Y20b]}}{\leq} m_{L} \sum_{\zeta \in \Lambda} \lVert a_{\zeta} \mathbb{W}_{\zeta} \rVert_{C_{t}L_{x}^{2}} \overset{\eqref{[Equation (77), Y20b]}}{\lesssim} m_{L} c_{R}^{\frac{1}{4}} \delta_{q+1}^{\frac{1}{2}} M_{0}(t)^{\frac{1}{2}}, \label{[Equation (146a), Y20b]}\\
&  \lVert w_{q+1}^{(p)} \rVert_{C_{t}L_{x}^{p}} \overset{\eqref{[Equation (141), Y20b]}}{\leq} m_{L} \sup_{s \in [0,t]} \sum_{\zeta \in \Lambda} \lVert a_{\zeta} (s) \rVert_{L_{x}^{\infty}} \lVert \mathbb{W}_{\zeta} (s) \rVert_{L_{x}^{p}} 
\overset{\eqref{[Equation (78a), Y20b]}}{\lesssim} m_{L} \delta_{q+1}^{\frac{1}{2}} M_{0}(t)^{\frac{1}{2}} l^{-\frac{3}{2}} r^{1- \frac{2}{p}}, \label{[Equation (146b), Y20b]} \\
& \lVert w_{q+1}^{(c)} \rVert_{C_{t}L_{x}^{p}} \overset{\eqref{[Equation (141), Y20b]}}{\leq} m_{L}  \sum_{\zeta \in \Lambda} \lVert \nabla^{\bot} (a_{\zeta} \eta_{\zeta})  \rVert_{C_{t}L_{x}^{p}} \lVert \psi_{\zeta} \rVert_{L_{x}^{\infty}} \overset{\eqref{[Equation (78b), Y20b]}}{\lesssim} m_{L} \delta_{q+1}^{\frac{1}{2}} M_{0}(t)^{\frac{1}{2}} l^{-\frac{11}{2}} \sigma r^{2- \frac{2}{p}}. \label{[Equation (146c), Y20b]} 
\end{align}
\end{subequations}
Finally, the estimate of $\lVert w_{q+1}^{(t)} \rVert_{C_{t}L_{x}^{p}}$ in \eqref{[Equation (78c), Y20b]} remains valid. Therefore, for all $t\in [0,T_{L}]$ we can estimate from \eqref{[Equation (145), Y20b]} by taking $c_{R} \ll 1$ and $a \in 10 \mathbb{N}$ sufficiently large 
\begin{align}
\lVert w_{q+1} \rVert_{C_{t}L_{x}^{2}} \overset{\eqref{[Equation (78c), Y20b]} \eqref{[Equation (146a), Y20b]} \eqref{[Equation (146c), Y20b]}}{\lesssim}& m_{L} c_{R}^{\frac{1}{4}} \delta_{q+1}^{\frac{1}{2}} M_{0}(t)^{\frac{1}{2}} + m_{L} \delta_{q+1}^{\frac{1}{2}} M_{0}(t)^{\frac{1}{2}} l^{- \frac{11}{2}} \sigma r + \mu^{-1} \delta_{q+1}M_{0}(t) l^{-3} r \nonumber\\
\overset{\eqref{[Equation (57), Y20b]}}{\leq}& m_{L} M_{0}(t)^{\frac{1}{2}} \delta_{q+1}^{\frac{1}{2}} [ \frac{3}{8} + C \lambda_{q+1}^{11\alpha-4 \eta} + C M_{0} (L)^{\frac{1}{2}} \lambda_{q+1}^{6 \alpha - 2 \eta} ] \nonumber\\
\leq& \frac{3m_{L} M_{0}(t)^{\frac{1}{2}} \delta_{q+1}^{\frac{1}{2}}}{4}, \label{[Equation (147), Y20b]}
\end{align}
where the last inequality used the facts that $11 \alpha - 4 \eta < 0$ and $6 \alpha - 2 \eta < 0$, both of which may be readily verified by \eqref{[Equation (2.2), LQ20]}-\eqref{alpha}. It follows from similar computations to \eqref{estimate 53}  that \eqref{[Equation (126a), Y20b]} at level $q + 1$ and \eqref{[Equation (133), Y20b]} can now be verified as follows: 
\begin{align*}
&\lVert v_{q+1} \rVert_{C_{t}L_{x}^{2}} \overset{\eqref{[Equation (145), Y20b]}}{\leq} \lVert v_{l} \rVert_{C_{t}L_{x}^{2}} + \lVert w_{q+1} \rVert_{C_{t}L_{x}^{2}} 
\overset{\eqref{[Equation (140b), Y20b]}\eqref{[Equation (147), Y20b]}}{\leq} m_{L} M_{0}(t)^{\frac{1}{2}} (1+ \sum_{1\leq \iota \leq q+1} \delta_{\iota}^{\frac{1}{2}}), \\
& \lVert v_{q+1}(t) - v_{q}(t) \rVert_{L_{x}^{2}} \overset{\eqref{[Equation (140a), Y20b]} \eqref{[Equation (145), Y20b]}}{\leq} \lVert w_{q+1} \rVert_{C_{t}L_{x}^{2}} + \frac{m_{L}}{4} M_{0}(t)^{\frac{1}{2}} \delta_{q+1}^{\frac{1}{2}} 
\overset{\eqref{[Equation (147), Y20b]}}{\leq} m_{L} M_{0}(t)^{\frac{1}{2}} \delta_{q+1}^{\frac{1}{2}}.
\end{align*} 
Next, we estimate for all $t \in [0, T_{L}]$  
\begin{subequations}\label{[Equation (148), Y20b]}
\begin{align}
\lVert w_{q+1}^{(p)} \rVert_{C_{t,x}^{1}} \leq&  \sum_{\zeta \in \Lambda} \lVert \bar{a}_{\zeta}\rVert_{C_{t,x}^{1}} \lVert \mathbb{W}_{\zeta} \rVert_{C_{t,x}^{1}} \label{[Equation (148a), Y20b]} \\
&\overset{\eqref{[Equation (4.22), LQ20]} \eqref{[Equation (143), Y20b]}}{\lesssim} (m_{L} \delta_{q+1}^{\frac{1}{2}} M_{0}(t)^{\frac{1}{2}} l^{-\frac{13}{2}} ) \lambda_{q+1} \sigma \mu r^{2} \leq m_{L} M_{0}(t)^{\frac{1}{2}} l^{- \frac{13}{2}} \lambda_{q+1} \sigma \mu r^{2}, \nonumber\\
 \lVert w_{q+1}^{(c)} \rVert_{C_{t,x}^{1}} \leq&  \sum_{\zeta \in \Lambda} \lVert  \nabla^{\bot} (\bar{a}_{\zeta} \eta_{\zeta}) \psi_{\zeta} \rVert_{C_{t,x}^{1}} \overset{\eqref{[Equations (4.4) and (4.5), LQ20]} \eqref{[Equation (4.23), LQ20]} \eqref{[Equation (143), Y20b]}}{\lesssim} m_{L} \delta_{q+1}^{\frac{1}{2}} M_{0}(t)^{\frac{1}{2}} \lambda_{q+1}^{1- 6 \eta} \label{[Equation (148b), Y20b]}\\
&  \times [ l^{-\frac{21}{2}} \lambda_{q+1}^{-1} +  l^{-\frac{11}{2}} \lambda_{q+1}^{1- 8 \eta} + l^{-\frac{13}{2}} \lambda_{q+1}^{-4 \eta} + l^{-\frac{3}{2}} \lambda_{q+1}^{2 - 12 \eta}] \lesssim m_{L} \delta_{q+1}^{\frac{1}{2}} M_{0}(t)^{\frac{1}{2}} \lambda_{q+1}^{3 - 18 \eta} l^{-\frac{3}{2}},\nonumber 
\end{align} 
\end{subequations}
where we used $\delta_{q+1}^{\frac{1}{2}}$ to eliminate implicit constant in \eqref{[Equation (148a), Y20b]}. On the other hand, the estimate of $ \lVert w_{q+1}^{(t)} \rVert_{C_{t,x}^{1}}$ from \eqref{[Equation (83), Y20b]} remains applicable for us. We may now verify \eqref{[Equation (126b), Y20b]} at level $q+1$ as follows. For any $t \in [0, T_{L}]$
\begin{align}
& \lVert v_{q+1} \rVert_{C_{t,x}^{1}} \overset{\eqref{[Equation (140c), Y20b]} \eqref{[Equation (145), Y20b]}}{\leq} l^{-1} m_{L} M_{0}(t)^{\frac{1}{2}} \lambda_{q+1}^{-\alpha} + \lVert w_{q+1}^{(p)} \rVert_{C_{t,x}^{1}} + \lVert w_{q+1}^{(c)} \rVert_{C_{t,x}^{1}} + \lVert w_{q+1}^{(t)} \rVert_{C_{t,x}^{1}} \\
&\leq m_{L} M_{0}(t)^{\frac{1}{2}}  [ l^{-1} \lambda_{q+1}^{-\alpha} + C  \lambda_{q+1}^{13 \alpha + 3 - 14 \eta} + C \lambda_{q+1}^{3- 18 \eta} l^{- \frac{3}{2}} + C \lambda_{q+1}^{3- 16 \eta + \alpha} M_{0}(t)^{\frac{1}{2}} l^{-3} ] \leq m_{L} M_{0}(t)^{\frac{1}{2}} \lambda_{q+1}^{4} \nonumber 
\end{align} 
where the last inequality used \eqref{estimate 38} and that $13 \alpha + 3 - 14 \eta < 4$ which can be readily verified by \eqref{[Equation (2.2), LQ20]}-\eqref{alpha}. Next, as a consequence of \eqref{[Equation (5.9), LQ20]} we have the identity of 
\begin{equation}\label{estimate 54}
(w_{q+1}^{(p)} + w_{q+1}^{(c)})(t,x) \overset{\eqref{[Equation (141), Y20b]}}{=} \Upsilon_{l}^{ -\frac{1}{2}}(t) \nabla^{\bot} (\sum_{\zeta \in \Lambda} a_{\zeta}(t,x) \eta_{\zeta}(t,x) \psi_{\zeta}(x)). 
\end{equation} 
This allows us to estimate for all $t \in [0, T_{L}]$ and $p \in (1, \infty)$, by utilizing \eqref{[Equation (88a), Y20b]} and \eqref{simple} 
\begin{equation}\label{estimate 57}
 \lVert w_{q+1}^{(p)} + w_{q+1}^{(c)} \rVert_{C_{t}W_{x}^{1,p}} \leq \lVert \Upsilon_{l}^{-\frac{1}{2}} \rVert_{C_{t}} \lVert \nabla^{\bot} \sum_{\zeta \in \Lambda}  a_{\zeta} \eta_{\zeta}\psi_{\zeta} \rVert_{C_{t}W_{x}^{1,p}} \lesssim m_{L} \delta_{q+1}^{\frac{1}{2}} M_{0}(t)^{\frac{1}{2}} r^{1 - \frac{2}{p}} l^{- \frac{3}{2}} \lambda_{q+1}.
\end{equation}
On the other hand, the estimate of $\lVert w_{q+1}^{(t)} \rVert_{C_{t}W_{x}^{1,p}}$ from \eqref{[Equation (88b), Y20b]} remains applicable for us. 

\subsubsection{Reynolds stress}

We can choose the same $p^{\ast}$ from \eqref{p ast} and compute from \eqref{[Equation (122), Y20b]}, \eqref{[Equation (138), Y20b]}, and \eqref{[Equation (145), Y20b]}  
\begin{align}
& \text{div}\mathring{R}_{q+1} - \nabla p_{q+1} \label{[Equation (151), Y20b]}\\
=& \underbrace{\frac{1}{2} w_{q+1} + (-\Delta)^{m} w_{q+1} + \partial_{t} (w_{q+1}^{(p)} + w_{q+1}^{(c)}) + \Upsilon_{l} \text{div} (v_{l} \otimes w_{q+1} + w_{q+1} \otimes v_{l})}_{\text{div}(R_{\text{lin}} ) + \nabla p_{\text{lin}}} \nonumber\\
&+ \underbrace{\Upsilon_{l} \text{div}((w_{q+1}^{(c)} + w_{q+1}^{(t)}) \otimes w_{q+1} + w_{q+1}^{(p)} \otimes (w_{q+1}^{(c)} + w_{q+1}^{(t)}))}_{\text{div} (R_{\text{cor}}) + \nabla p_{\text{cor}}} \nonumber\\
& + \underbrace{\text{div}(\Upsilon_{l} w_{q+1}^{(p)} \otimes w_{q+1}^{(p)} + \mathring{R}_{l}) + \partial_{t}w_{q+1}^{(t)}}_{\text{div} (R_{\text{osc}}) + \nabla p_{\text{osc}}} + \underbrace{(\Upsilon - \Upsilon_{l}) \text{div}(v_{q+1} \otimes v_{q+1})}_{\text{div} (R_{\text{com2}}) + \nabla p_{\text{com2}}} + \text{div}(R_{\text{com1}} ) - \nabla p_{l} \nonumber
\end{align} 
where 
\begin{subequations}
\begin{align}
R_{\text{lin}} \triangleq& R_{\text{linear}} \nonumber\\
\triangleq& \mathcal{R} ( \frac{1}{2} w_{q+1} + (-\Delta)^{m} w_{q+1} + \partial_{t} (w_{q+1}^{(p)} + w_{q+1}^{(c)})) + \Upsilon_{l} (v_{l} \mathring{\otimes} w_{q+1} + w_{q+1} \mathring{\otimes} v_{l}), \label{[Equation (152a), Y20b]} \\
p_{\text{lin}} \triangleq& p_{\text{linear}} \triangleq \Upsilon_{l} (v_{l} \cdot w_{q+1}), \label{[Equation (152b), Y20b]} \\
R_{\text{cor}} \triangleq& R_{\text{corrector}} \triangleq \Upsilon_{l} ((w_{q+1}^{(c)} + w_{q+1}^{(t)}) \mathring{\otimes} w_{q+1} + w_{q+1}^{(p)} \mathring{\otimes} (w_{q+1}^{(c)} + w_{q+1}^{(t)})), \label{[Equation (152c), Y20b]} \\
p_{\text{cor}} \triangleq& p_{\text{corrector}} \triangleq \frac{\Upsilon_{l}}{2} ((w_{q+1}^{(c)} + w_{q+1}^{(t)}) \cdot w_{q+1} + w_{q+1}^{(p)} \cdot (w_{q+1}^{(c)} + w_{q+1}^{(t)})), \label{[Equation (152d), Y20b]} \\
R_{\text{com2}} \triangleq& R_{\text{commutator2}} \triangleq (\Upsilon - \Upsilon_{l}) (v_{q+1} \mathring{\otimes} v_{q+1}), \label{[Equation (152g), Y20b]} \\
p_{\text{com2}} \triangleq& p_{\text{commutator2}} \triangleq \frac{\Upsilon - \Upsilon_{l}}{2} \lvert v_{q+1} \rvert^{2}.\label{[Equation (152h), Y20b]}
\end{align}
\end{subequations}
Concerning $R_{\text{osc}}$ and $p_{\text{osc}}$ we have 
\begin{align}
& \text{div} ( \Upsilon_{l} w_{q+1}^{(p)} \otimes w_{q+1}^{(p)} + \mathring{R}_{l}) + \partial_{t} w_{q+1}^{(t)} \label{estimate 37}\\
\overset{\eqref{[Equation (141), Y20b]}}{=}& \text{div}( (\sum_{\zeta \in \Lambda} a_{\zeta} \mathbb{W}_{\zeta} ) \otimes ( \sum_{\zeta' \in \Lambda} a_{\zeta'} \mathbb{W}_{\zeta'} ) + \mathring{R}_{l} ) + \partial_{t} w_{q+1}^{(t)} \nonumber \\
\overset{\eqref{estimate 26}}{=}&  \frac{1}{2} \sum_{\zeta, \zeta' \in \Lambda} \mathcal{E}_{\zeta, \zeta', 1} + \frac{1}{2} \sum_{\zeta, \zeta' \in \Lambda} \sum_{k= 1, 3, 4} \mathcal{E}_{\zeta, \zeta', 2, k} + A_{2} + A_{3} \nonumber \\
&+ \nabla [ \frac{1}{2} \lvert  \sum_{\zeta \in \Lambda} a_{\zeta} \mathbb{W}_{\zeta} \rvert^{2} + \frac{1}{2} \sum_{\zeta, \zeta' \in \Lambda} \mathbb{P}_{\neq 0} (a_{\zeta} a_{\zeta'} \mathbb{P}_{\geq \frac{\lambda_{q+1}}{10}} (\eta_{\zeta} \eta_{\zeta'} \lambda_{q+1}^{2} \psi_{\zeta} \psi_{\zeta'})) \nonumber \\
& \hspace{5mm} + \frac{1}{2} \sum_{\zeta \in \Lambda} a_{\zeta}^{2} \mathbb{P}_{\geq \frac{\lambda_{q+1} \sigma}{2}} \eta_{\zeta}^{2} - \Delta^{-1} \nabla\cdot \mu^{-1} (\sum_{\zeta \in \Lambda^{+}} - \sum_{\zeta \in \Lambda^{-}}) \mathbb{P}_{\neq 0} \partial_{t} (a_{\zeta}^{2} \mathbb{P}_{\neq 0} \eta_{\zeta}^{2} \zeta )]. \nonumber 
\end{align} 
Therefore, we can define similarly to \eqref{estimate 27} - \eqref{estimate 28}
\begin{subequations}
\begin{align}
R_{\text{osc}} \triangleq& R_{\text{oscillation}} \triangleq \mathcal{R} (\frac{1}{2} \sum_{\zeta, \zeta' \in \Lambda} \mathcal{E}_{\zeta, \zeta', 1} + \frac{1}{2} \sum_{\zeta, \zeta' \in \Lambda} \sum_{k= 1, 3, 4} \mathcal{E}_{\zeta, \zeta', 2, k} + A_{2} + A_{3} ), \\
p_{\text{osc}} \triangleq& p_{\text{oscillation}} \triangleq \frac{1}{2} \lvert \sum_{\zeta \in \Lambda} a_{\zeta} \mathbb{W}_{\zeta} \rvert^{2} + \frac{1}{2} \sum_{\zeta, \zeta' \in \Lambda} \mathbb{P}_{\neq 0} (a_{\zeta} a_{\zeta'} \mathbb{P}_{\geq \frac{\lambda_{q+1} }{10}} (\eta_{\zeta} \eta_{\zeta'} \lambda_{q+1}^{2} \psi_{\zeta} \psi_{\zeta'})) 1_{\zeta + \zeta' \neq 0}  \nonumber\\
& \hspace{12mm} + \frac{1}{2} \sum_{\zeta \in \Lambda} a_{\zeta}^{2} \mathbb{P}_{\geq \frac{\lambda_{q+1} \sigma}{2}} \eta_{\zeta}^{2} - \Delta^{-1} \nabla\cdot \mu^{-1} (\sum_{\zeta \in \Lambda^{+}} - \sum_{\zeta \in \Lambda^{-}}) \mathbb{P}_{\neq 0} \partial_{t} (a_{\zeta}^{2} \mathbb{P}_{\neq 0} \eta_{\zeta}^{2} \zeta ) 
\end{align} 
\end{subequations} 
and claim the same bound as in \eqref{estimate 29} for $R_{\text{osc}}$. Thus, let us define formally  
\begin{equation}
p_{q+1} \triangleq - p_{\text{lin}}  - p_{\text{cor}} - p_{\text{osc}}  - p_{\text{com2}} + p_{l} \text{ and } \mathring{R}_{q+1} \triangleq R_{\text{lin}} + R_{\text{cor}}  + R_{\text{osc}} + R_{\text{com2}} + R_{\text{com1}}.
\end{equation}

Now we compute for all $t\in [0, T_{L}]$ from \eqref{[Equation (152a), Y20b]} 
\begin{align}
 \lVert R_{\text{lin}} \rVert_{C_{t}L_{x}^{p^{\ast}}} \lesssim& \lVert w_{q+1} \rVert_{C_{t}L_{x}^{p^{\ast}}} + \lVert \mathcal{R} ( -\Delta)^{m} w_{q+1} \rVert_{C_{t}L_{x}^{p^{\ast}}} \nonumber  \\
&+ \lVert \mathcal{R} \partial_{t} (w_{q+1}^{(p)} + w_{q+1}^{(c)}) \rVert_{C_{t} L_{x}^{p^{\ast}}} + \lVert \Upsilon_{l} (v_{l} \mathring{\otimes} w_{q+1} + w_{q+1} \mathring{\otimes} v_{l}) \rVert_{C_{t} L_{x}^{p^{\ast}}}. \label{estimate 30} 
\end{align} 
First, by the estimate of $m_{L} \leq M_{0}(t)^{\frac{1}{2}}$ from \eqref{special} we can compute  from \eqref{[Equation (145), Y20b]} for all $t\in [0, T_{L}]$
\begin{align}
\lVert w_{q+1} \rVert_{C_{t}L_{x}^{p^{\ast}}}  \overset{\eqref{[Equation (78c), Y20b]}\eqref{[Equation (146b), Y20b]} \eqref{[Equation (146c), Y20b]} }{\lesssim}& m_{L} \delta_{q+1}^{\frac{1}{2}} M_{0}(t)^{\frac{1}{2}} l^{-\frac{3}{2}} r^{1- \frac{2}{p^{\ast}}} + m_{L} \delta_{q+1}^{\frac{1}{2}} M_{0}(t)^{\frac{1}{2}} l^{- \frac{11}{2}} \sigma r^{2- \frac{2}{p^{\ast}}} \label{another special} \\
& + \mu^{-1} \delta_{q+1} M_{0}(t) l^{-3} r^{2- \frac{2}{p^{\ast}}} \overset{\eqref{[Equation (2.2), LQ20]} \eqref{[Equation (2.3), LQ20]} \eqref{alpha} \eqref{[Equation (57), Y20b]}}{\lesssim} \delta_{q+1}^{\frac{1}{2}} M_{0}(t)^{\frac{1}{2}} r^{1- \frac{2}{p^{\ast}}} m_{L} l^{-\frac{3}{2}}.  \nonumber 
\end{align}
By Gagliardo-Nirenberg's inequality this also leads us to 
\begin{align}
 \lVert \mathcal{R} (-\Delta)^{m} w_{q+1} \rVert_{C_{t}L_{x}^{p^{\ast}}}\lesssim&  [ \delta_{q+1}^{\frac{1}{2}} M_{0}(t)^{\frac{1}{2}} r^{1- \frac{2}{p^{\ast}}} m_{L} l^{-\frac{3}{2}} ]^{1- m^{\ast}}  [ \lVert w_{q+1}^{(p)} + w_{q+1}^{(c)} \rVert_{C_{t} W_{x}^{1,p^{\ast}}} + \lVert w_{q+1}^{(t)} \rVert_{C_{t}W_{x}^{1, p^{\ast}}}]^{m^{\ast}}  \nonumber\\
& \hspace{30mm} \overset{\eqref{[Equation (88b), Y20b]} \eqref{estimate 57}}{\lesssim}  \delta_{q+1}^{\frac{1}{2}} M_{0}(t)^{\frac{1}{2}} r^{1- \frac{2}{p^{\ast}}} m_{L}l^{-\frac{3}{2}} \lambda_{q+1}^{m^{\ast}}. \label{estimate 31}
\end{align}
Second,  for all $t\in [0, T_{L}]$  we can make use of \eqref{estimate 10} and \eqref{simple} and estimate 
\begin{align}
& \lVert \mathcal{R} \partial_{t} (w_{q+1}^{(p)} + w_{q+1}^{(c)} ) \rVert_{C_{t}L_{x}^{p^{\ast}}} \label{estimate 32}\\
\overset{\eqref{estimate 52}}{\lesssim}& \sum_{\zeta \in \Lambda} \lVert \Upsilon_{l}^{-\frac{1}{2}} \rVert_{C_{t}}^{3} \lVert \partial_{t} \Upsilon_{l}\rVert_{C_{t}} \lVert a_{\zeta} \rVert_{C_{t}C_{x}} \lVert \eta_{\zeta} \rVert_{C_{t}L_{x}^{p^{\ast}}} \lVert \psi_{\zeta} \rVert_{C_{x}} + \lVert \Upsilon_{l}^{-\frac{1}{2}} \rVert_{C_{t}} \lVert \partial_{t} (a_{\zeta} \eta_{\zeta})\psi_{\zeta} \rVert_{C_{t}C_{x}} \nonumber \\
\overset{\eqref{[Equation (4.3), LQ20]} \eqref{[Equation (4.23), LQ20]} \eqref{[Equation (73), Y20b]} \eqref{estimate 10}}{\lesssim}& m_{L}^{3} l^{-1} \lVert \Upsilon \rVert_{C_{t}} \delta_{q+1}^{\frac{1}{2}} M_{0}(t)^{\frac{1}{2}} l^{-\frac{3}{2}} r^{1- \frac{2}{p^{\ast}}} \lambda_{q+1}^{-1}  \nonumber\\
& \hspace{15mm} + m_{L} \delta_{q+1}^{\frac{1}{2}} M_{0}(t)^{\frac{1}{2}} r^{1- \frac{2}{p^{\ast}}} l^{-\frac{3}{2}} \lambda_{q+1}^{1- 8 \eta}  \overset{\eqref{[Equation (124), Y20b]}}{\lesssim} m_{L} l^{-\frac{3}{2}} \delta_{q+1}^{\frac{1}{2}}M_{0}(t)^{\frac{1}{2}} r^{1- \frac{2}{p^{\ast}}} \lambda_{q+1}^{1- 8 \eta}.  \nonumber
\end{align}
Third, we can estimate  for all $t\in [0, T_{L}]$  
\begin{align}
\lVert \Upsilon_{l} (v_{l} \mathring{\otimes} w_{q+1} + w_{q+1} \mathring{\otimes} v_{l} ) \rVert_{C_{t}L_{x}^{p^{\ast}}} 
\lesssim& \lVert \Upsilon \rVert_{C_{t}} \lVert v_{q} \rVert_{C_{t,x}^{1}} \lVert w_{q+1} \rVert_{C_{t}L_{x}^{p^{\ast}}} \nonumber\\
\overset{\eqref{[Equation (124), Y20b]} \eqref{[Equation (126b), Y20b]} \eqref{another special}}{\lesssim}& m_{L}^{4} M_{0}(t) \lambda_{q}^{4} r^{1- \frac{2}{p^{\ast}}} \delta_{q+1}^{\frac{1}{2}} l^{-\frac{3}{2}}. \label{estimate 33}
\end{align}
Hence, applying \eqref{another special}-\eqref{estimate 33} to \eqref{estimate 30} and taking $a \in 10 \mathbb{N}$ sufficiently large give us 
\begin{align}
 \lVert R_{\text{lin}} \rVert_{C_{t}L_{x}^{p^{\ast}}} 
\overset{\eqref{estimate 22} \eqref{[Equation (57), Y20b]}}{\lesssim}& M_{0}(t) \delta_{q+2} [ \lambda_{q+2}^{2\beta} (\lambda_{q+1}^{1- 6 \eta})^{1- \frac{2}{p^{\ast}}} m_{L} \lambda_{q+1}^{3\alpha} \lambda_{q+1}^{m^{\ast}}  \nonumber \\ 
& \hspace{5mm} + \lambda_{q+2}^{2\beta} m_{L} \lambda_{q+1}^{3 \alpha} (\lambda_{q+1}^{1- 6 \eta})^{1- \frac{2}{p^{\ast}}} \lambda_{q+1}^{1- 8 \eta} + \lambda_{q+2}^{2\beta} m_{L}^{4} \lambda_{q+1}^{\frac{\alpha}{4}} (\lambda_{q+1}^{1- 6 \eta})^{1- \frac{2}{p^{\ast}}} \lambda_{q+1}^{3\alpha}]  \nonumber \\
\overset{\eqref{estimate 21} \eqref{p ast}}{\lesssim}& M_{0}(t) \delta_{q+2} [m_{L} \lambda_{q+1}^{- \frac{275\alpha}{8}} + m_{L}^{4} \lambda_{q+1}^{ \frac{ - 273 \alpha - 8 + 64 \eta}{8}}]    \leq (2\pi)^{-2 ( \frac{p^{\ast} - 1}{p^{\ast}} )} \frac{M_{0}(t) c_{R} \delta_{q+2}}{5} \label{[Equation (155), Y20b]}
\end{align}
where we used the facts that $2\beta b < \frac{\alpha}{8}$ due to \eqref{estimate 21} and $- 273 \alpha - 8 + 64 \eta \leq - 273 \alpha - 8m^{\ast} < 0$ due to \eqref{[Equation (2.3), LQ20]}. 

Next, for all $t\in [0, T_{L}]$ we estimate from \eqref{[Equation (152c), Y20b]} by taking $a \in 10 \mathbb{N}$ sufficiently large 
\begin{align}
& \lVert R_{\text{cor}} \rVert_{C_{t}L_{x}^{p^{\ast}}} \label{[Equation (157), Y20b]}\\
\lesssim& \lVert \Upsilon_{l} \rVert_{C_{t}} ( \lVert w_{q+1}^{(c)} \rVert_{C_{t}L_{x}^{2p^{\ast}}} + \lVert w_{q+1}^{(t)} \rVert_{C_{t}L_{x}^{2p^{\ast}}} ) ( \lVert w_{q+1}^{(c)} \rVert_{C_{t}L_{x}^{2p^{\ast}}} + \lVert w_{q+1}^{(t)} \rVert_{C_{t}L_{x}^{2p^{\ast}}} + \lVert w_{q+1}^{(p)} \rVert_{C_{t}L_{x}^{2p^{\ast}}} )  \nonumber \\
\overset{\eqref{estimate 22} \eqref{[Equation (57), Y20b]} \eqref{[Equation (124), Y20b]}}{\lesssim}& m_{L}^{2} M_{0}(t) [ m_{L} \lambda_{q+1}^{- \frac{31 \alpha}{4} - 3 \eta} + \lambda_{q+1}^{- \eta - 3 4 \alpha} ] [ m_{L} \lambda_{q+1}^{- \frac{31 \alpha}{4} - 3 \eta} + \lambda_{q+1}^{- \eta - 34 \alpha} + m_{L} \lambda_{q+1}^{\eta - \frac{63\alpha}{4}}]   \nonumber \\
\lesssim& M_{0}(t) \delta_{q+2} \lambda_{q+2}^{2\beta} m_{L}^{3} \lambda_{q+1}^{- 34 \alpha - \frac{63\alpha}{4}} 
 \overset{\eqref{estimate 21}}{\lesssim} M_{0}(t) \delta_{q+2} m_{L}^{3} \lambda_{q+1}^{- 34 \alpha - \frac{125 \alpha}{8}}   \leq  (2\pi)^{-2 (\frac{p^{\ast} - 1}{p^{\ast}})} \frac{M_{0}(t) c_{R} \delta_{q+2}}{5}. \nonumber 
\end{align} 

Next, for all $t\in [0, T_{L}]$ we estimate using \eqref{simple}, $\lambda_{q}^{4} l^{\frac{1}{2} - 2 \delta} \lesssim \delta_{q+2} \lambda_{q}^{-\frac{8}{3}}$ from \eqref{[Equation (105), Y20b]}, and taking $a \in 10 \mathbb{N}$ sufficiently large 
\begin{align} 
\lVert R_{\text{com1}} \rVert_{C_{t}L_{x}^{1}} \overset{\eqref{estimate 56} \eqref{[Equation (124), Y20b]} \eqref{[Equation (126), Y20b]} }{\lesssim}& m_{L}^{4} M_{0}(t) l^{\frac{1}{2} - 2 \delta} \lambda_{q}^{4} \overset{\eqref{[Equation (105), Y20b]}}{\lesssim} M_{0} (t) \delta_{q+2} m_{L}^{4} \lambda_{q}^{- \frac{8}{3}}  \leq \frac{c_{R} M_{0}(t) \delta_{q+2}}{5}. \label{[Equation (158), Y20b]}
\end{align}

Finally, using $\lvert \Upsilon_{l}(t) - \Upsilon (t) \rvert \overset{\eqref{[Equation (124), Y20b]}}{\lesssim} l^{\frac{1}{2} - 2 \delta} m_{L}^{2}$,  and $\lambda_{q}^{4} l^{\frac{1}{2} - 2 \delta} \lesssim \delta_{q+2} \lambda_{q}^{-\frac{8}{3}}$ from \eqref{[Equation (105), Y20b]} again, and taking $a \in 10 \mathbb{N}$ sufficiently large we obtain for all $t\in [0, T_{L}]$ 
\begin{equation}\label{[Equation (159), Y20b]}
\lVert R_{\text{com2}} \rVert_{C_{t}L_{x}^{1}}  \overset{\eqref{[Equation (152g), Y20b]}}{\leq} \lVert \Upsilon_{l} - \Upsilon \rVert_{C_{t}} \lVert v_{q+1} \rVert_{C_{t}L_{x}^{2}}^{2} \overset{\eqref{[Equation (140b), Y20b]} \eqref{[Equation (147), Y20b]}}{\lesssim}  l^{\frac{1}{2} - 2 \delta} m_{L}^{4} M_{0}(t) \leq \frac{M_{0}(t) c_{R} \delta_{q+2}}{5}. 
\end{equation}
Therefore, considering \eqref{[Equation (155), Y20b]}, \eqref{[Equation (157), Y20b]},  \eqref{estimate 29}, \eqref{[Equation (158), Y20b]}, and \eqref{[Equation (159), Y20b]}, we are able to conclude that $\lVert \mathring{R}_{q+1} \rVert_{C_{t}L_{x}^{1}} \leq M_{0}(t) c_{R} \delta_{q+2}$ identically as we did in \eqref{[Equation (107), Y20b]}. This verifies \eqref{[Equation (126c), Y20b]} at level $q+ 1$. 

Finally, essentially identical arguments in the proof of Proposition \ref{[Proposition 4.8, Y20b]} shows that $(v_{q}, \mathring{R}_{q})$ being $(\mathcal{F}_{t})_{t\geq 0}$-adapted leads to $(v_{q+1}, \mathring{R}_{q+1})$ being $(\mathcal{F}_{t})_{t\geq 0}$-adapted, and that $(v_{q}, \mathring{R}_{q})(0,x)$ being deterministic implies $(v_{q+1}, \mathring{R}_{q+1})(0,x)$ being deterministic. 

\section{Appendix}\label{Appendix}

\subsection{Past results}
We collect results from previous works which were used in the proofs of Theorems \ref{[Theorem 2.1, Y20b]}-\ref{[Theorem 2.4, Y20b]}.

\begin{lemma}\label{[Definition 9, Lemma 10, CDS12]}
\rm{ (\cite[Definition 9, Lemma 10]{CDS12}, also \cite[Definition 7.1, Lemmas 7.2 and 7.3]{LQ20})} For $f \in C(\mathbb{T}^{2})$, set 
\begin{equation}\label{estimate 2}
\mathcal{R} f \triangleq \nabla g + (\nabla g)^{T} - (\nabla\cdot g) \text{Id}, 
\end{equation} 
where $\Delta g = f - \fint_{\mathbb{T}^{2}} fdx$ and $\fint_{\mathbb{T}^{2}} g  dx= 0$. Then for any $f \in C(\mathbb{T}^{2})$ such that $\fint_{\mathbb{T}^{2}} f dx = 0$, $\mathcal{R}f(x)$ is a trace-free symmetric matrix for all $x \in \mathbb{T}^{2}$.  Moreover, $\nabla\cdot \mathcal{R} f = f$ and $\fint_{\mathbb{T}^{2}} \mathcal{R} f(x) dx = 0$. When $f$ is not mean-zero, we overload the notation and denote by $\mathcal{R} f \triangleq \mathcal{R} (f - \int_{\mathbb{T}^{2}} fdx)$.  Finally, for all $p \in (1,\infty)$, $\lVert \mathcal{R} \rVert_{L_{x}^{p} \mapsto W_{x}^{1,p}} \lesssim 1, \lVert \mathcal{R} \rVert_{C_{x} \mapsto C_{x}} \lesssim 1$, and $\lVert \mathcal{R}  f \rVert_{L_{x}^{p}} \lesssim \lVert (-\Delta)^{-\frac{1}{2}} f \rVert_{L_{x}^{p}}$.  
\end{lemma}

\begin{lemma}\label{[Lemma 6.2, LQ20]}
\rm{(\cite[Lemma 6.2]{LQ20})} Let $f, g \in C^{\infty} (\mathbb{T}^{2})$ where $g$ is also $(\mathbb{T}/\kappa)^{2}$-periodic for some $\kappa \in \mathbb{N}$. Then there exists a constant $C \geq 0$ such that 
\begin{equation}\label{[Equation (6.3), LQ20]}
\lVert fg \rVert_{L_{x}^{2}} \leq \lVert f \rVert_{L_{x}^{2}} \lVert g \rVert_{L_{x}^{2}} + C \kappa^{-\frac{1}{2}} \lVert f \rVert_{C_{x}^{1}} \lVert g \rVert_{L_{x}^{2}}. 
\end{equation} 
\end{lemma}

\begin{lemma}\label{[Lemma 7.4, LQ20]}
\rm{(\cite[Lemma 7.4]{LQ20})} For any given $p \in (1,\infty), \lambda \in \mathbb{N}, a \in C^{2}(\mathbb{T}^{2})$,  and $f \in L^{p}(\mathbb{T}^{2})$, 
\begin{equation}\label{estimate 50}
\lVert (-\Delta)^{-\frac{1}{2}} \mathbb{P}_{\neq 0} (a \mathbb{P}_{\geq \lambda} f) \rVert_{L_{x}^{p}} \lesssim \lambda^{-1} \lVert a \rVert_{C_{x}^{2}} \lVert f \rVert_{L_{x}^{p}}. 
\end{equation} 
\end{lemma}
 
\subsection{Continuation of the proof of Proposition \ref{[Proposition 4.1, Y20b]}}
First, the proof of the following result from \cite{HZZ19} in case $x \in \mathbb{T}^{3}$ goes through verbatim in case $x \in \mathbb{T}^{2}$. 
\begin{lemma}\label{[Lemma A.1, HZZ19]} 
\rm{ (\cite[Lemma A.1]{HZZ19})} Let $\{(s_{n}, \xi_{n}) \}_{n\in\mathbb{N}} \subset [0,\infty) \times L_{\sigma}^{2}$ be a family such that $\lim_{n\to\infty} \lVert (s_{n}, \xi_{n}) - (s, \xi^{\text{in}}) \rVert_{\mathbb{R} \times L_{x}^{2}} = 0$ and $\{P_{n}\}_{n\in\mathbb{N}}$ be a family of probability measures on $\Omega_{0}$ satisfying for all $n \in \mathbb{N}$, $P_{n}(\{\xi(t) = \xi_{n} \hspace{1mm} \forall \hspace{1mm} t \in [0, s_{n}]\}) = 1$ and for some $\gamma, \kappa > 0$ and any $T > 0$,  
\begin{equation}\label{[Equation (177), Y20b]} 
\sup_{n\in\mathbb{N}} \mathbb{E}^{P_{n}} [ \lVert \xi \rVert_{C([0,T]; L_{x}^{2})} + \sup_{r, t \in [0,T]: \hspace{0.5mm}  r \neq t} \frac{ \lVert \xi(t) - \xi(r) \rVert_{H_{x}^{-3}}}{\lvert t- r \rvert^{\kappa}} + \lVert \xi \rVert_{L^{2}([s_{n}, T]; H_{x}^{\gamma})}^{2}] < \infty.  
\end{equation} 
Then $\{P_{n}\}_{n\in\mathbb{N}}$ is tight in $\mathbb{M} \triangleq C_{\text{loc}} ([0,\infty); H^{-3}(\mathbb{T}^{2})) \cap L_{\text{loc}}^{2}([0,\infty); L_{\sigma}^{2})$. 
\end{lemma}
Now we fix $\{P_{n}\} \subset \mathcal{C} ( s_{n}, \xi_{n}, \{C_{t,q}\}_{q\in\mathbb{N}, t \geq s_{n}})$ and will show that it is tight in $\mathbb{M}$ by relying on Lemma \ref{[Lemma A.1, HZZ19]}.  We define $J(\xi) \triangleq - \mathbb{P} \text{div} (\xi \otimes \xi) - (-\Delta)^{m} \xi$. By definition of $\mathcal{C} (s_{n}, \xi_{n}, \{C_{t,q}\}_{q \in \mathbb{N}, t \geq s_{n}} )$ and (M2) of Definition \ref{[Definition 4.1, Y20b]}, we know that for all $n \in \mathbb{N}$ and $t \in [s_{n}, \infty)$  
\begin{equation}\label{[Equation (178), Y20b]} 
\xi(t) = \xi_{n} + \int_{s_{n}}^{t} J(\xi(r)) dr + M_{t, s_{n}}^{\xi} \hspace{3mm} P_{n}\text{-a.s.,}
\end{equation} 
where the map $t \mapsto M_{t, s_{n}}^{\xi, i} \triangleq \langle M_{t, s_{n}}^{\xi}, \mathfrak{g}_{i} \rangle$ for $\xi \in \Omega_{0}$ and $\mathfrak{g}_{i} \in C^{\infty} (\mathbb{T}^{2}) \cap L_{\sigma}^{2}$ is a continuous, square-integrable martingale w.r.t. $(\mathcal{B}_{t})_{t \geq s_{n}}$ such that $\langle \langle M_{t, s_{n}}^{\xi, i} \rangle \rangle = \int_{s_{n}}^{t} \lVert G(\xi(r))^{\ast} \mathfrak{g}_{i} \rVert_{U}^{2} dr$. We can compute for any $p \in (1,\infty)$, 
\begin{align*}
& \mathbb{E}^{P_{n}} [ \sup_{r, t \in [s_{n}, T]: \hspace{0.5mm} r \neq t} \frac{ \lVert \int_{r}^{t} J(\xi(l)) dl \rVert_{H_{x}^{-3}}^{p} }{\lvert t-r \rvert^{p-1}} ] \leq \mathbb{E}^{P_{n}} [ \int_{s_{n}}^{T} (\lVert \xi \otimes \xi \rVert_{H_{x}^{-2}} + \lVert \xi \rVert_{H_{x}^{2m-3}})^{p} dl]
\end{align*} 
by H$\ddot{\mathrm{o}}$lder's inequality where $\lVert \xi \otimes \xi \rVert_{H_{x}^{-2}}  \lesssim \lVert \xi \rVert_{L_{x}^{2}}^{2}$ and $\lVert \xi \rVert_{H_{x}^{2m-3}} \lesssim 1 + \lVert \xi \rVert_{L_{x}^{2}}^{2}$ because $m \in (0,1)$. Therefore, 
\begin{align}\label{[Equation (179), Y20b]}
 \mathbb{E}^{P_{n}} [ \sup_{r, t \in [s_{n}, T]: \hspace{0.5mm} r \neq t} \frac{ \lVert \int_{r}^{t} J(\xi(l)) dl \rVert_{H_{x}^{-3}}^{p} }{\lvert t-r \rvert^{p-1}} ]   \overset{(M3)}{\lesssim}_{p} TC_{T,p} (1+ \lVert \xi_{n} \rVert_{L_{x}^{2}}^{2p}).  
\end{align} 
On the other hand, making use of \eqref{[Equations (11) and (12), Y20b]}, (M2) and (M3) of Definition \ref{[Definition 4.1, Y20b]} and Kolmogorov's test (e.g., \cite[Theorem 3.3]{DZ14}) gives us for any $\alpha \in (0, \frac{p-1}{2p})$  
\begin{equation}\label{[Equation (180), Y20b]}
\mathbb{E}^{P_{n}} [ \sup_{r, t \in [0, T]: \hspace{0.5mm} r \neq t} \frac{ \lVert M_{t, s_{n}}^{\xi} - M_{r, s_{n}}^{\xi} \rVert_{L_{x}^{2}}}{\lvert t-r \rvert^{\alpha}} ] \lesssim_{p} C_{t,p} (1+ \lVert \xi_{n} \rVert_{L_{x}^{2}}^{2p}). 
\end{equation} 
Making use of \eqref{[Equation (178), Y20b]}-\eqref{[Equation (180), Y20b]} leads to for all $\kappa \in (0, \frac{1}{2})$, 
\begin{equation}\label{[Equation (181), Y20b]}
\sup_{n \in \mathbb{N}} \mathbb{E}^{P_{n}} [ \sup_{r, t \in [0,T]: \hspace{0.5mm} r \neq t} \frac{ \lVert \xi(t) - \xi(r) \rVert_{H_{x}^{-3}}}{\lvert t-r \rvert^{\kappa}} ] < \infty. 
\end{equation} 
Hence, (M1), \eqref{[Equation (15), Y20b]} with $q=1$, and \eqref{[Equation (181), Y20b]} together allow us to deduce that $\{P_{n}\}$ is tight in $\mathbb{M}$ by Lemma \ref{[Lemma A.1, HZZ19]}. By Prokhorov's theorem (e.g., \cite[Theorem 2.3]{DZ14}) we deduce that $P_{n}$ converges weakly to some $P \in \mathcal{P}(\Omega_{0})$ and by Skorokhod's representation theorem (e.g., \cite[Theorem 2.4]{DZ14}) there exists a probability space $(\tilde{\Omega}, \tilde{\mathcal{F}}, \tilde{P})$ and $\mathbb{M}$-valued random variables $\{ \tilde{\xi}_{n}\}_{n\in\mathbb{N}}$ and $\tilde{\xi}$ such that 
\begin{equation}\label{[Equation (182), Y20b]}
\tilde{\xi}_{n} \text{ has the law } P_{n} \hspace{1mm} \forall \hspace{1mm} n \in \mathbb{N}, \tilde{\xi}_{n} \to \tilde{\xi} \text{ in } \mathbb{M} \hspace{1mm}  \tilde{P}\text{-a.s. and }\tilde{\xi} \text{ has the law } P. 
\end{equation} 
Making use of \eqref{[Equation (182), Y20b]} and (M1) for $P_{n}$ immediately leads to 
\begin{equation}\label{[Equation (183), Y20b]}
P( \{ \xi(t) = \xi^{\text{in}} \hspace{1mm} \forall \hspace{1mm} t \in [0,s] \}) =  \lim_{n\to\infty} \tilde{P} ( \{ \tilde{\xi}_{n}(t) = \xi_{n} \hspace{1mm} \forall \hspace{1mm} t \in [0, s_{n}] \}) =  1,
\end{equation} 
which implies (M1) for $P$. Next, it follows immediately that for every $\mathfrak{g}_{i} \in C^{\infty}(\mathbb{T}^{2})$, $\tilde{P}$-a.s. 
\begin{equation}\label{[Equation (184), Y20b]}
\langle \tilde{\xi}_{n}(t), \mathfrak{g}_{i} \rangle \to \langle \tilde{\xi}(t), \mathfrak{g}_{i} \rangle, \hspace{3mM} \int_{s_{n}}^{t} \langle J(\tilde{\xi}_{n}(r)), \mathfrak{g}_{i} \rangle dr \to \int_{s}^{t} \langle J(\tilde{\xi} (r)), \mathfrak{g}_{i} \rangle dr.
\end{equation} 
In particular, to prove the second convergence we can write 
\begin{align*}
& \mathbb{E}^{\tilde{P}}[ \int_{s_{n}}^{t} \langle J(\tilde{\xi}_{n}(r), \mathfrak{g}_{i} \rangle dr - \int_{s}^{t} \langle J(\tilde{\xi}(r)), \mathfrak{g}_{i} \rangle dr] \\
=& \mathbb{E}^{\tilde{P}}[ \int_{s_{n}}^{s} \langle - \mathbb{P} \text{div} (\tilde{\xi}_{n} \otimes \tilde{\xi}_{n}) - (-\Delta)^{m} \tilde{\xi}_{n}, \mathfrak{g}_{i} \rangle dr \\
&+ \int_{s}^{t} \langle - \mathbb{P} \text{div} (\tilde{\xi}_{n} \otimes \tilde{\xi}_{n}) + \mathbb{P} \text{div} (\tilde{\xi} \otimes \tilde{\xi}), \mathfrak{g}_{i} \rangle dr + \int_{s}^{t} \langle - (-\Delta)^{m} (\tilde{\xi}_{n} - \tilde{\xi}), \mathfrak{g}_{i} \rangle dr], 
\end{align*} 
among which we only point out that 
\begin{align*}
& \mathbb{E}^{\tilde{P}} [ \int_{s_{n}}^{s} \langle - (-\Delta)^{m} \tilde{\xi}_{n}, \mathfrak{g}_{i} \rangle dr] 
\leq \mathbb{E}^{\tilde{P}} [ \int_{s_{n}}^{s} \lVert \tilde{\xi}_{n} \rVert_{L_{x}^{2}} \lVert (-\Delta)^{m} \mathfrak{g}_{i} \rVert_{L_{x}^{2}} dr] \to 0, \\
&  \mathbb{E}^{\tilde{P}} [ \int_{s}^{t} \langle (-\Delta)^{m}( \tilde{\xi}_{n} - \tilde{\xi}), \mathfrak{g}_{i} \rangle dr] \leq \mathbb{E}^{\tilde{P}} [ \int_{s}^{t} \lVert \tilde{\xi}_{n} - \tilde{\xi} \rVert_{L_{x}^{2}} \lVert (-\Delta)^{m} \mathfrak{g}_{i} \rVert_{L_{x}^{2}} dr]  \to 0  
\end{align*} 
as $n\to\infty$ by \eqref{[Equation (182), Y20b]}. Next, we can compute for every $t \in [s, \infty)$ and $p \in (1,\infty)$, 
\begin{align}
\sup_{n \in \mathbb{N}} \mathbb{E}^{\tilde{P}} [ \lvert M_{t, s_{n}}^{\tilde{\xi}_{n}, i} \rvert^{2p}] \overset{\text{(M3)} \eqref{[Equation (178), Y20b]}}{\lesssim_{p}} 1  \hspace{2mm} \text{ and } \hspace{2mm} \lim_{n\to\infty} \mathbb{E}^{\tilde{P}} [ \lvert M_{t, s_{n}}^{\tilde{\xi}_{n}, i} - M_{t, s}^{\tilde{\xi}, i} \rvert]  \overset{\eqref{[Equation (178), Y20b]} \eqref{[Equation (184), Y20b]}}{=} 0. \label{especial}
\end{align}
Next, we let $t  > r \geq s$ and $g$ be any $\mathbb{R}$-valued, $\mathcal{B}_{r}$-measurable and continuous function on $\mathbb{M}$. Then we can compute 
\begin{align}
 \mathbb{E}^{P} [ (M_{t,s}^{\xi, i} - M_{r,s}^{\xi, i} ) g(\xi) ] \overset{\eqref{especial}}{=} \lim_{n\to\infty} \mathbb{E}^{\tilde{P}} [ ( M_{t, s_{n}}^{\tilde{\xi}_{n}, i} - M_{r, s_{n}}^{\tilde{\xi}_{n}, i}) g(\tilde{\xi}_{n})]  = 0. \label{[Equation (187), Y20b]}
\end{align} 
This implies that the map $t \mapsto M_{t,s}^{i}$ is a $(\mathcal{B}_{t})_{t\geq s}$-martingale under $P$. Next, we can deduce 
\begin{equation}\label{[Equation (188), Y20b]}
\lim_{n\to\infty} \mathbb{E}^{\tilde{P}} [ \lvert M_{t, s_{n}}^{\tilde{\xi}_{n}, i} - M_{t,s}^{\tilde{\xi}, i} \rvert^{2}] \overset{\eqref{especial}}{=} 0. 
\end{equation} 
This leads us to  
\begin{align}
\mathbb{E}^{P} [((M_{t,s}^{\xi, i})^{2} - (M_{r,s}^{\xi, i})^{2} - \int_{r}^{t} \lVert G(\xi (l))^{\ast} \mathfrak{g}_{i} \rVert_{U}^{2} dl) g(\xi)] \overset{\eqref{[Equation (182), Y20b]} \eqref{[Equation (188), Y20b]}}{=} 0. \label{[Equation (189), Y20b]} 
\end{align} 
Therefore, $(M_{t,s}^{\xi, i})^{2} - \int_{s}^{t} \lVert G(\xi(l))^{\ast} \mathfrak{g}_{i} \rVert_{U}^{2} dl$ is a $(\mathcal{B}_{t})_{t\geq s}$-martingale under $P$ which implies $\langle \langle M_{t,s}^{\xi, i} \rangle \rangle = \int_{s}^{t} \lVert G(\xi(l))^{\ast} \mathfrak{g}_{i} \rVert_{U}^{2} dl$ under $P$; it follows that $M_{t,s}^{\xi, i}$ is square-integrable. Therefore, (M2) for $P$ was shown. Finally, to prove (M3) it suffices to define 
\begin{equation}\label{[Equation (190), Y20b]}
R(t, s, \xi) \triangleq \sup_{r \in [0,t]} \lVert \xi(r) \rVert_{L_{x}^{2}}^{2q} + \int_{s}^{t} \lVert \xi(r) \rVert_{H_{x}^{\varepsilon}}^{2} dr, 
\end{equation} 
and observe that the map $\xi \mapsto R(t, s, \xi)$ is lower semicontinuous on $\mathbb{M}$ so that $\mathbb{E}^{P} [ R(t, s, \xi) ] \leq C_{t,q} (1+ \lVert \xi^{\text{in}} \rVert_{L_{x}^{2}}^{2q})$. Therefore, (M3) holds for $P$ so that $P \in \mathcal{C} (s,\xi_{0}, \{C_{t,q}\}_{q\in\mathbb{N}, t \geq s})$. 

\subsection{Continuation of the proof of Theorem \ref{[Theorem 2.3, Y20b]} assuming Proposition \ref{[Proposition 5.7, Y20b]}}

We fix any $T > 0, K > 1$ and $\kappa \in (0,1)$, and take $L$ sufficiently large that satisfies \eqref{[Equation (132), Y20b]}, as well as 
\begin{equation}\label{[Equation (134), Y20b]}
(\frac{1}{\sqrt{2}} - \frac{1}{2}) e^{2L T} > (\frac{1}{\sqrt{2}} + \frac{1}{2} ) e^{2L^{\frac{1}{2}}} \text{ and } L > [ \ln (K e^{\frac{T}{2}})]^{2}.
\end{equation}
We start from $(v_{0}, \mathring{R}_{0})$ in Proposition \ref{[Proposition 5.6, Y20b]}, and via Proposition \ref{[Proposition 5.7, Y20b]} inductively obtain $(v_{q}, \mathring{R}_{q})$ that satisfies \eqref{[Equation (122), Y20b]}, \eqref{[Equation (126), Y20b]}, and \eqref{[Equation (133), Y20b]}. Identically to \eqref{[Equation (46), Y20b]} we can show that for any $\varepsilon \in (0, \frac{\beta}{4+\beta})$ and any $t \in [0, T_{L}]$, $\sum_{q\geq 0} \lVert v_{q+1}(t) - v_{q}(t) \rVert_{H_{x}^{\varepsilon}} \lesssim m_{L} M_{0}(t)^{\frac{1}{2}}$ by \eqref{[Equation (133), Y20b]} and \eqref{[Equation (126b), Y20b]}. This allows us to deduce the limiting solution $\lim_{q\to\infty} v_{q} \triangleq v \in C([0, T_{L}]; H^{\varepsilon}(\mathbb{T}^{2}))$ that is $(\mathcal{F}_{t})_{t\geq 0}$-adapted because $\lim_{q\to \infty} \lVert \mathring{R}_{q} \rVert_{C_{T_{L}} L_{x}^{1}} = 0$ due to \eqref{[Equation (126c), Y20b]}. Because $u  = e^{B(t)} v$ where $\lvert e^{B(t)} \rvert \leq e^{L^{\frac{1}{4}}}$ for all $t \in [0, T_{L}]$ due to \eqref{[Equation (123), Y20b]}, we are able to deduce \eqref{[Equation (7), Y20b]} by choosing $\mathfrak{t} = T_{L}$ for $L$ sufficiently large. Moreover, we can show identically to \eqref{[Equation (49), Y20b]} that for all $t \in [0, T_{L}]$, $\lVert v(t) - v_{0}(t) \rVert_{L_{x}^{2}} \leq \frac{m_{L}}{2} M_{0}(t)^{\frac{1}{2}}$ by \eqref{[Equation (129), Y20b]} and \eqref{[Equation (133), Y20b]} which in turn implies
\begin{equation}\label{[Equation (135), Y20b]}  
e^{2 L^{\frac{1}{2}}} \lVert v(0) \rVert_{L_{x}^{2}}  \leq  e^{2L^{\frac{1}{2}}} ( \lVert v(0) - v_{0}(0) \rVert_{L_{x}^{2}} + \lVert v_{0} (0) \rVert_{L_{x}^{2}})\overset{\eqref{[Equation (130), Y20b]}}{\leq} e^{2L^{\frac{1}{2}}} (\frac{1}{2} + \frac{1}{\sqrt{2}} )m_{L} M_{0}(0)^{\frac{1}{2}}.   
\end{equation} 
These lead us to, on a set $\{T_{L} \geq T \}$ 
\begin{equation}\label{[Equation (136), Y20b]}
\lVert v(T) \rVert_{L_{x}^{2}} \overset{\eqref{[Equation (130), Y20b]}}{\geq} \frac{m_{L} M_{0}(T)^{\frac{1}{2}}}{\sqrt{2}} - \lVert v(T) - v_{0}(T) \rVert_{L_{x}^{2}}  \overset{ \eqref{[Equation (134), Y20b]}\eqref{[Equation (135), Y20b]}  }{\geq} e^{2L^{\frac{1}{2}}} \lVert v(0) \rVert_{L_{x}^{2}}^{2}. 
\end{equation} 
Moreover, for the fixed $T > 0$, $\kappa \in (0,1)$, one can take $L$ even larger to deduce $\textbf{P} (\{T_{L} \geq T \}) > \kappa$. We also see that $u^{\text{in}}(x) = \Upsilon(0)v(0,x) = v(0,x)$ which is deterministic because $v_{q}(0,x)$ is deterministic for all $q \in \mathbb{N}_{0}$ by Propositions \ref{[Proposition 5.6, Y20b]} and \ref{[Proposition 5.7, Y20b]}. Clearly from \eqref{[Equation (121), Y20b]}, $u = \Upsilon v$ is a $(\mathcal{F}_{t})_{t\geq 0}$-adapted solution to \eqref{[Equation (2), Y20b]}. Furthermore, it follows from \eqref{[Equation (123), Y20b]}, \eqref{[Equation (134), Y20b]},  and \eqref{[Equation (136), Y20b]}  that $\lVert u(T) \rVert_{L_{x}^{2}} \geq e^{L^{\frac{1}{2}}} \lVert u^{\text{in}} \rVert_{L_{x}^{2}} > K e^{\frac{T}{2}} \lVert u^{\text{in}} \rVert_{L_{x}^{2}}$ on the set $\{\mathfrak{t} \geq T \}$ which implies \eqref{[Equation (8), Y20b]}. 

\section*{Acknowledgments}
The author expresses deep gratitude to Prof. Jiahong Wu for very valuable discussions, as well as the Editors and the anonymous Referees for their valuable comments and suggestions that improved this manuscript significantly.

\end{document}